\documentclass[a4paper]{amsart}
\usepackage[dvips]{graphicx}
\usepackage{amsmath}
\usepackage{verbatim,enumerate}
\usepackage[usenames]{color}

\usepackage{amsthm} 
\theoremstyle{plain}
 \newtheorem{theorem}{Theorem}[section]
 \newtheorem*{theorem*}{Theorem}
 \newtheorem*{proposition*}{Proposition}
 \newtheorem{proposition}[theorem]{Proposition}
 \newtheorem{lemma}[theorem]{Lemma}
 \newtheorem{corollary}[theorem]{Corollary}
 \newtheorem{fact}[theorem]{Fact}
 \newtheorem{introtheorem}{Theorem}
 
 \newtheorem{introcorollary}[introtheorem]{Corollary}
 \newtheorem{introproposition}[introtheorem]{Proposition}
 
\theoremstyle{remark}
 \newtheorem{definition}[theorem]{Definition}
 \newtheorem{remark}[theorem]{Remark}
 \newtheorem*{remark*}{Remark}
 \newtheorem{example}[theorem]{Example}
 
 \newtheorem*{acknowledgement}{Acknowledgements}
\numberwithin{equation}{section}
\renewcommand{\theenumi}{{\rm(\arabic{enumi})}}
\renewcommand{\labelenumi}{\theenumi}

\newcommand{\Lor}{\boldsymbol{L}}
\newcommand{\Z}{\boldsymbol{Z}}
\newcommand{\Q}{\boldsymbol{Q}}
\newcommand{\R}{\boldsymbol{R}}
\newcommand{\C}{\boldsymbol{C}}
\newcommand{\E}{\mathcal{E}}

\newcommand{\SL}{\operatorname{SL}}
\newcommand{\SU}{\operatorname{SU}}
\newcommand{\U}{\operatorname{U}}

\newcommand{\PSL}{\operatorname{PSL}}

\newcommand{\Exp}{\operatorname{Exp}}
\renewcommand{\sl}{\operatorname{\mathfrak{sl}}}
\newcommand{\Herm}{\operatorname{Herm}}
\newcommand{\Cyc}{\varGamma}

\newcommand{\Horo}{\mathcal{H}}
\newcommand{\Euc}{\boldsymbol{E}}
\newcommand{\GCD}{\operatorname{\mbox{\footnotesize{GCD}}}}

\renewcommand{\Re}{\operatorname{Re}}
\renewcommand{\Im}{\operatorname{Im}}

\newcommand{\id}{\operatorname{id}}
\newcommand{\imag}{\mathrm{i}}
\newcommand{\trace}{\operatorname{trace}}
\newcommand{\ord}{\operatorname{ord}}
\newcommand{\sign}{\operatorname{sgn}}

\newcommand{\Kext}{K_{\mathrm{ext}}}
\newcommand{\secondff}{\mbox{I\!I}}

\newcommand{\inner}[2]{\left\langle{#1},{#2}\right\rangle}
\newcommand{\trans}[1]{{\vphantom{#1}}^t#1}
\newcommand{\vect}[1]{\boldsymbol{#1}}
\title[Asymptotic behavior of flat surfaces]{
  Asymptotic behavior of flat  surfaces\\
  in hyperbolic 3-space}
\thanks{
Masatoshi Kokubu, Wayne Rossman, Masaaki Umehara and Kotaro 
Yamada were supported by Grant-in-Aid for 
Scientific Research (C) No.~18540096, 
Exploratory Research No.~19654010,
Scientific Research (A) No.~19204005 
and Scientific Research (B) No.~14340024,
respectively, from the Japan Society for the Promotion of Science.
}
\author[M.~Kokubu]{Masatoshi Kokubu}
\address[Masatoshi Kokubu]{%
   Department of Mathematics,
   School of Engineering,
   Tokyo Denki University,
   2-2 Kanda-Nishiki-cho,
   Chiyoda-ku, Tokyo 101-8457,
   Japan
}
\email{kokubu@cck.dendai.ac.jp}
\author[W.~Rossman]{Wayne Rossman}
\address[Wayne Rossman]{%
   Department of Mathematics, Faculty of Science,
   Kobe University,
   Rokko, Kobe 657-8501, Japan%
}
\email{wayne@math.kobe-u.ac.jp}
\author[M.~Umehara]{Masaaki Umehara}
\address[Masaaki Umehara]{%
   Department of Mathematics, Graduate School of Science,
   Osaka University,
   Toyonaka, Osaka 560-0043,
   Japan
}
\email{umehara@math.sci.osaka-u.ac.jp}
\author[K.~Yamada]{Kotaro Yamada}
\address[Kotaro Yamada]{%
   Faculty of Mathematics,
   Kyushu University, 
   Higashi-ku, Fukuoka 812-8581, Japan%
}
\email{kotaro@math.kyushu-u.ac.jp}
\dedicatory{%
  Dedicated to Professor Seiki Nishikawa on 
  the occasion of his sixtieth birthday
}
\subjclass[2000]{Primary 53C42; Secondary 53A35}
\keywords{Flat surface, flat front, end, asymptotic behavior, 
hyperbolic 3-space}
\begin{document}
\begin{abstract}
 In this paper, we investigate the asymptotic behavior of regular
 ends of flat surfaces in the hyperbolic $3$-space $H^3$.
 G\'alvez, Mart\'\i{}nez and Mil\'an showed that when the singular set 
 does not accumulate at an end, 
 the end is asymptotic to a
 rotationally symmetric flat surface. 
 As a refinement of their result,
 we show that the asymptotic order (called {\it pitch\/} $p$) 
 of the end determines the limiting shape,
 even when the singular set does accumulate at the end.
 If the singular set is bounded away from the end, we have $-1<p\le 0$. 
 If the singular set accumulates at the end, 
 the pitch $p$ is a positive rational number
 not equal to $1$.  
 Choosing appropriate positive integers $n$ and $m$ so that $p=n/m$, 
 suitable slices of the end by horospheres are asymptotic to 
 $d$-coverings ($d$-times wrapped coverings) of epicycloids
 or $d$-coverings of hypocycloids with $2n_0$ cusps and whose normal
 directions have  winding number $m_0$, where $n=n_0d$, $m=m_0d$ 
 ($n_0$, $m_0$ are integers or half-integers) and
 $d$ is the greatest common divisor of $m-n$ and $m+n$. 
 Furthermore, it is known that the caustics of flat surfaces are also 
 flat.  
 So, as an application, we give a useful explicit formula for the pitch
 of ends of caustics of complete flat fronts.
\end{abstract}

\maketitle

\section*{Introduction}
\begingroup
\renewcommand{\theequation}{\arabic{equation}}

Let $f\colon{}D^*\to H^3$ be an immersion of the unit punctured disc
$D^*:=\{z\in \C\,;\, 0<|z|<1\}$ into the hyperbolic $3$-space $H^3$.
Then $f$ is called {\it flat\/} if the Gaussian curvature vanishes
everywhere, 
and, assuming this is the case,
we call $f$ an {\it end\/} of a flat surface.
Since any flat surface is orientable \cite{KRUY}, this is the general
setup for ``ends'' of flat surfaces.
Moreover, $f$ is called a {\it complete\/} end if $f$ is complete at the
origin $z=0$ with respect to the Riemannian metric induced by $f$.
Then the two hyperbolic Gauss maps 
$G,G_* \colon D^*\to \partial H^3=\C\cup\{\infty\}$ are defined on $D^*$
\cite{GMM}.  
If both $G(z)$ and $G_*(z)$ can be extended smoothly across $z=0$,
$f$ is called a {\it regular\/} end, and otherwise $f$ is called an 
{\it irregular\/} end.

Let $\nu$ be the unit normal vector field to $f$, and set 
\[
    f_t\colon{}D^*\ni z\longmapsto f_t(z) = \Exp_{f(z)}
        \bigl(t\nu(z)\bigr)
                \in H^3
\]
for each real number $t$,
where ``$\Exp$'' denotes the exponential map of the Riemannian manifold
$H^3$
(see \eqref{eq:parallel} in the next section for a more explicit
 description of $f_t$).
This surface $f_t$ is called a {\em parallel surface\/} of $f$.
A parallel surface $f_t$ may have singular points, 
but it will be considered here as a ({\em wave}) {\em front\/}, 
i.e., a surface which admits certain kinds of singularities 
(see \cite{GMM}, \cite{KUY2}).  
Moreover, any parallel surface $f_t$, away from singular points,
 is flat if $f$ is flat. 
It is often reasonable to begin arguments 
under the assumption that the flat surface is a front. 
When we wish to emphasize that assumption, 
we speak of it as a {\it flat front\/} instead of a flat surface.

 From now on, we assume $f\colon D^* \to H^3$ is a flat front.
Even if $f$ is a complete end, 
$f_t$ might not be complete at the origin in general, that is, 
it can happen that the singular points of $f_t$ accumulate at the origin.   
However, each $f_t$, including $f=f_0$, is weakly complete and of finite type
in the sense of \cite{KRUY} 
(see Definition~\ref{def:complete}).
Moreover, for each non-umbilic point $z\in D^*$, 
there is a unique $t(z) \in \R$ so that
$f_{t(z)}$ is not an immersion at $z$, i.e., $z$ is a singular 
point of $f_{t(z)}$. 
Then the singular locus (or equivalently, the set of focal points) 
is the image of the map
\[
  C_f\colon{} D^*\setminus\{\text{umbilic points}\}\ni
         z\longmapsto f_{t(z)}\in H^3,
\]
which is called the {\it caustic\/} (or {\it focal surface}) of $f$.
Note that caustics can be defined not only for ends but globally for
non-totally umbilic flat fronts.
Roitman \cite{R} proved that $C_f$ is flat 
(in fact, it is locally a flat front, see \cite{KRSUY} and \cite{KRUY}),
and gave a holomorphic representation formula for such caustics. 

A caustic can have more symmetry than the original surface:
Figure~\ref{fig:caustic} shows a symmetric four-noid and its
caustic.
The caustic of the four-noid as in 
Figure \ref{fig:caustic} (right) has octahedral symmetry, 
though the original surface has only dihedral symmetry.
This shows that an end of a caustic coming from an umbilic point 
of the original surface can be
congruent to another of the caustic's ends coming from an end of the
original surface.

\begin{figure}
\begin{center}
 \begin{tabular}{cc}
  \includegraphics[width=0.45\textwidth]{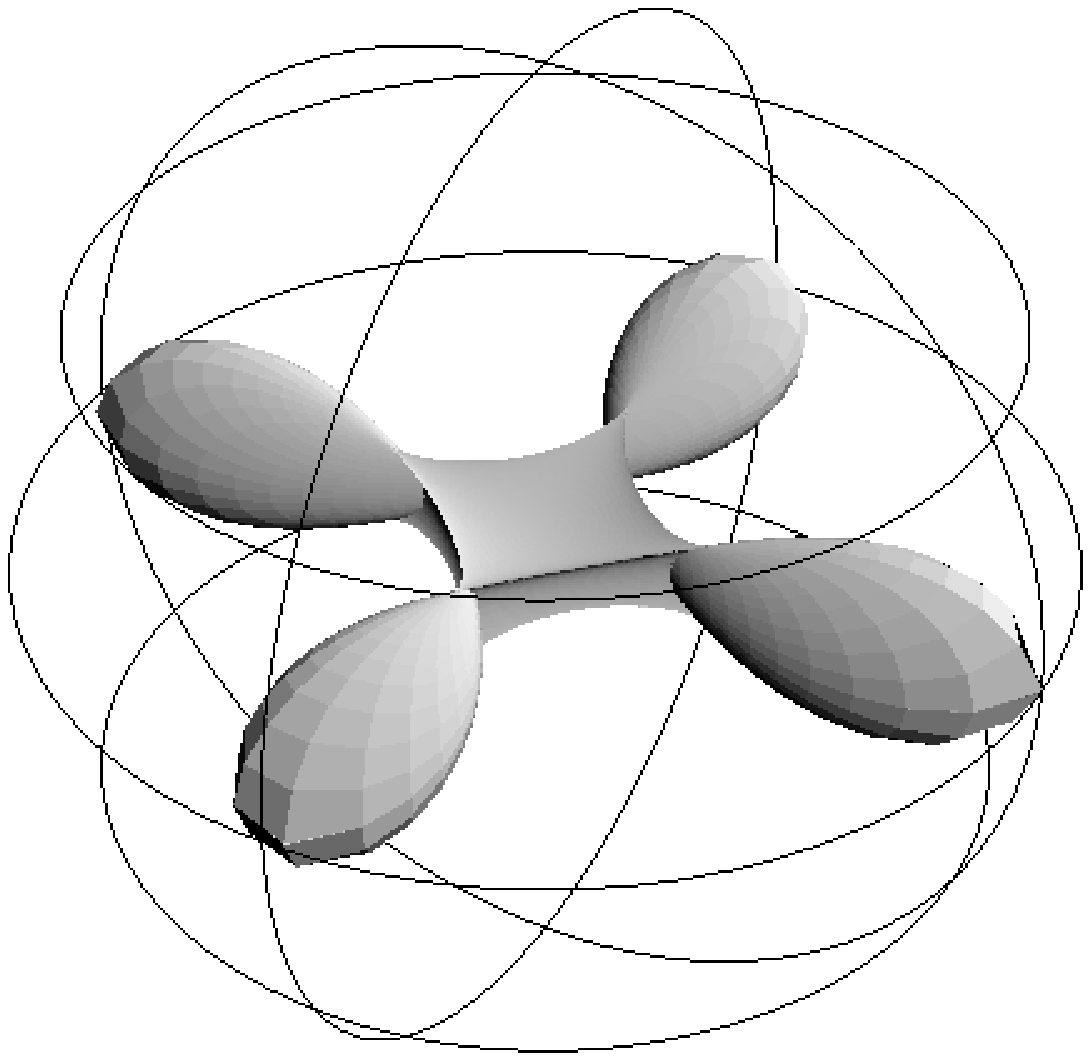} &
  \includegraphics[width=0.45\textwidth]{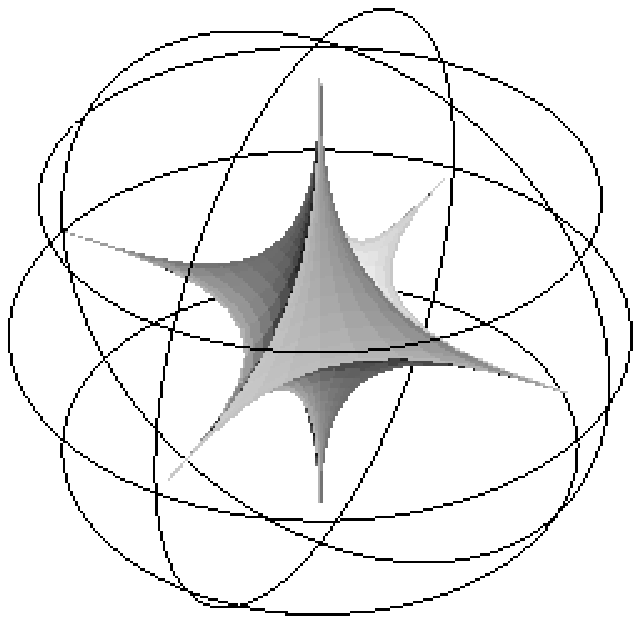} \\
  a four-noid & its caustic 
\end{tabular}
\end{center}
\caption{A flat four-noid and its caustic}%
\label{fig:caustic}
\end{figure}

As seen in Figure~\ref{fig:caustic}, the ends of caustics are typically
highly acute, 
and the singular sets accumulate at the ends, even though
non-cylindrical complete ends are tangent to the ideal boundary
(see also Figure~\ref{fig:nis1andmis4} in Section~\ref{sec:example}).
Prompted by Roitman's work, the authors numerically examined such
incomplete ends on several caustics and were surprised 
at their acuteness and 
at the additional symmetry as mentioned above, 
and so wished to analyze their behavior precisely. 
This is the central motivation of this paper, 
which is a sequel of the previous paper \cite{KRUY}.

As an analogue of a result in \cite{UY} for constant mean curvature one
surfaces (CMC-1 surfaces) in $H^3$,
\cite{GMM} showed that a complete regular end is asymptotic to the
$m$-fold cover of one of the rotationally symmetric flat surfaces, 
and that $m=1$ implies proper embeddedness of the end.  
In order to state both this result and other new results,
we fix the setting as follows: Let $f\colon D^* \to H^3$ be a weakly
complete end of finite type, 
which we abbreviate as ``{\em WCF-end\/}'' 
(Weakly complete and finite type were defined in \cite{KRUY}, and are
also defined in Definition~\ref{def:complete} here).

Moreover, we assume the WCF-end $f$ is regular. 
(The regularity for WCF-ends is defined in the same way as for complete
ends.)
Denoting by $\pi \colon H^3 \to \R^3_+$ the projection of $H^3$ 
to the Poincar\'e upper half-space model 
$\R^3_+:=\{ (\zeta,h)\in \C\times \R\,;\, h>0\}$, 
we discuss the asymptotic behavior of regular WCF-ends in terms of 
$\pi\circ f$.

Note that WCF-ends are generalizations of complete ends. 
Moreover, 
{\it all ends of caustics of complete regular-ended flat fronts are 
regular WCF-ends 
{\rm (}see \cite[Theorems 7.4 and 7.6]{KRUY}{\rm )}}.

G\'alvez, Mart\'\i{}nez and Mil\'an \cite{GMM} proved that
each complete regular end is asymptotic to a finite 
covering of a rotationally symmetric end.
The following proposition is essentially the same as
their result, but now stated in terms of a geometric quantity 
we call the {\it pitch} 
(see Proposition \ref{fact:complete-asymptotic} 
and Theorem~\ref{thm:incomplete}), 
and also in terms of the ratio of the Gauss maps 
(see \cite{KRUY} or \eqref{eq:gauss-ratio} for a definition):
\begin{introproposition}\label{fact:complete-asymptotic}
 Suppose that the flat surface  $f$ is a complete regular end.
 Then for a sufficiently small $\varepsilon>0$, 
 the image 
$\pi\circ f(D^*_\varepsilon)$
$(D^*_\varepsilon :=  \{z\in\C\,|\,0<|z|<\varepsilon\})$
 is congruent to a portion of the image of 
 $[0,2\pi) \times (0,h_0) \ni(t,h)\mapsto (\varphi_h(t),h)\in \R_+^3$ 
 with
 \begin{equation}\label{eq:asymp-complete}
      \varphi_h(t)=c e^{\imag mt}h^{1+p}+o(h^{1+p}),
 \end{equation}
 for a non-zero constant $c$, a nonpositive constant $p$, 
 and a positive integer $m$.
 Here, 
 $m$ is the multiplicity of the end  as in  \eqref{eq:multiplicity},
 $o(h^{1+p})$ denotes terms of order higher than $h^{1+p}$
 as $h\to 0$, 
 and the exponent $p$ {\rm (}called the {\em pitch\/} of $f${\rm )}
 is related to the ratio $\alpha$ of the Gauss maps  
  by
 \[
     p=-\frac{1+\alpha}2 \in (-1,0].
 \]
 In particular, 
 the pitch of each parallel surface $f_t$ is  also equal to $p$ whenever
 $f_t$ is complete. 
\end{introproposition}

Later, we shall give a refinement of this assertion, 
that is,  we shall compute the second term of the expansion 
in \eqref{eq:asymp-complete} 
(Theorems~\ref{thm:non-cyl} and \ref{thm:cyl}).

A description of the asymptotic behavior of CMC-$1$ surfaces in $H^3$
was first given in \cite{UY}, 
and refinements were given by Sa Earp and Toubiana \cite{ET}
and Daniel \cite{D}.
In particular, Daniel's refinement gives relationships
between the flux and asymptotic behavior of complete regular ends of 
CMC-1 surfaces.  
To prove Proposition~\ref{fact:complete-asymptotic} 
and its refinements (Theorems~\ref{thm:non-cyl} and \ref{thm:cyl}), 
we define an analogue of the flux matrix as in \cite{RUY}.  
In this sense, 
Theorems~\ref{thm:non-cyl}, \ref{thm:cyl} and
Theorem~\ref{thm:incomplete} below are an analogue of Daniel's line of
investigation. 

On the other hand, the asymptotic behavior of an incomplete end, 
as in Theorem~\ref{thm:incomplete} below, 
has clearly not been analyzed in the case
of CMC-1 surfaces, for the obvious reason that those surfaces do not have
singularities.  
Analysis of the incomplete end case leads to a mysterious connection
between flat surfaces and cycloid curves:
%
\begin{figure}
\footnotesize
\begin{center}
\begin{tabular}{c@{\hspace{2em}}c@{\hspace{2em}}c}
 \includegraphics[width=0.2\textwidth]{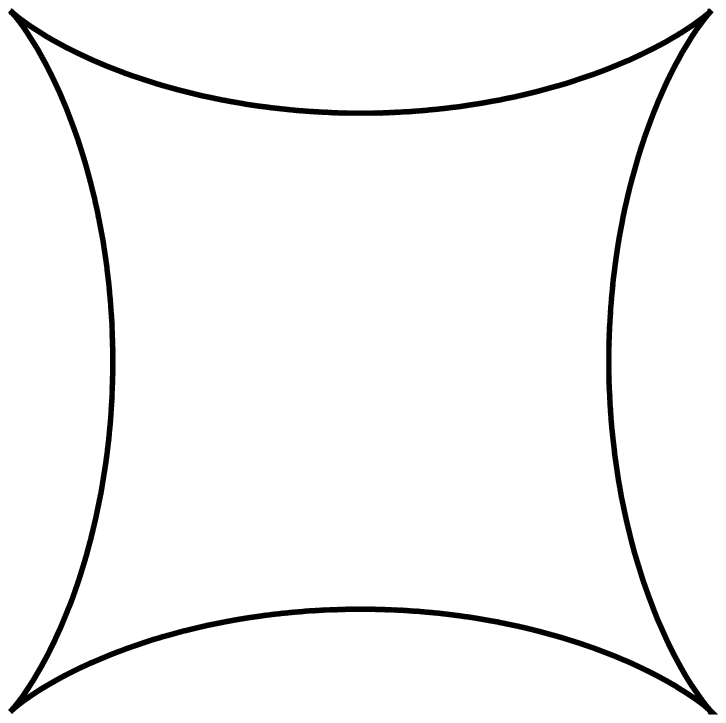} &
 \includegraphics[width=0.2\textwidth]{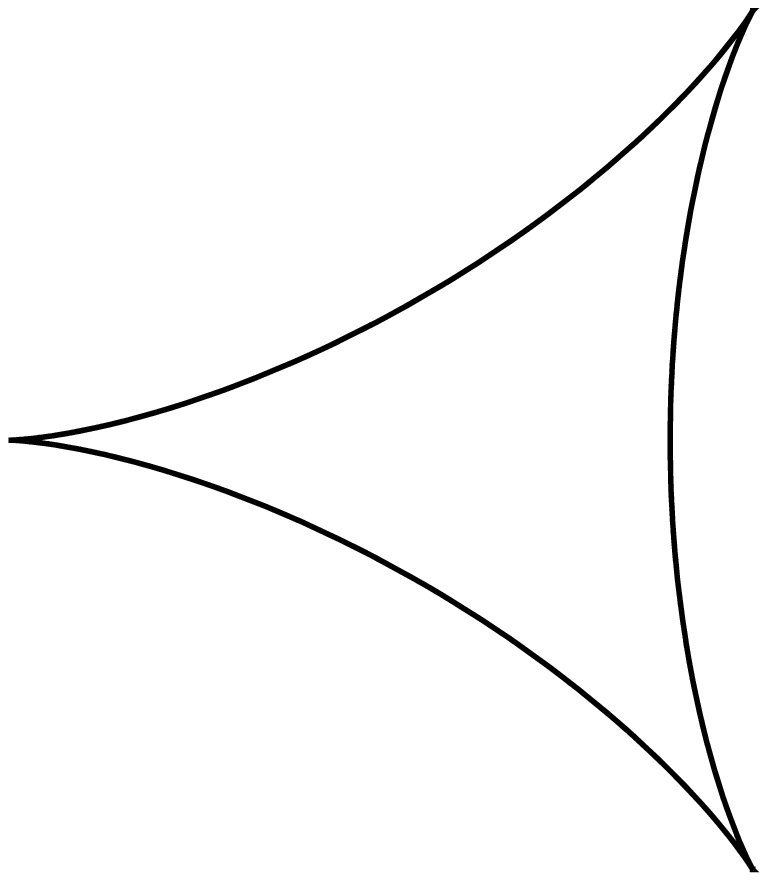} &
 \includegraphics[width=0.2\textwidth]{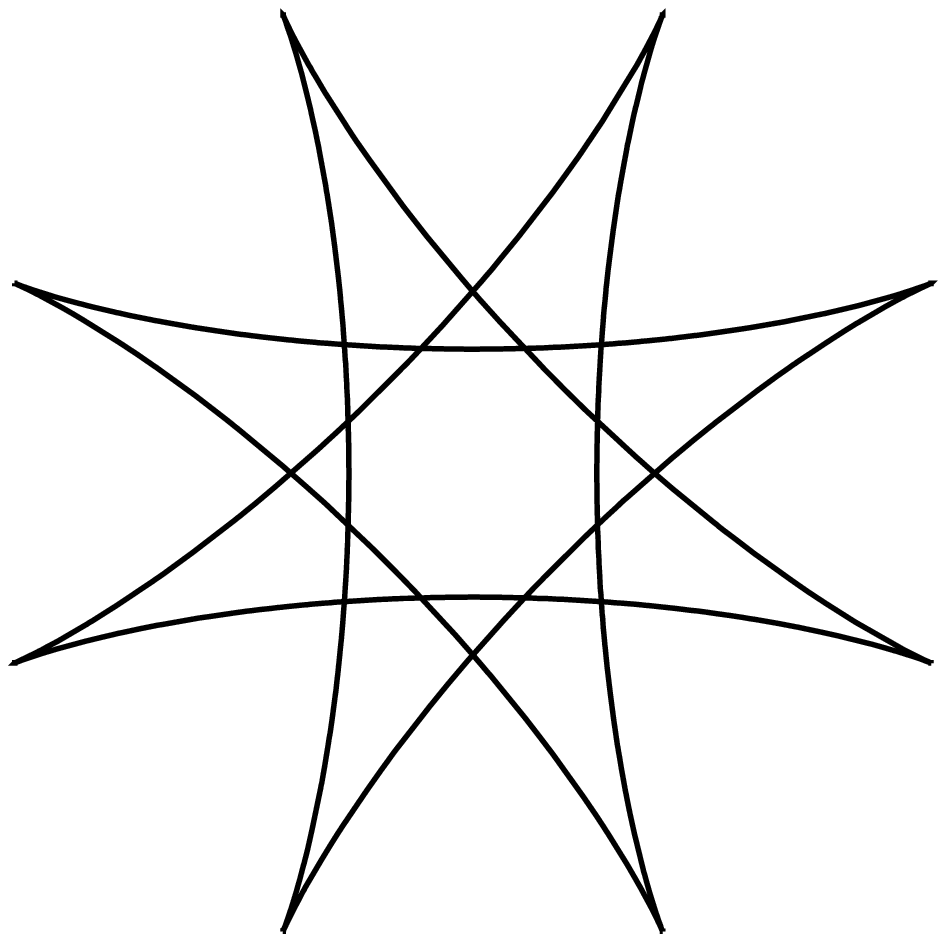} \\
 $\begin{array}{r@{}l}
     (m,n)&=(1,2)\\
    (m_0,n_0)&=(1,2)\\
      d&=1
  \end{array}$&
 $\begin{array}{r@{}l}
     (m,n)&=(1,3)\\
    (m_0,n_0)&=(\frac{1}{2},\frac{3}{2})\\
      d&=2
  \end{array}$&
 $\begin{array}{r@{}l}
     (m,n)&=(1,4)\\
    (m_0,n_0)&=(1,4)\\
      d&=1
  \end{array}$\\[12pt]
 \multicolumn{3}{c}{\small{Hypocycloids}}\\[8pt]
 \includegraphics[width=0.2\textwidth]{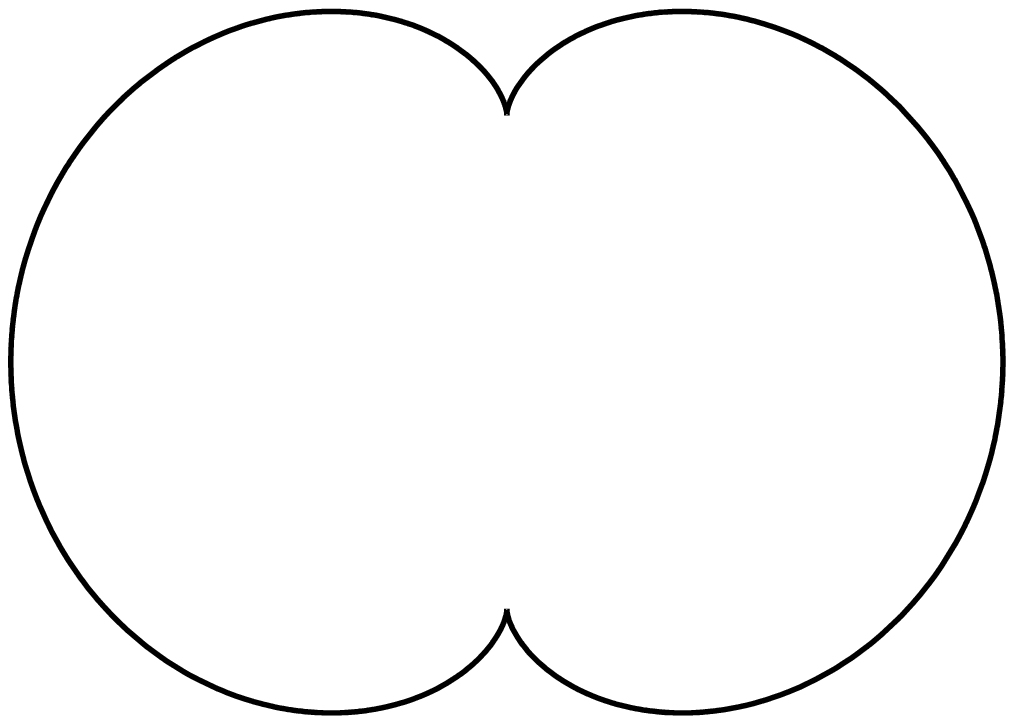} &
 \includegraphics[width=0.2\textwidth]{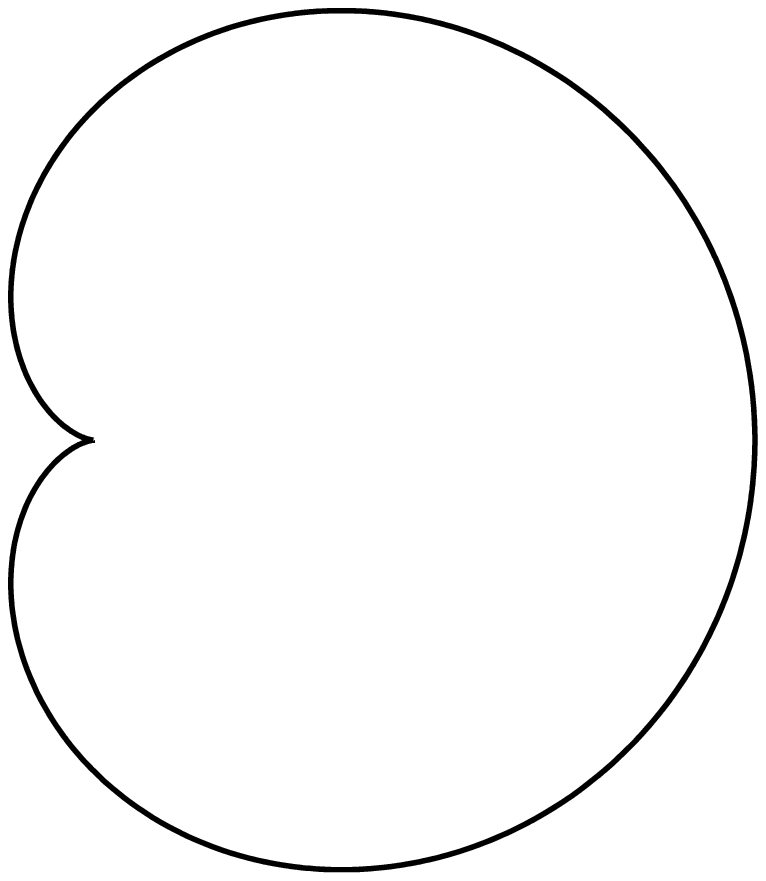} &
 \includegraphics[width=0.2\textwidth]{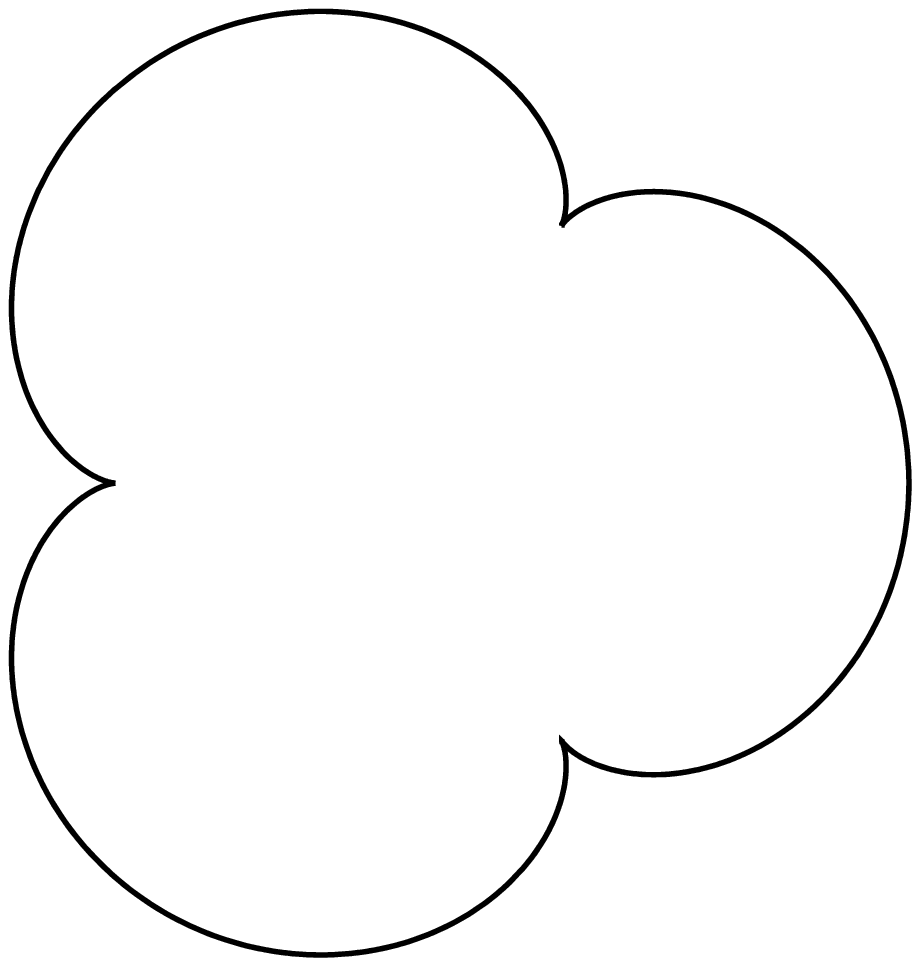} \\
 $\begin{array}{r@{}l}
     (m,n)&=(2,1)\\
    (m_0,n_0)&=(2,1)\\
      d&=1
  \end{array}$&
 $\begin{array}{r@{}l}
     (m,n)&=(3,1)\\
    (m_0,n_0)&=(\frac{3}{2},\frac{1}{2})\\
      d&=2
  \end{array}$&
 $\begin{array}{r@{}l}
     (m,n)&=(5,3)\\
    (m_0,n_0)&=(\frac{5}{2},\frac{3}{2})\\
      d&=2
  \end{array}$\\[12pt]
 \multicolumn{3}{c}{\small{Epicycloids}}
\end{tabular}
\end{center}
\caption{Cycloids}\label{fig:cyc}
\end{figure}
A cycloid is the image of the map 
 \[
      \Cyc_{m,n}(t):=
      \frac{1}{m}
      \left[(m+n)e^{\imag(m-n)t}
              +(m-n)e^{\imag(m+n)t}\right]. 
 \]
Let $d$ be the greatest common divisor of $m+n$ and $m-n$ and set 
$m=m_0d$, $n=n_0 d$ ($m_0,n_0\in \frac12\Z$).
Then the image of $\Cyc_{m,n}$ is determined by the pair 
$(m_0,n_0)$ satisfying $m_0,n_0 \in \frac12\Z$,
$m_0\pm n_0\in \Z$ and
$\GCD[m_0+n_0,m_0-n_0]=1$.
Since such a pair $(m_0,n_0)$ corresponds bijectively 
to a rational number $n_0/m_0\in \Q_+\setminus\{1\}$, 
we denote the image of $\Cyc_{m,n}$ by $c_{n/m}(=c_{n_0/m_0})$,
and is called an {\it epicycloid\/} if $n/m<1$ and
a {\it hypocycloid\/} if $n/m>1$.
It is well-known that $c_{n/m}$ is created by the trace of a point on a
circle of signed radius 
$1-p$ 
rolled along another circle of radius
$2p=2n/m=2n_0/m_0$ without slippage, 
and has no self-intersections if and only if $|n_0-m_0|=1$.
Cycloids admit $(3/2)$-cusps, at which their unit normal vectors are
well-defined smooth vector fields.
The number $m_0$ takes a value in $\frac12\Z$, and is called 
the {\it winding number\/} of the cycloid $c_{n_0/m_0}$, 
which is the number of times that the unit normal vector to the cycloid
winds around $S^1$ as the cycloid is traversed once.  
The map $\Cyc_{m,n}$ represents a $d$-covering of the cycloid $c_{n/m}$ 
(see Figure \ref{fig:cyc}).

In order to describe our main results, 
we give an integer $2n$,  
which will correspond both to the number of singularities of a cycloid
and to the number of connected components of cuspidal edges appearing in
an incomplete WCF-end.
The canonical forms $\theta$ and $\omega$ associated to a flat front
will be defined in Section~\ref{sec:prelim}.
Note that incomplete regular WCF-ends are all cylindrical.
If a regular WCF-end 
$f:D^*\to H^3$ is cylindrical, then 
$\rho(z):=\theta/\omega$ (given in \eqref{eq:rho-def}) is a
nonvanishing  
holomorphic function near $z=0$.
Moreover if $f$ is incomplete, then  $|\rho(0)|=1$ 
(see Lemma \ref{lem:incomplete}), and 
whenever $\rho$ is non-constant, 
the ramification order $n$ of $\rho$ is computed as 
\begin{equation}
 \label{eq:formula-drhorho}
  n = 1+ \ord_0\frac{d\rho}{\rho}, \qquad
        \frac{d\rho}{\rho}=
           \Bigl(\frac{\hat\theta'}{\hat\theta}-
                 \frac{\hat\omega'}{\hat\omega}
           \Bigr)\,dz
	   =  
           \left(\frac{G''_*}{G'_*}-\frac{G''}{G'}
	        +2\frac{G'+G'_*}{G-G_*}\right)\,dz,
\end{equation}
where $\omega=\hat\omega\,dz$, $\theta=\hat\theta\,dz$,
$'=d/dz$
(see Section~\ref{sec:prelim} and \cite[(3-15)]{KRSUY}),
and  $\ord_0\varpi=k$ for a given meromorphic $1$-form $\varpi$ if it is
written as  $\varpi=z^k \varphi(z)\,dz$ ($\varphi(0)\neq 0$).
\begin{introtheorem}\label{thm:incomplete}
 Suppose that the flat front $f$ is an incomplete
 but weakly complete regular end of finite type 
 {\rm (}i.e., an incomplete regular WCF-end{\rm)},
 which is not contained in a geodesic line in $H^3$.
 Then for a sufficiently small $\varepsilon>0$, 
 the image $\pi \circ f(D^*_\varepsilon)$
 is congruent to a portion of the image of 
 $[0,2\pi) \times (0,h_0) \ni (t,h)\mapsto (\varphi_h(t),h)\in \R_+^3$,
 with
 \begin{equation}\label{eq:asymp-incomplete}
    \varphi_h(t)=
             h^{1+p}\Cyc_{m,n}(t)+o(h^{1+p}),
	     \quad p=\frac{n}{m}\in (0,1)\cup (1,\infty) \; .
 \end{equation}
 Here, $m$ is the ramification order of the hyperbolic Gauss map
 $G(z)$  at $z=0$ 
 {\rm (}which coincides with that of the other hyperbolic Gauss map
 $G_*(z)${\rm )},  and $n$ $(\neq m)$ 
 is the ramification order of the function
 $\rho:=\theta/\omega$ at $z=0$.
 The set of singular points of $f$ consists 
 only of cuspidal edges, corresponding to the cusps
 on the slices $\varphi_h(t)$ created by cutting with 
 $h=\mbox{constant}$.
 Here, the exponent $p$ is again called the {\em pitch\/} of $f$.
 In particular, $f$ has no self-intersections outside of a sufficiently
 large geodesic ball in $H^3$
 if and only if $d=1$ and $|m_0-n_0|=1$, where $m_0,n_0$ and $d$
 are the numbers determined by the map $\Cyc_{m,n}$.
\end{introtheorem}

It should be remarked that a surface given by
\[
   f_p(\vartheta,h):=(r(\vartheta)e^{\imag \vartheta}h^{1+p},h)\in \R^3_+
   \qquad (p>0,~ p\ne 1)
\]
is asymptotically flat with respect to the Poincar\'e
metric of constant curvature $-1$ as $h\to +0$ if and only if
\[
   (p^2-1)\frac{du}{d \vartheta}= p^2+ (p^2 + 1)u^2+u^4,
\]
where $u= d \log r/d \vartheta$.
The general solution of this ordinary differential equation 
for $p=n/m$ characterizes the cycloids $c_{n/m}$, 
which removes the mystery why cycloids appear in the asymptotic
behavior of ends of flat surfaces (see Appendix~\ref{B}).

Theorem \ref{thm:incomplete} implies that the pitch of any incomplete
WCF-end is a positive rational number not equal to $1$, 
so let us use the following terminology: 
an incomplete WCF-end is of {\it hypocycloid-type\/} 
if the pitch is greater than $1$, or of {\it epicycloid-type\/}
if the pitch is less than $1$, respectively.

Theorem \ref{thm:incomplete} is particularly useful for studying
caustics of complete flat fronts in $H^3$, as those caustics have
incomplete ends in general.
The pitch of each end of the caustic can be computed using 
just the pair of hyperbolic Gauss maps for the original
flat front as described in Theorem~\ref{thm:Sec5} in the final section, 
and it then tells us the asymptotic behavior of the end of the caustic.

Even if we do not know a priori whether the end is complete, 
any regular WCF-end is asymptotic to either 
\eqref{eq:asymp-complete} in 
Proposition~\ref{fact:complete-asymptotic} or 
\eqref{eq:asymp-incomplete} in Theorem~\ref{thm:incomplete}. 
Thus the {\it pitch\/} $p$ is a single entity that encompasses 
both cases \eqref{eq:asymp-complete} and \eqref{eq:asymp-incomplete}.  
As a corollary of Proposition~\ref{fact:complete-asymptotic},
Theorem~\ref{thm:incomplete} and \cite[Propositions 3.1 and 7.3]{KRUY},
we have: 
\begin{figure}
\footnotesize
\begin{center}
 \begin{tabular}{c@{}c@{}c@{}c}
  \includegraphics[width=0.25\textwidth]{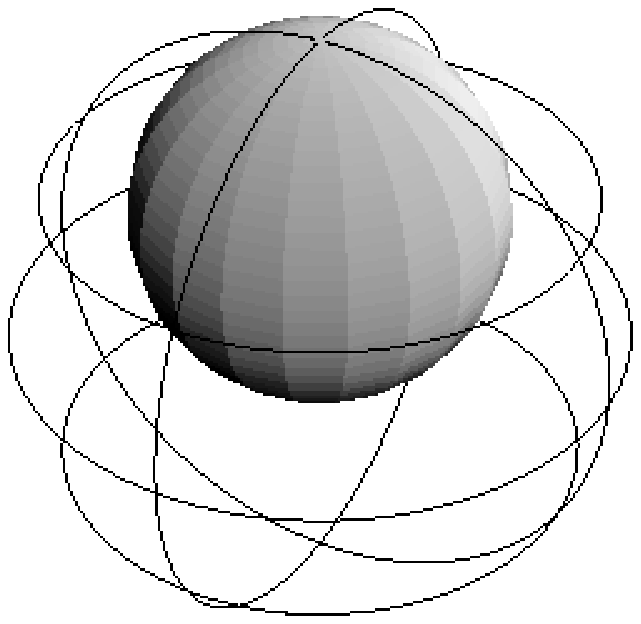}&
  \includegraphics[width=0.25\textwidth]{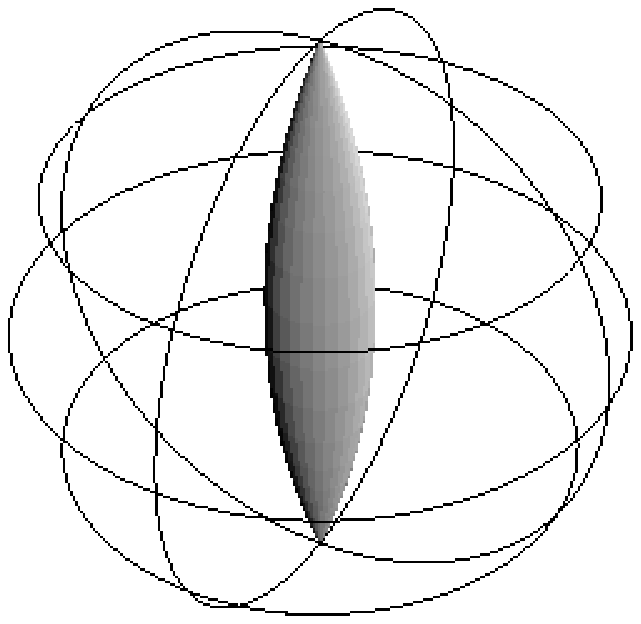}&
  \includegraphics[width=0.25\textwidth]{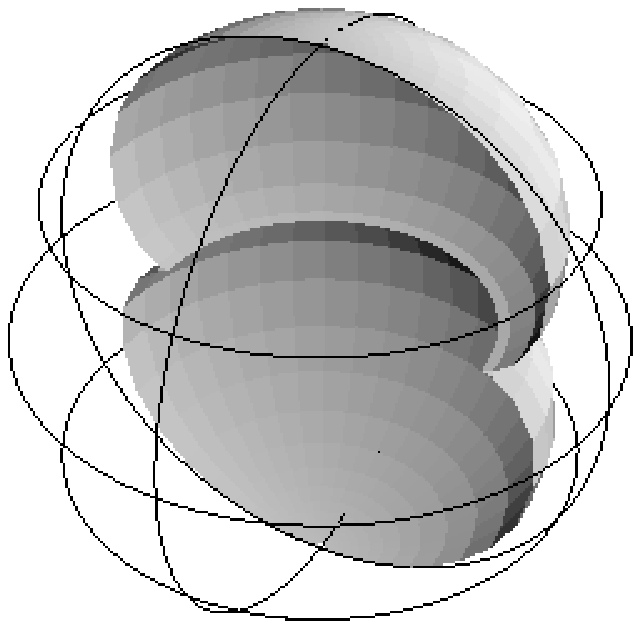}&
  \includegraphics[width=0.25\textwidth]{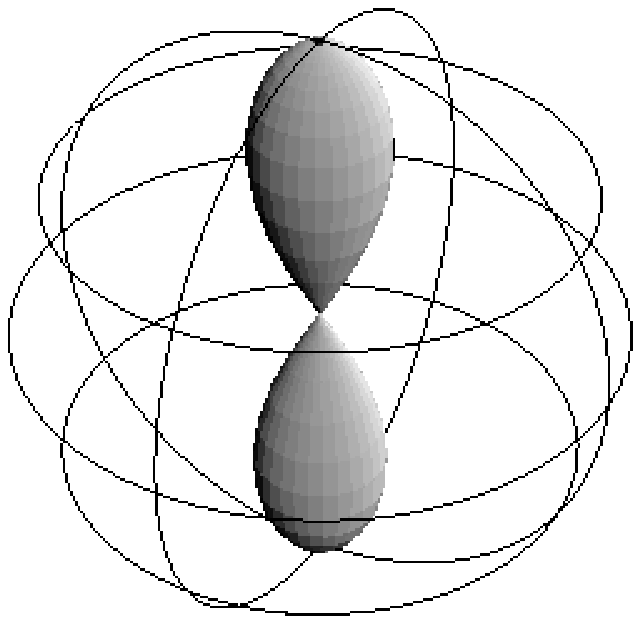}\\
  horosphere &
  cylinder &
  snowman &
  hourglass
 \end{tabular}
\end{center}
\caption{Flat fronts of revolution, shown in the Poincar\'e ball 
         model for $H^3$}\label{fig:revolution}
\end{figure}
\begin{introcorollary}\label{cor:ends}
 The pitch $p$ of a complete regular end takes its value in $(-1,0]$, 
 and the pitch $p$ of an incomplete regular WCF-end
 takes its value in 
 $\Q_{+}\setminus\{1\}$, where $\Q_+$ is the set of 
 positive rational numbers.
 Moreover, a regular WCF-end is 
 \begin{itemize}
 \item a snowman-type end if and only if $-1<p<-1/2$, 
 \item a horospherical end if and only if $p=-1/2$,
 \item an hourglass-type end if and only if $-1/2<p<0$,
 \item a complete cylindrical end if and only if $p=0$,
 \item an end of epicycloid-type with 
       $2n$ cusps and winding number $m$ if and only if $p=n/m\in (0,1)$, 
 \item 	an end of hypocycloid-type with 
	$2n$ cusps and winding number $m$
       if and only if $p=n/m\in (1,\infty)$.
 \end{itemize}
\end{introcorollary}
Snowman-type ends, horospherical ends, 
hourglass-type ends and cylindrical ends were defined in \cite{KRUY} 
using properties of the canonical $1$-forms and the Hopf differential
(see Definition~\ref{def:types}), and asymptotic behavior was not 
established there.  
But as a consequence of  Proposition~\ref{fact:complete-asymptotic} and
Corollary~\ref{cor:ends},
snowman-type ends, horospherical ends, hourglass-type ends and complete
cylindrical ends are now known to be 
asymptotic to the finite covering of a snowman, 
the horosphere, an hourglass and a cylinder, respectively 
{\rm (}see Figure~{\rm \ref{fig:revolution})}.

We will see that any WCF-end of epicycloid-type or hypocycloid-type is 
necessarily a cylindrical end, but it is not complete
(see Lemma~\ref{lem:incomplete}). 

Corollary~\ref{cor:ends} means that the shape of any end tells us what
its pitch $p$ is, and vice versa.  
For example, Figure~\ref{fig:caustic} (right) indicates the
caustic of a flat front of genus $0$ with $4$ ends 
(Figure~\ref{fig:caustic} left). 
The central end (converging to the north pole in $\partial H^3$) shown
there has four cuspidal edges.  
Since the winding number of slices of the end in this case is $1$, 
we can conclude that the pitch of the end is $p=2$ 
(cf.\ Example~\ref{exa:n-noid} in Section~\ref{sec:example}).
\begin{acknowledgement}
 The third and fourth authors would like to thank 
 Jose Antonio G\'alvez and 
 Antonio Mart\'\i{}nez for fruitful discussions 
 during their stay at Granada.
 The authors also thank  
 the referee for comments that significantly improved the results here.
\end{acknowledgement}
\endgroup
\section{Fundamental properties of regular ends of flat fronts}
\label{sec:prelim}
In this section, we shall describe fundamental properties of flat
fronts in the hyperbolic $3$-space.
See \cite{KUY1}, \cite{KUY2}, \cite{KRUY} for precise arguments 
and proofs.

\subsection*{The hyperbolic space}
The hyperbolic $3$-space $H^3$ of constant sectional curvature 
$-1$ is realized as the upper half component of the hyperboloid
of the Minkowski $4$-space $\Lor^4$ with 
inner product $\inner{~}{~}$ of signature $(-,+,+,+)$:
\begin{equation}\label{eq:hyp-lor}
   H^3 =\{
         x= (x_0,x_1,x_2,x_3)\in\Lor^4\,;\,
         \inner{x}{x}=-1,
         x_0>0\}.
\end{equation}
Identifying $\Lor^4$ with the set $\Herm(2)$ of $2\times 2$ Hermitian
matrices as
\begin{equation}\label{eq:lor-herm}
    \Lor^4\ni (x_0,x_1,x_2,x_3) \longleftrightarrow
   \begin{pmatrix}
     x_0 + x_ 3 & x_1 + \imag x_2 \\
     x_1 - \imag x_2  & x_0 - x_3
   \end{pmatrix}
   \qquad \left(\imag=\sqrt{-1}\right), 
\end{equation}
we can write
\begin{align}
  H^3 &= 
        \{X\in\Herm(2)\,;\,\det X=1,\trace X>0\} 
 \label{eq:hyp-mat}
 \\
      &= \{ uu^{*}\,|\,u\in\SL(2,\C)\} = \SL(2,\C)/\SU(2)
 \qquad (u^*=\trans{\bar{u}}).
 \nonumber
\end{align}
The complex Lie group $\SL(2,\C)$ acts isometrically on $H^3$ 
by 
\begin{equation}\label{eq:isometry}
    \iota_{u}\colon{} H^3 \ni X 
              \longmapsto u X u^* \in H^3
    \qquad (u\in\SL(2,\C)).
\end{equation}
In fact, the identity component of the isometry group of $H^3$ 
is identified with $\PSL(2,\C)=\SL(2,\C)/\{\pm 1\}$.

We consider the projection 
\begin{equation}\label{eq:proj0}
   \pi:H^3 \ni (x_0,x_1,x_2,x_3)\longmapsto \frac{1}{x_0-x_3}
   (x_1+ \imag x_2,1)\in \R^3_+ , 
\end{equation}
where $\R^3_+:=\{ (\zeta,h)\in \C\times \R\,;\, h>0\}$.
The map $\pi$ is an isometry from $H^3$ to the Poincar\'e upper 
half-space model
\begin{equation}\label{eq:upper-metric}
 \left(\R^3_+,\frac{|d\zeta|^2+dh^2}{h^2}\right).
\end{equation}
Under the parametrization as in \eqref{eq:hyp-mat}, 
we can write
\begin{equation}\label{eq:proj-mat}
   \pi(uu^*) = 
      \frac{
      \left(
          u_{11} \overline{u_{21}}  + 
          u_{12} \overline{u_{22}} , 1
      \right)}{
      u_{21}\overline{u_{21}}+u_{22}\overline{u_{22}}},
      \qquad
      \text{where}\quad
      u=\begin{pmatrix}
	  u_{11} & u_{12} \\
	  u_{21} & u_{22}
	 \end{pmatrix}\in\SL(2,\C).
\end{equation}
The ideal boundary $\partial H^3$ is identified with $\C\cup\{\infty\}$
as in \eqref{eq:asymptotic} in Appendix~\ref{app:boundary}.

\subsection*{Flat fronts}
A smooth map $f\colon{}M^2\to H^3$ from a $2$-manifold into the 
hyperbolic $3$-space is called a {\em front\/} if 
there exists a Legendrian immersion $L_f\colon{}M^2\to T_1^*H^3$
into the unit cotangent bundle of $H^3$ whose projection is $f$.
Identifying $T_1^*H^3$ with the unit tangent bundle $T_1H^3$, 
$L_f$ corresponds to the {\em unit normal vector field\/}
$\nu$ of $f$, that is, 
the immersion $(f,\nu)\colon{}M^2\to T_1H^3$ satisfies
$\inner{\nu}{\nu} = 1$ and $\inner{\nu}{df} = 0$.  
A point $x \in M^2$ where $\text{rank} (df)_x < 2$ 
is called a ({\it Legendrian}) {\it singularity\/} or 
{\it singular point\/}. 

The {\em parallel front\/} $f_t$ of a front $f$ at distance 
$t$ is given by $f_t(x)=\Exp_{f(x)}\bigl(t\nu(x)\bigr)$, 
where ``$\Exp$'' denotes the exponential map of $H^3$. 
In the model for $H^3$ as in \eqref{eq:hyp-lor}, we can write
\begin{equation}\label{eq:parallel}
   f_t = (\cosh t) f + (\sinh t) \nu, \quad
   \nu_t =(\cosh t) \nu + (\sinh t) f,
\end{equation}
where $\nu_t$ is the unit normal vector field of $f_t$.

Based on the fact that any parallel surface of a flat surface 
is also flat at regular points (i.e., non-singular points), 
we define flat fronts as follows:
A front $f\colon{}M^2\to H^3$ is called a {\em flat front\/} 
if, for each $x\in M^2$, there exists $t\in\R$ such that the
parallel front $f_t$ is a flat immersion at $x$. 
By definition, $\{f_t\}$ forms a family of flat fronts.
We assume this is the case.
As in \eqref{eq:hyp-mat}, the hyperbolic 3-space $H^3$
can be considered as a subset of $\SL(2,\C)$,
and there exist a complex structure on $M^2$ and 
a holomorphic Legendrian immersion
\begin{multline}\label{eq:hol-lift}
   \E_f:\widetilde M^2\longrightarrow \SL(2,\C)\\
   \text{such that}\quad f=\E_f\E_f^* 
   \quad \text{and}\quad
   \nu = \E_f e_3 \E_f^*\qquad 
    \left( e_3 := \begin{pmatrix}  
                 1 & \hphantom{-} 0 \\
	         0 &  -1 
	   \end{pmatrix}\right),
\end{multline}
where $\widetilde M^2$ is the universal cover of $M^2$ 
(see \cite{GMM} and \cite{KUY1}).
We call $\E_f$ the {\em holomorphic Legendrian lift\/} of the 
flat front $f$.
Here, $\E_f$ being a holomorphic Legendrian map means that 
the $\sl (2, \C)$-valued $1$-form 
$\E_f^{-1}d\E_f$ is off-diagonal, 
where $\sl(2,\C)$ is the Lie algebra of $\SL(2,\C)$
(see \cite{GMM}, \cite{KUY1}, \cite{KUY2}, \cite{KRUY}).
So, we can write
\begin{equation}\label{eq:can-form}
 \E_f^{-1}d\E_f=
  \begin{pmatrix}
   0 & \theta \\
   \omega & 0
  \end{pmatrix},
\end{equation}
for holomorphic $1$-forms $\omega$ and $\theta$ on 
$\widetilde{M}^2$. We call $\omega$ and $\theta$ the {\em canonical forms}.

The first and second fundamental forms $ds^2=\inner{df}{df}$
and $\secondff=-\inner{df}{d\nu}$ are given by 
\begin{equation}\label{fff-sff}
\begin{aligned}
 ds^2&=|\omega+\bar \theta|^2= 
   Q + \overline Q + 
      (|\omega|^2+|\theta|^2), \qquad Q=\omega\theta  , \\
 \secondff&=|\theta|^2-|\omega|^2 . 
\end{aligned}
\end{equation}
Note that
$|\omega|^2$ and $|\theta|^2$ are well-defined on $M^2$ itself, 
though $\omega$ and $\theta$ are 
generally only defined on $\widetilde M^2$. 
The holomorphic $2$-differential $Q$ appearing in the $(2,0)$-part of
$ds^2$ is defined on $M^2$, 
and is called the {\it Hopf differential\/} of $f$.
By definition, the umbilic points of $f$ equal the zeros of $Q$.  
Defining a meromorphic function on $\widetilde{M}^2$ by 
\begin{equation}\label{eq:rho-def}
   \rho=\frac{\theta}{\omega}  , 
\end{equation}
then $|\rho|\colon{}M^2\to[0,+\infty]$ is well-defined
on $M^2$, 
and $x\in M^2$ is a singular point of $f$ if and only
if $|\rho(x)|=1$.

We note that the $(1,1)$-part of the first fundamental form 
\begin{equation}\label{eq:one-one-part}
     ds^2_{1,1}=|\omega|^2+|\theta|^2
\end{equation}
is positive definite on $M^2$ because 
it is the pull-back of 
the canonical Hermitian metric of $\SL(2,\C)$ by the immersion $\E_f$. 
Moreover, $2ds^2_{1,1}$ coincides with the  pull-back of the Sasakian
metric on the unit cotangent bundle $T^*_1 H^3$ by the Legendrian 
lift $L_f$ of $f$ 
(which is the sum of the first and third fundamental forms, 
see \cite[Section 2]{KUY2} for details).  
The complex structure on $M^2$ is 
compatible with
the conformal metric $ds^2_{1,1}$.  
Note that any flat front is orientable (\cite[Theorem B]{KRUY}).
Throughout this paper, we always 
consider $M^2$ as a Riemann surface 
with this complex structure, 
for each flat front $f\colon{}M^2\to H^3$.

The two {\em hyperbolic Gauss maps\/} are defined as
\begin{equation*}
   G=\frac{E_{11}}{E_{21}}, \quad G_*=\frac{E_{12}}{E_{22}}, \qquad
 \text{where}\quad
    \E_f=(E_{ij}).
\end{equation*}
It can be shown that the hyperbolic Gauss maps are 
well-defined as meromorphic functions on $M^2$. In fact, 
geometrically, $G$ and $G_*$ represent the intersection points in the 
ideal boundary $\partial H^3 = \C \cup \{ \infty \}$ of $H^3$ of 
the two oppositely-oriented normal geodesics 
emanating from $f$ in the $\nu$ and $-\nu$ directions, 
respectively 
(see Proposition~\ref{prop:lightcone-g} in the appendix). 
In particular, parallel fronts have the same hyperbolic 
Gauss maps.
For $u\in\SL(2,\C)$, 
the change $\E_f\mapsto u\E_f$ corresponds
to the rigid motion 
$f\mapsto \iota_u\circ f=ufu^*$ 
in $H^3$ as in \eqref{eq:isometry}.
Under this change, 
the hyperbolic Gauss maps change by the M\"obius transformation:
\begin{equation}\label{eq:gauss-moebius}
   G \mapsto u\star G =
    \frac{u_{11}G+u_{12}}{%
          u_{21}G+u_{22}},\qquad
   G_*\mapsto u\star G_*=    
    \frac{u_{11}G_*+u_{12}}{%
          u_{21}G_*+u_{22}},
\end{equation}
where $u=(u_{ij})$,
in contrast to the canonical forms 
$\omega$, $\theta$ which are unchanged.
The canonical forms, the hyperbolic Gauss maps 
and the Hopf differential are related as follows:
Let $g$ and $g_*$ be holomorphic functions on the universal cover
$\widetilde M^2$ of $M^2$ such that $dg=\omega$ and $dg_*=\theta$.
Then it holds that
\begin{equation}\label{eq:schwarz}
  S(g)-S(G) = S(g_*)-S(G_*) = 2Q,\quad
   S(h) = \left\{
   \left(\frac{h''}{h'}\right)'-\frac{1}{2}\left(\frac{h''}{h'}\right)^2
   \right\}dz^2,
\end{equation}
where $z$ is a local complex coordinate and $'=d/dz$, that is, $S(\cdot)$
denotes the Schwarzian derivative with respect to $z$.
The relations in \eqref{eq:schwarz} suggest us that
$G$, $G_*$, $\omega$ and $\theta$ can be considered as pairs
$(G,\omega)$ and $(G_*,\theta)$. 
In fact, as we shall see in \eqref{eq:g-omega-repr},
the front $f$ can be represented via the pair $(G,\omega)$
or $(G_*,\theta)$.
\begin{definition}\label{def:omega-theta}
 The canonical form $\omega$ (resp.\ $\theta$) 
 is said to be {\em associated with} $G$ (resp. $G_*$).
\end{definition}
The holomorphic Legendrian lift $\E_f$ has a $\U(1)$-ambiguity,
that is, 
\[
   \E_f^{\tau}:=\E_f
      \begin{pmatrix} 
       e^{\imag\tau/2} & 0 \\
           0           & e^{-\imag\tau/2}
      \end{pmatrix}\qquad(\tau\in\R)
\]
is also a holomorphic Legendrian lift of $f$.
Under this transformation, the canonical forms and the function $\rho$
change as
\begin{equation}\label{eq:u-one-amb}
   \omega \mapsto e^{\imag\tau}\omega,\qquad
   \theta \mapsto e^{-\imag\tau}\theta,\qquad
   \rho \mapsto e^{-2\imag\tau}\rho, 
\end{equation}
in contrast to the hyperbolic Gauss maps 
$G$, $G_*$ which are unchanged.
On the other hand, the projection of 
\begin{equation}\label{eq:dual}
   \E_f^{\natural} :=
   \E_f\begin{pmatrix} 
	0 & \imag \\ \imag & 0
       \end{pmatrix}
\end{equation}
is also the same front $f$, but the unit normal
$\E_f^{\natural}e_3(\E_f^{\natural})^*=-\nu$ is reversed, where $\nu$ is
the unit normal in \eqref{eq:hol-lift}.
We call $\E_f^{\natural}$ the {\em dual\/} of $\E_f$ 
(see \cite[Remark 2.1]{KRUY}).
The hyperbolic Gauss maps $G^{\natural}$, $G^{\natural}_*$, the
canonical forms $\omega^{\natural}$, $\theta^{\natural}$
and the Hopf differential $Q^{\natural}$ are related to the original data
by 
\begin{equation*}
 (G^{\natural},\omega^{\natural})  = (G_*,\theta),\qquad
 (G^{\natural}_*,\theta^{\natural})  = (G,\omega),\qquad
 Q^{\natural}= Q.
\end{equation*}
A holomorphic Legendrian lift $\E_f$ can be expressed by 
the pair $(G,\omega)$ of the hyperbolic
Gauss map $G$ and the canonical form $\omega$, as in \cite{KUY1}:
\begin{equation}\label{eq:g-omega-repr}
  \E_f = \begin{pmatrix}
	    GC & d(GC)/\omega \\
	     C & dC/\omega
	 \end{pmatrix},\qquad\text{where}\quad
	 C = \imag \sqrt{\frac{\omega}{dG}}.
\end{equation}
We use the above formula in what follows, but, 
using the duality \eqref{eq:dual}, we could express $\E_f$
in terms of the pair $(G_*,\theta)$ as well.  
The fact that we have these two different expressions for 
$\E_f$ will play a crucial role in our 
investigation of the asymptotic behavior of WCF-ends.

Another representation formula for $\E_f$ in terms of the hyperbolic
Gauss maps is given in \cite{KUY1}:
\begin{equation}\label{eq:gauss-repr}
   \E_f = \begin{pmatrix}
	     G/\xi & \xi G_*/(G-G_*) \\
	     1/\xi & \xi/(G-G_*)
	  \end{pmatrix}
	  \qquad
	  \left(
	     \xi = \delta\exp \int_{z_0}^z \frac{dG}{G-G_*}
	  \right),
\end{equation}
where $z_0 \in M^2$ is a base point and $\delta\in\C\setminus\{0\}$ is a 
constant. Note that the choice on $\delta$ corresponds to 
the $\mathrm{U}(1)$-ambiguity, as well as to the family of parallel fronts.
The canonical form $\omega$ and the Hopf differential $Q$ are expressed
as \begin{equation}\label{eq:can-hopf}
   \omega = -\frac{dG}{\xi^2}, \qquad
    Q = -\frac{dG\,dG_*}{(G-G_*)^2}.
\end{equation}
\begin{remark}[Flat surfaces in de Sitter 3-space]
 We set 
 \begin{align*}
  S^3_1 &= 
     \{X\in\Herm(2)\,;\,\det X=-1\} 
  \\
     &= \{ ue_3u^{*}\,;\,u\in\SL(2,\C)\} = \SL(2,\C)/\SU(1,1)
 \quad 
  \left(
       e_3=\begin{pmatrix} 1 & \hphantom{-}0 \\ 0 & -1 \end{pmatrix}
  \right),
 \end{align*}
 which gives a Lorentzian space form of positive curvature
 called {\it de Sitter $3$-space}.
 A smooth map $f\colon M^2\to S^3_1$ is called a {\it spacelike front\/} if
 its unit normal vector field $\nu$ is globally defined on
 $M^2$ and gives a front in $H^3$.
 The unit normal vector field $\nu$ of
 flat fronts in $H^3$ gives spacelike flat fronts in 
 de Sitter 3-space $S^3_1$, and vice versa.
\end{remark}
\begin{remark}[A characterization of the horosphere]%
\label{rmk:G_const}
If either $G$ or $G_*$ is constant,
the Hopf differential $Q$ vanishes everywhere because of
\eqref{eq:can-hopf}.
Then the surface lies in a horosphere.
\end{remark}

\subsection*{Ends of flat fronts}
Let $f\colon{}M^2\to H^3$ be a flat front.
If $M^2$ is homeomorphic to a compact Riemann surface
$\overline M^2$ excluding a finite number of points $p_1,\dots,p_n $,
each point $p_j$ represents an {\em end\/} of $f$.
Moreover, if a neighborhood of $p_j$ is {\em biholomorphic\/} to
the punctured disc 
$D^*=\{z\in\C\,;\,0<|z|<1\}$,
then $p_j$ is called a {\em puncture-type} end.
{\it We often refer to 
the restriction of $f$ to a neighborhood $D^*$ as the end, as well.}

Puncture-type ends can appear in a flat front 
with some kinds of ``completeness'' properties:
A flat front $f\colon{}M^2 \to H^3$ is called {\em complete\/} if
there exists a symmetric $2$-tensor $T$ such that 
$T=0$ outside a compact set $C\subset M^2$ and $ds^2+T$ 
is a complete metric of $M^2$.
In other words, the set of singular points of $f$ is compact and
each divergent path has infinite length (see Definition 
\ref{def:complete} below).  
On the other hand, $f$ is called {\em weakly complete\/} 
(resp.\ {\em of finite type}) if the metric $ds^2_{1,1}$
as in \eqref{eq:one-one-part} is complete 
(resp.\ of finite total curvature).
The following fact is fundamental:
\begin{fact}[{\cite[Proposition 3.2]{KRUY}}]%
\label{fact:finite-top}
 If a flat front $f\colon{}M^2\to H^3$ 
 is  weakly complete and of finite type, 
 then there exists a compact Riemann surface $\overline M^2$
 and a finite set of points $\{p_1,\dots,p_n\}$
 such that $M^2$ is biholomorphic to
 $\overline{M}^2 \setminus \{p_1,\dots,p_n\}$.
\end{fact}
We can also define {\it completeness} of an end itself: 
\begin{definition}\label{def:complete}
 An end $f\colon{}D^*\to H^3$ is 
\begin{itemize}
 \item  {\em complete\/}
	if $f$ is complete at the origin,
	that is, the set of singular points 
	does not accumulate at the origin and 
	any path in $D^*$ approaching the origin has infinite length,
	or
 \item  {\em incomplete WCF\/}
	if $ds^2_{1,1}$ in \eqref{eq:one-one-part} is  complete
	at the origin, the total curvature of $ds^2_{1,1}$ 
	on a neighborhood of the origin is finite,
	and $f$ is incomplete at the origin.
\end{itemize} 
Namely, an incomplete WCF-end is an ``incomplete, 
{\it W}eakly {\it C}omplete
end of {\it F}inite type''.
\end{definition}

\begin{fact}[{\cite{KUY2}, \cite[Proposition 3.1]{KRUY}}]%
\label{fact:complete}
 A complete end is a WCF-end.
 Conversely, a WCF-end $f\colon D^*\to H^3$
 is complete if the  singular set 
 does not accumulate at the origin.
\end{fact}
The weak completeness and the finite-type property of
a complete end are shown in \cite[Corollary 3.4]{KUY2} and 
\cite[Proposition 3.1]{KRUY} respectively.
The second assertion of Fact 
\ref{fact:complete} follows from
\cite[Theorem 3.3]{KRUY}.

\begin{fact}[{\cite{GMM}, \cite{KUY2}, \cite[Proposition 3.2]{KRUY}}]%
\label{fact:finite}
 Let $f\colon{}D^* \to H^3$ be 
 a WCF-end of a flat front. 
 Then the canonical forms $\omega$ and $\theta$ are expressed as 
 \[
     \omega = z^{\mu} \omega_1(z)\,dz,\qquad
     \theta = z^{\mu_*} \theta_1(z)\,dz
     \qquad (\mu,\mu_*\in\R, \mu+\mu_*\in\Z),
 \]
 where
 $\omega_1$ and $\theta_1$ are holomorphic functions in $z$
 which do not vanish at the origin.
 In particular,
 the function $|\rho|\colon{}D^*\to [0,+\infty]$ as in \eqref{eq:rho-def}
 can be extended  across the end\/ $0$.
\end{fact}
Here $|\omega|^2$ and $|\theta|^2$ are considered as conformal 
flat metrics
on $D^*_\varepsilon$ for sufficiently small $\varepsilon>0$.
The real numbers $\mu$ and $\mu_*$ are the orders of the metrics
$|\omega|^2$ and $|\theta|^2$ at the origin respectively, that is,
\begin{equation}\label{eq:order-canonical}
    \mu=\ord_0|\omega|^2,\qquad
    \mu_*=\ord_0|\theta|^2.
\end{equation}
Since $ds^2_{1,1}=|\omega|^2 + |\theta|^2$ 
in \eqref{eq:one-one-part} is complete at the origin,
it holds that 
\begin{equation}\label{eq:WC-order}
    \min \{\mu,\mu_*\}=
    \min\left\{
	  \ord_0|\omega|^2,\ord_0|\theta|^2
	\right\}\leq -1
\end{equation}
for a WCF-end.  By \eqref{fff-sff},
the order of the Hopf differential is 
\begin{equation}\label{eq:Q-omega-theta}
    \ord_0 Q = \mu +\mu_* = \ord_0|\omega|^2 + \ord_0 |\theta|^2,
\end{equation}
 where 
 $\ord_0 Q=m$ if $Q=z^m\bigl(a+o(1)\bigr)dz^2$ holds for 
 some $a\neq 0$.

The following assertion is essentially shown in the proof of 
\cite[Theorem 3.4]{KRUY}. 
However, for the sake of convenience we give a proof here.
\begin{proposition}\label{prop:parallel}
 Let $f\colon{}D^*\to H^3$ be a complete end of a flat front.
 Then the parallel front 
 $f_t$ as in \eqref{eq:parallel} is a WCF-end.  
 Conversely, for an incomplete WCF-end $f\colon{}D^*\to H^3$ 
 of a flat front, $f_t$ is a complete end
 for any $t\ne 0$.
\end{proposition}

\begin{proof}
 The canonical forms of the parallel front $f_t$ are 
 expressed by those of $f$ as
 \[
    \omega_t = e^{t}\omega,\qquad
    \theta_t = e^{-t}\theta,\qquad
    \rho_t = e^{-2t}\rho.
 \]
 By \eqref{eq:WC-order}, the completeness of $ds^2_{1,1}$
 is preserved by taking parallel fronts. 
 On the other hand, it follows from 
 \cite[(3.2)]{KRUY} that the finiteness of total curvature 
 of $ds^2_{1,1}$ for $f_t$ is equivalent the finiteness of orders of
 $|\omega_t|^2$ and $|\theta_t|^2$ at the end $z=0$.
 This implies that the finiteness of the total curvature of $ds^2_{1,1}$
 is also preserved by taking parallel fronts.
 By Fact~\ref{fact:complete}, if $f$ is complete, it is a weakly
 complete end of finite type.  
 Hence so is $f_t$, that is, $f_t$ is WCF.
 Conversely, if $f$ is incomplete WCF,
 $|\rho_t(0)|=e^{-2t}|\rho(0)|=e^{-2t}\neq 1$ for $t\neq 0$.
 Since the singular point $z\in D^*$ of $f_t$ is
 characterized by $|\rho_t(z)|=1$,
 the singular set of $f_t$ ($t\ne 0$) does not accumulate at the  end. 
 Hence $f_t$ ($t\ne 0$) is complete at the origin.
\end{proof}

\subsection*{Behavior of regular ends}

\begin{fact}[\cite{GMM}, \cite{KUY2}]\label{fact:regular}
 The hyperbolic Gauss maps 
 $G$, $G_*$ of a weakly complete  end $f\colon{}D^*\to H^3$
 of a flat front satisfy either
 \begin{itemize}
  \item both $G$ and $G_*$ have at most pole singularities 
        at the origin, and have the same value at the end, or
  \item both $G$ and $G_*$ have essential singularities at the origin.
 \end{itemize}
\end{fact}

We call the end {\em regular\/} if both $G$ and $G_*$ have at most
poles, and {\em irregular\/} otherwise.
By \eqref{eq:schwarz} and Fact~\ref{fact:finite},
we have
\begin{lemma}[\cite{GMM}, \cite{KUY2}]\label{lem:Q-regular}
 A WCF-end $f\colon{}D^*\to H^3$ of a flat front is regular if and only
 if the Hopf differential has a pole of order at most $2$ at $0$, 
that is, $\ord_0 Q\geq -2 $ holds.
\end{lemma}
\begin{proposition}\label{prop:end-limit}
 Let $f\colon{}D^*\to H^3$ be a regular weakly complete
 end, and denote its hyperbolic Gauss maps by $G$ and $G_*$.
 Then
 \[
     \lim_{z\to 0}\pi\circ f(z)=
            \bigl(G(0),0\bigr)=\bigl(G_*(0),0\bigr)
 \]
 holds if $G(0)\ne \infty$,
 where $\pi\colon{}H^3\to\R^3_+$ is the projection
 as in \eqref{eq:proj0}.  
\end{proposition}

\begin{proof}
 The proof of \cite[Lemma 3.10]{KUY2}
 applies to weakly complete ends as well.
 So we have
 $G(0)=G_*(0)(=a)$, where $a \in \C\cup\{\infty\}$.
 By a suitable rigid motion in $H^3$, we may assume $a \neq\infty$.
 In this case, we can write
 \[
     G = a + \psi(z),\quad
     G_* = a+\psi_*(z),
    \qquad \bigl(\psi(0)=\psi_*(0)=0\bigr),
 \]
 where $\psi(z)$, $\psi_*(z)$ are holomorphic
 functions defined on a sufficiently small closed disc 
 $\{|z|\le \varepsilon\}$.
 We write $\pi\circ f = (\zeta,h)$.
 Then 
 by \eqref{eq:gauss-repr}, we have
\[
   \frac{1}{h}=|E_{21}|^2+|E_{22}|^2=
   \frac{1}{|\xi|^2}+\frac{|\xi|^2}{|G-G_*|^2}
   \ge \frac{2}{|G-G_*|}\longrightarrow +\infty
   \qquad (z\to 0).
\]
 In particular, $h\to 0$ as $z\to 0$.
 On the other hand, we have
 \begin{align*}
    \zeta
      &= \frac{E_{11}\overline E_{21}+E_{12}\overline{E_{22}}}
      {|E_{21}|^2+|E_{22}|^2}
       = 
    \frac{G|E_{21}|^2+G_*|E_{22}|^2}
      {|E_{21}|^2+|E_{22}|^2}\\
     &=\frac{(a+\psi(z))|E_{21}|^2+
            (a+\psi_*(z))|E_{22}|^2}
             {|E_{21}|^2+|E_{22}|^2}
     =a + 
         \frac{\psi(z)|E_{21}|^2+
         \psi_*(z)|{E_{22}}|^2}
             {|E_{21}|^2+|E_{22}|^2}.
 \end{align*}
 Thus we have
 \[
      |\zeta -a|\le  \max_{|z|\le \varepsilon}
      \{|\psi(z)|,|\psi_*(z)|\}.
 \]
 The right-hand side tends to zero as $\varepsilon \to 0$, hence 
 the left-hand side $|\zeta -a|$ converges to zero as $z\to 0$.
  This completes the proof.
\end{proof}

From now on, we consider a regular WCF-end $f\colon{}D^*\to H^3$ of a
flat front.
By a rigid motion, we may assume $G(0)(=G_*(0)) \ne \infty$.
In the case where $G$ and $G_*$ are both non-constant
(cf.\ Remark \ref{rmk:G_const}),
we have the expressions
\begin{equation}\label{eq:gauss-normal0}
\begin{aligned}
   G(z)&=a+b_1 z^{m_1}+o(z^{m_1}), \\
 G_*(z)&=a+b_2 z^{m_2}+o(z^{m_2}),
\end{aligned}
   \qquad \bigl(a=G(0)=G_*(0),~ b_1,b_2\in \C\setminus\{0\}\bigr)
\end{equation}
on a neighborhood of $z=0$. 
We set
\begin{equation}\label{eq:multiplicity}
   m=\min\{m_1,m_2\}=\min\{\ord_0\,G'(z),~\ord_0\,G'_*(z)\}+1,
\end{equation}
which is called the {\it multiplicity\/} of the end
$f$.
In the case where one of $G$, $G_*$ is constant, 
the {\it multiplicity\/} of the end is defined to be the 
ramification order of whichever of $G$ or $G_*$ is nonconstant.
The multiplicity $m$ of the end has the following important property.

\begin{fact}[\cite{GMM} and \cite{KUY2}]
 Let $f:D^*\to H^3$ be a complete regular end.
 Then the multiplicity $m$ of $f$ is equal to $1$
 if and only if $f(D_{\varepsilon}^*)$ 
 is properly embedded for a sufficiently small $\varepsilon>0$.
\end{fact}

Recall that (see \cite[(7.1)]{KRUY}) the constant 
\begin{equation}\label{eq:gauss-ratio}
\alpha: = \begin{cases}
    (dG_*/dG)(0) & (\text{if $|(dG_*/dG)(0)|\le 1$}), \\
    (dG/dG_*)(0) & (\text{if $|(dG_*/dG)(0)|>1$})
   \end{cases}  
\end{equation}
is called the {\it ratio of the Gauss maps}.
As seen in \cite[Propositions 3.1 and 7.3]{KRUY},
$\alpha$ is a real number which is not equal to $1$, 
so $\alpha\in [-1,1)$.

In particular, if $\alpha\neq 0$, the ramification orders of $G$ and
$G_*$ coincide, and are equal to the multiplicity of the end.

To fix the expression of the ratio of Gauss maps uniquely, 
we wish to distinguish the pairs $(G,\omega)$ and $(G_*,\theta)$
of $f$
(given just before Definition \ref{def:omega-theta})
as follows:
\begin{definition}
 The pair $(G,\omega)$ {\rm (}resp. $(G_*,\theta)${\rm)}
 is  a {\it dominant pair\/}
 with respect to the regular end $z=0$ if
 $|(dG_*/dG)(0)|\le 1$ {\rm(}resp. $|(dG_*/dG)(0)|\ge 1${\rm)}.
 Moreover, $(G,\omega)$ {\rm (}resp. $(G_*,\theta)${\rm )} is called
 the {\it strictly dominant pair\/} if $|(dG_*/dG)(0)|< 1$ 
 {\rm (}resp. $|(dG_*/dG)(0)|> 1${\rm)}.
\end{definition}
\begin{remark}\label{rem:dominant-pair}
 For a regular WCF-end,
 $(G,\omega)$ and $(G_*,\theta)$
 are both dominant if and only if $\alpha=-1$, 
 which corresponds to a regular cylindrical end 
 (see Definition~\ref{def:types} and Proposition~\ref{lem:alpha-mu} below).
 If $(G_*,\theta)$ is strictly dominant, $(G,\omega)$ is not strictly dominant. 
 In this case, by taking the dual as in  \eqref{eq:dual}, 
 we can exchange the roles of $(G,\omega)$ and $(G_*,\theta)$.
 Thus, we may always assume that $(G,\omega)$ is a dominant pair.
 Then it holds that
 \begin{equation}\label{eq:remark_new}
     \alpha=(dG_*/dG)(0),\qquad m=m_1.
 \end{equation}
 In particular, we have the
 expressions
 \begin{equation}\label{eq:gauss-normal}
  \begin{aligned}
     G(z)&=a+c z^m+o(z^m), \\
     G_*(z)&=a+\alpha c z^m+o(z^m),
  \end{aligned}
   \qquad \bigl(c\ne 0 \bigr).
 \end{equation}
\end{remark}
We shall use frequently these expressions, or more normalized forms of
them.

The following assertion holds:
\begin{proposition}\label{prop:alpha-mu-new}
 Let $(G,\omega)$ be a dominant pair.
 Then the ratio of Gauss maps $\alpha$ and the multiplicity $m$ of the
 end satisfy the following identity
 \begin{equation}\label{eq:alpha-mu}
     \mu=-\frac{1+\alpha}{1-\alpha}m-1(\le -1),\qquad
     \text{that is,}\qquad
     \alpha = \frac{1+\mu+m}{1+\mu-m},
 \end{equation}
 where $\mu=\ord_0 |\omega|^2$.
 In particular, $\mu_*=\ord_0 |\theta|^2$ satisfies
 \begin{equation}\label{eq:mu-mu*}
  \mu+\mu_*\ge -2,\quad \mu_*\ge -1,\quad (\mu\le -1). 
 \end{equation}
\end{proposition}
\begin{proof}
 Substituting \eqref{eq:gauss-normal}
 into \eqref{eq:gauss-repr} and \eqref{eq:can-hopf}, 
 and noticing that $\alpha\neq 1$, we have  \eqref{eq:alpha-mu}.
 The second assertion follows from
 \eqref{eq:Q-omega-theta} and Lemma \ref{lem:Q-regular}.
\end{proof}

Since the Hopf differential $Q$ is written as in
\eqref{eq:can-hopf}, 
it has the following expansion 
\begin{equation*}
     Q=\frac{1}{z^2}\left( 
           \frac{-m^2\alpha}{(1-\alpha)^2}+o(1)
		    \right)dz^2,
\end{equation*}
where $o(1)$ is a term having order
higher than $1$ as $z\to 0$.
The term
\begin{equation}\label{eq:Q-top}
    q_{-2}:=\frac{-m^2\alpha}{(1-\alpha)^2}
\end{equation}
is called the {\it top-term coefficient\/} of  $Q$.
\begin{definition}[cf. {\cite[Definition 7.1]{KRUY}}]%
\label{def:types}
 A regular WCF-end $f\colon{}D^*\to H^3$ of a flat front is called
 \begin{enumerate}
 \item {\em horospherical\/} if $q_{-2}=0$, that is, $\ord_0 Q\geq -1$,
  \item of {\em snowman-type\/} if $q_{-2}<0$, 
  \item of {\em hourglass-type\/} 
	if $q_{-2}>0$ and $\ord_0|\omega|^2 \ne \ord_0|\theta|^2$, or
  \item {\em cylindrical\/} 
	if  $\ord_0|\omega|^2=\ord_0|\theta|^2$. 
	(In this case, $q_{-2}$ is positive. See 
	  Corollary \ref{cor:alpha-mu} below.)
 \end{enumerate}
\end{definition}
These types of ends are characterized as follows:

\begin{proposition}\label{lem:alpha-mu}
 Let $f\colon{}D^*\to H^3$ be a regular WCF-end of a flat front
 and $(G,\omega)$ a dominant pair.
 Then the end is
 \begin{enumerate}
  \item horospherical if and only if $\alpha=0$, that is, $\mu=-m-1$,
  \item snowman-type if and only if $0<\alpha<1$, that is, $\mu<-m-1$,
  \item hourglass-type if and only if $-1<\alpha<0$, that is,
	$-m-1<\mu<-1$, and
  \item cylindrical if and only if $\alpha=-1$, that is, $\mu=-1$.
 \end{enumerate}
\end{proposition}
\begin{proof}
 If $\alpha\neq 0$, $\ord_0 Q=-2$ because of \eqref{eq:Q-top},
 and 
 \begin{equation}\label{eq:ord-dual}
    \ord_0|\theta|^2=-2-\ord_0|\omega|^2=
       -2-\mu  =-\frac{\alpha+1}{\alpha-1}m-1,
 \end{equation}
 because of \eqref{eq:Q-omega-theta} and \eqref{eq:alpha-mu}.
 Then the conclusion follows.
\end{proof}%
\begin{corollary}\label{cor:alpha-mu}
 If a regular WCF-end of a flat front is cylindrical, then
 it holds that
 \[
     q_{-2}>0,\qquad\text{and}\qquad
     \ord_0|\omega|^2=\ord_0|\theta|^2=-1.
 \]
\end{corollary}
\begin{proof}
 Substitute $\alpha=-1$ into \eqref{eq:Q-top} and \eqref{eq:ord-dual}.
\end{proof}

\begin{example}[Flat fronts of revolution]
 \label{ex:revolution}
 Take a positive integer $m$ and $\alpha\in [-1,1)$,
 and set $(G,G_*)=(z^m,\alpha z^m)$.
 Then by \eqref{eq:gauss-repr}, we have a flat front
 $f\colon{}\C\setminus\{0\}\to H^3$ whose canonical forms 
 are given by
 \[
     \omega = -\frac{m}{\delta^2} z^\mu\,dz,\qquad
     \theta = \frac{m\alpha\delta^2}{(1-\alpha)^2} z^{-2-\mu}\,dz
     \qquad
     \left(
          \mu=\frac{\alpha+1}{\alpha-1}m-1
     \right),
 \]
 where $\delta$ is a constant as in \eqref{eq:gauss-repr}.
 The front $f$ is 
 the $m$-fold cover of the {\em hourglass\/} 
 (resp.\ {\em the snowman}\/)
 if $-1<\alpha<0$ (resp. $0<\alpha<1$).
 When $\alpha=0$, $f$ gives the horosphere
 (resp.\ the $m$-fold branched cover of the {\em horosphere\/}
 with branch point $z=\infty$)
 if $m=1$ (resp.\ $m\geq 2$).
 In the case of $\alpha=-1$, 
 $f$ gives the $m$-fold cover of a {\em cylinder\/} if $|\delta|^2\neq 2$.
 When $\alpha=-1$ and $|\delta|^2=2$, all points are singularities
 of $f$, and the image $f(\C\setminus\{0\})$ is the geodesic
 joining $0$ and $\infty$.
 Here, we identify $\partial H^3$ with  $\C\cup\{\infty\}$
 as in \eqref{eq:asymptotic} in the appendix.
 In all cases, 
 $f$ is a flat front of revolution 
 whose axis is the geodesic joining $0$ and $\infty\in\partial H^3$,
 see Figure~\ref{fig:revolution} in the introduction.

 Conversely, any flat front of revolution whose axis 
 is the geodesic joining $0$ and $\infty\in\partial H^3$ is 
 obtained in such a way.
 In particular, one can choose the complex coordinate $z$ such that
 $G=z^m$, and the canonical form 
 $\omega=cz^{\mu}\,dz$, where $c$ is a non-zero constant.
\end{example}

\subsection*{Behavior of singular points on a regular WCF-end}
\begin{lemma}\label{lem:incomplete}
 A WCF-end $f\colon{}D^*\to H^3$ of a flat front is 
 cylindrical if and only if $\rho(z)=\theta/\omega$
 as in \eqref{eq:rho-def} is a nonvanishing holomorphic function
 near $z=0$.
 On the other hand, a regular WCF-end $f$ is 
 incomplete if and only if
 it is cylindrical and $|\rho(0)|=1$.
\end{lemma}
\begin{proof}
 Note that this lemma holds not only for regular ends but for 
 WCF-ends.
 The first assertion is obvious.
 In particular, if the end is not cylindrical, $\mu\neq \mu_*$
 in \eqref{eq:order-canonical}
 and then
 \[
     \lim_{z\to 0} |\rho(z)|= 0 \quad\text{or}\quad
                              +\infty.
 \]
 This implies that the singular set $\{|\rho|=1\}$ does not 
 accumulate at the origin.
 Thus incomplete ends are all cylindrical.
 Moreover, if the singular set accumulates at the origin, then 
 $|\rho(0)|=1$.
 Conversely, assume that $f$ is cylindrical and $|\rho(0)|=1$.
 By the $\U(1)$-ambiguity as in \eqref{eq:u-one-amb}, one can
 assume $\rho(0)=1$ without loss of generality.
 If $\rho$ is constant, all points are singular, and
 then the end is incomplete.  Otherwise, $\rho$ 
 can be expanded as 
 $\rho(z)=1+b z^n +o(z^n)$ ($b\neq 0$),
 where $n$ is the ramification order of $\rho$.
 Hence one can take a complex coordinate  $w$ ($w(0)=0$)
 such that
 \begin{equation}\label{eq:rho-can}
     \log \rho = w^n.
 \end{equation}
 Then the singular set 
 \begin{equation}\label{eq:sing-norm}
      \{|\rho|=1\}=  \{w\,;\,\Re (w^n) = 0 \} 
 \end{equation}
 accumulates at the origin.
\end{proof}

\begin{proposition}\label{lem:cuspidal}
 Let $f \colon D^*\to H^3$ be an incomplete regular WCF-end
 of a flat front, 
 whose  image is not contained in a geodesic line in $H^3$.
 Then, for a sufficiently small $\varepsilon>0$,
 only cuspidal edge singularities 
 appear in the image $f(D^*_\varepsilon)$,
 and the set of cuspidal edges has $2n$ components,
 where $n$ is the ramification order of 
 $\rho(z)$ at $z=0$.
\end{proposition}
\begin{proof}
 By Lemma~\ref{lem:incomplete},  
 $|\rho|$ is a well-defined function on a neighborhood of the origin
 satisfying $|\rho(0)|=1$.
 By the $\U(1)$-ambiguity as in \eqref{eq:u-one-amb},
 we may assume $\rho(0)=1$ without loss of generality.
 If $\rho$ is identically $1$, then
 \cite[Proposition 4.7]{KRSUY} yields that $f$ is rotationally
 symmetric, and then, the image of $f$ is a geodesic line.
 Thus  $\rho$ is not identically $1$, and then
 we can take a complex local coordinate $w$
 around the origin  as in \eqref{eq:rho-can}.
 Hence the set of singularities is expressed as in \eqref{eq:sing-norm},
 which consists of $2n$ rays starting at the origin 
 in the $w$-plane.

 By Proposition~\ref{lem:alpha-mu} and Lemma~\ref{lem:incomplete}, 
 we have $\alpha=-1$.
 Thus the Hopf differential $Q$ expands as 
 \begin{equation}\label{eq:Q-top-incomplete}
      Q = \frac{m^2}{4z^2}\bigl(1+o(1)\bigr)\,dz^2,
 \end{equation}
 where $m$ is the multiplicity of the end.
 On the other hand, a singular point which is not a cuspidal edge point
 must be a zero of the imaginary part of 
 the function 
 \[
    \sqrt{\zeta_c}
            =\frac{d\bigl(\log\rho\bigr)}{\sqrt{Q}},
 \]
 (see \cite[Proposition~4.7]{KRSUY}), and we have the following expansion
 \[
     \sqrt{\zeta_c(w)}=\frac{2n}{m} w^{n}\bigl(1+o(1)\bigr),
 \]
 where $w$ is the local coordinate near the origin given in
 \eqref{eq:rho-can}.
 Then the singular set is given by $\{\Re(w^n)=0\}$, and 
 the zeros of the imaginary part of $\sqrt{\zeta_c}$ are approximated by 
 $\{\Im(w^n)=0\}$.
 Since  the two  sets $\{\Re(w^n)=0\}$ and 
 $\{\Im(w^n)=0\}$ are disjoint near $w=0$,
 there are no singular points  
 other than cuspidal edge points near $z=0$.
\end{proof}

\section{Flux and axes of ends}\label{sec:flux}
\subsection*{The flux matrix}
Let $f \colon D^*\to H^3$ be an end of a flat front such that the
complex structure of $D^*$ is compatible with the metric
\eqref{eq:one-one-part}.
Regarding $H^3\subset \SL(2,\C)$ as in \eqref{eq:hyp-mat}, 
the {\it flux matrix\/} of $f$ is defined by 
\begin{equation}\label{eq:flux}
   \Phi_f:=\frac{\imag}{2\pi}\int_{\gamma} 
         (\partial f) f^{-1}  \in\sl(2,\C), 
\end{equation}
where $\sl(2,\C)$ is the Lie algebra of $\SL(2,\C)$
and $\gamma$ is an arbitrary loop in $D^*$ 
going around the origin in the counterclockwise direction.
Here, $\partial f$ is the $(1,0)$-part of $df$, that is, 
$\partial f = f_z\,dz$ for a complex coordinate $z$.

We first show the following:
\begin{proposition}\label{prop:flux}
 Let $\E_f \colon \widetilde D^*\to \SL(2,\C)$ be a holomorphic
 Legendrian lift of $f$, where $\widetilde D^*$ is the
 universal cover of $D^*$.  
 Then the following formula holds{\rm :}
 \begin{equation}\label{eq:two-equalities-in-one-line}
    (\partial f) f^{-1}
        =d\E^{}_f \E_f^{-1}=
	\frac{1}{(G-G_*)^2}
	\begin{pmatrix}
	 -G_*dG-G dG_* & G_*^2dG+G^2 dG_* \\
	 -dG-dG_* & G_*dG+GdG_*
	\end{pmatrix} \; , 
 \end{equation}
 where $G$ and $G_*$ are the hyperbolic Gauss maps.
 In particular, the $\sl(2,\C)$-valued $1$-form
 $(\partial f) f^{-1}$
 is holomorphic, and common to the parallel family 
 $\{f_t\}_{t\in \R}$, 
 that is, $(\partial f) f^{-1}=(\partial f^{}_t) f^{-1}_t$
 holds.
\end{proposition}
\begin{proof}
 Since $\E_f$ is holomorphic,
 $(\partial f) f^{-1}
     =d\E^{}_f \E_f^{-1}$ follows from $f=\E^{}_f\E_f^*$.
 Then by \eqref{eq:gauss-repr}, we have the conclusion.
\end{proof}
As defined in \cite{KRUY}, a smooth map
$f \colon M^2\to H^3$ on a $2$-manifold $M^2$ 
is called a flat {\it p-front\/} if for each $x\in M^2$, 
there exists a neighborhood $U$ of $x$ such that 
the restriction of $f$ to $U$ is a flat front.
Roughly speaking,  a p-front is locally a front,
but its unit normal vector field $\nu$ may not be
globally single-valued.
The caustics (i.e., focal surfaces)
of flat fronts are also flat, but in general 
they are not fronts but only p-fronts.  
So if we wish to analyze the asymptotic behavior of
ends of caustics, 
we must work in the category of p-fronts.
A p-front is called {\em non-co-orientable\/} if it is not a front.

We now assume $f\colon M^2\to H^3$ is a weakly complete flat p-front
of finite type. 
Since Fact \ref{fact:finite-top} holds also for flat p-fronts
(see \cite[Proposition 5.4]{KRUY}), 
there exist a compact Riemann surface $\overline M^2$
and a finite set of points $\{p_1,\dots,p_n\}$ such that
$M^2$ is biholomorphic to 
$\overline M^2\setminus \{p_1,\dots,p_n\}$.
Though $G$ and $G_*$ may have essential singularities
at $p_j$, the holomorphic form $(\partial f) f^{-1}$ is a 
globally defined $\sl(2,\C)$-valued $1$-form on $M^2$.
Thus the total sum of the residues at $p_1,\dots,p_n$ vanishes:
\begin{corollary}[The balancing formula]\label{lem:flux} 
 Let $f \colon M^2\to H^3$ be a weakly complete flat p-front
 of finite type. 
 Then the sum of flux matrices  over its ends vanishes.
\end{corollary}
This suggests that the flux matrices just defined might 
be useful for the global study of flat fronts,
like as for the cases of CMC surfaces in $\R^3$ 
(cf.\ \cite{KKS}) and CMC-1 surfaces in $H^3$ (cf. \cite{RUY}).

By definition, the flux matrices are meaningful not only for 
regular ends but also irregular ends for which the hyperbolic Gauss
maps have essentially singularities at the end.
However we shall treat only regular ends in this paper.

On the other hand, if the end is not a front but a p-front,
by taking the double cover, it becomes a front 
(see \cite[Corollary 5.2]{KRUY}). 
So, from now on, 
we shall usually work in the category of fronts.

Next, we shall define a projection:
\begin{equation}\label{eq:c-2-proj}
   \Pi:\C^2 \setminus \left\{\begin{pmatrix}
			   0 \\ 0
	  		  \end{pmatrix}\right\} 
    \ni \begin{pmatrix}     x \\ y \end{pmatrix}
    \longmapsto \frac{x}y\in P^1(\C)=\C\cup\{\infty\}  .
\end{equation}
Since, for the Poincar\'e upper half-space model
$\R_+^3$, the ideal boundary $\partial H^3$
can be identified with $\C\cup \{\infty\}$
(see \eqref{eq:asymptotic} in the appendix),
the image of $\Pi$ is contained in $\partial H^3$.
\begin{theorem}\label{thm:limit}
 Let a flat front $f \colon D^*\to H^3$ be a regular WCF-end. 
 Then there exists an eigenvector
 $\vect{v}$ of the flux matrix $\Phi_f$ such that
 the projection $\Pi(\vect{v})\in \C\cup\{\infty\}$
 equals the limiting value of the end, 
 that is, 
 \[
    \Pi(\vect{v}) =\lim_{z\to 0} \zeta(z)=G(0)=G_*(0) \; , 
 \]
 where $\pi\circ f(z) =\bigl(\zeta(z), h(z)\bigr)$. 
 Moreover, the flux matrix $\Phi_f$ is a lower triangular matrix if
 $G(0)=0$.
\end{theorem}
\begin{proof}
 An isometric action $\iota_{u}$ ($u\in\SL(2,\C)$), 
 as in \eqref{eq:isometry}, 
 induces a flat front $ufu^*$ congruent to $f$,
 and the flux matrix of $ufu^*$ is given by
 $u \Phi_f u^{-1}$.  
 This implies that an eigenvector of
 $u \Phi_f u^{-1}$ must be $u\vect{v}$.
 On the other hand, the isometric action induces
 a transformation of the ideal boundary so that 
 \[
   \partial H^3=\C\cup\{\infty\}\ni \zeta\mapsto
     u\star \zeta:=
         \frac{u_{11}\zeta+u_{12}}{%
               u_{21}\zeta+u_{22}} \in 
      \C\cup \{\infty\}=\partial H^3 \quad 
         \bigl(u=(u_{ij})\bigr).
 \]
 Thus we have
 \[
    \Pi(u\vect{v})=u\star \Pi(\vect{v})\qquad 
         \left(\vect{v}\in \C^2\setminus \left
\{\begin{pmatrix} 0\\ 0\end{pmatrix} \right\}\right),
 \]
 which implies that the map $\Pi$ is equivariant.
 So to prove the assertion, we may 
 assume that $G(0)=G_*(0)=0$, replacing $f$ by $ufu^*$
 for a suitable isometry $u\in \SL(2,\C)$ if necessary.
 Without loss of generality,
 we may assume that $(G,\omega)$ is a dominant pair
 (see Remark \ref{rem:dominant-pair}).
 Then the hyperbolic Gauss maps are written as in \eqref{eq:gauss-normal} 
 for $a=G(0)=0$.
 Thus, we have
 \begin{align}
    G_*^2 dG+G^2 dG_*&=
      \left(mc^3\alpha(\alpha+1)z^{3m-1}+o(z^{3m-1})\right)\,dz,
  \label{eq:numerator}\\
    (G-G_*)^2&=c^2(1-\alpha)^2z^{2m}+o(z^{2m}).
  \label{eq:denominator}
 \end{align}
 Since the ratio $\alpha$ of the Gauss maps is not equal to $1$,
 $(G_*^2 dG+G^2 dG_*)/(G-G_*)^2$
 is a holomorphic  $1$-form at $z=0$. 
 In particular, 
 it follows from \eqref{eq:flux}, 
 \eqref{eq:two-equalities-in-one-line} that
 the flux matrix $\Phi_f$
 is a lower triangular matrix, and
 $\vect{v}=\begin{pmatrix}0\\1\end{pmatrix}$ is one of the
 eigenvectors.
 Then we have $\Pi(\vect{v})=0=G(0)=G_*(0)$
 which proves the assertion.
\end{proof}
The eigenvalues of the flux matrix are related to the ratio of 
the Gauss maps:
\begin{theorem}\label{thm:flux-eigenvalue}
 The eigenvalues of the flux matrix $\Phi_f$ of a regular WCF-end 
 $f : D^* \to H^3$ of a flat front are
 \[
      \pm\frac{2m\alpha}{(1-\alpha)^2}
      \left(
         =\mp\frac{2q_{-2}}{m}
      \right),
 \]
 where  $m$ is the multiplicity of the end
 {\rm (}cf.\ \eqref{eq:gauss-normal}{\rm)}, $\alpha$ is
 the ratio of the Gauss maps \eqref{eq:gauss-ratio},
 and 
 $q_{-2}$ is the top-term coefficient of the Hopf differential
 \eqref{eq:Q-top}.
 In particular, if $\alpha\ne 0$
 {\rm (}that is, if $f$ is not horospherical\/{\rm )},
 then $\Phi_f$ is diagonalizable.
\end{theorem}
\begin{proof}
 Let us take the same notation as in the proof of 
 Theorem~\ref{thm:limit}. 
 Then
 the diagonal components are just the eigenvalues of $\Phi_f$  
 since $\Phi_f$ is a triangular matrix.  
 We have 
 $G_* dG + G dG_* = 
  \left(2mc^2\alpha z^{2m-1}+o(z^{2m-1})\right)dz$,
 and by \eqref{eq:denominator}, we have
 \[
     \frac{G_*dG+GdG_*}{(G-G_*)^2}=
     \frac{1}{z}\left( 
          \frac{2m\alpha}{(1-\alpha)^2}+o(1)
        \right)dz. 
 \]
 It follows from \eqref{eq:flux}, 
 \eqref{eq:two-equalities-in-one-line} that
 the eigenvalues of $\Phi_f$ are $\pm {2m\alpha}/{(1-\alpha)^2}$, 
 which are equal to $\mp 2q_{-2}/{m}$ by \eqref{eq:Q-top}.
\end{proof}

By Theorem~\ref{thm:flux-eigenvalue},
$\Phi_f$ is diagonalizable if $\alpha\ne 0$.
In this case,
the flux matrix has two linearly independent
eigenvectors $\vect{v}_1$, $\vect{v}_2$.
Then there exists a unique geodesic line in $H^3$ connecting
$\Pi(\vect{v}_1)$ and $\Pi(\vect{v}_2)=G(0)$ $\in \partial H^3$,
which is called the {\it axis of the flux matrix\/} $\Phi_f$. 
We denote this geodesic by $\overline{\vect{v}_1,\vect{v}_2}$.
By Theorem~\ref{thm:limit}, one endpoint of the axis is 
just the limit point of $f$.

We finish this subsection with some lemmas concerning the axis, 
which will be needed in the following sections.
\begin{lemma}\label{lem:flux-diagonal}
 Let $f \colon D^*\to H^3$ be a regular WCF-end of a flat front.
Then the flux matrix
$\Phi_f$ is diagonal if and only if the axis is the geodesic joining 
 the origin and infinity, 
 i.e., the $h$-coordinate 
axis $\{(\zeta,h) ; \zeta=0 \}$ in the upper half-space 
model $\R^3_+$.
\end{lemma}
\begin{proof}
 If $\Phi_f$ is a diagonal matrix, the eigenvectors are 
 $\vect{v}_1=\begin{pmatrix}1\\0\end{pmatrix}$ and
 $\vect{v}_2=\begin{pmatrix}0\\1\end{pmatrix}$, 
 and vice versa. 
 In this case, 
 $\Pi(\vect{v}_1)=\infty$ and $\Pi(\vect{v}_2)=0$. 
\end{proof}
\begin{lemma}
 Let $f \colon D^*\to H^3$ be a regular WCF-end of a flat front.
 Assume that $\Phi_f$ has the axis $\overline{\vect{v}_1,\vect{v}_2}$. 
 Let $\tilde f$ be an end congruent to $f$, that is,  
 $\tilde f = ufu^{*}$ for some $u \in \SL(2, \C)$.  
 Then the axis of the flux matrix $\Phi_{\tilde f}$ is given by 
 $\overline{u\vect{v}_1,u\vect{v}_2}$.
\end{lemma}
\begin{proof}
 As we have already noted in the proof of Theorem~\ref{thm:limit}, 
 if $\vect{v}$ is an eigenvector of $\Phi_f$, 
 then $u\vect{v}$ is an eigenvector of $\Phi_{\tilde f}$. 
\end{proof}

\begin{lemma}\label{lem:vertical-axis}
 Let $f\colon D^*\to H^3$ be a regular WCF-end of a flat front.
 If the ratio $\alpha$ of the Gauss maps is not zero, 
 then there exists an isometry $\iota$ of $H^3$ such that 
 the axis of the flux matrix $\Phi_{\iota \circ f}$ coincides with 
 the  geodesic joining $\infty$ and $0$, that is, the
 $h$-coordinate axis in $\R^3_+$.
\end{lemma}
\begin{proof}
 As we have seen, $\Phi_f$ is diagonalizable if $\alpha \ne 0$. 
 Hence this lemma is a direct consequence of the two lemmas above. 
\end{proof}

\subsection*{The indentation number}
In this subsection, we introduce the {\em maximum indentation number\/}
$n$ of a regular WCF-end, which is a positive integer.
Firstly, we will define a positive integer $l_\gamma$,
called the {\em indentation number},
determined by the choice of geodesic $\gamma$ asymptotic to 
the point $G(0)=G_*(0)\in \partial H^3$ of $f$.
Then $n$ is the maximum of $l_\gamma$ for such geodesics:

Let $f\colon{}D^*\to H^3$ be a regular WCF-end of a flat front
and take a (parametrized) geodesic $\gamma(s)$ 
($s\in \R$) in $H^3$ whose endpoint
\[
  \gamma(+\infty):=\lim_{s\to +\infty}\gamma(s)\in  \partial H^3
          =\C\cup\{\infty\}
\]
coincides with $G(0)=G_*(0)\in\partial H^3$.
Here, we identify $\partial H^3$ with $\C\cup\{\infty\}$ as in
\eqref{eq:asymptotic} in the appendix.
Then there exists a rigid motion $\iota_u$ in $H^3$ ($u\in \SL(2,\C)$)
such that 
\begin{equation}\label{eq:normal-axis}
   \iota_u\circ\gamma(-\infty)=\infty,\qquad
   \iota_u\circ\gamma(+\infty)=u\star G(0)=0.
\end{equation}
Note that such a $u\in\SL(2,\C)$ is unique up to the change
\begin{equation}\label{eq:normalize-amb}
   u \longmapsto \begin{pmatrix}
		  \delta & 0 \\
		    0 & \delta^{-1}
		  \end{pmatrix} u
		  \qquad (\delta\in\C\setminus\{0\}).
\end{equation}
We may assume that the pair $(G,\omega)$ is a dominant pair,
and define a meromorphic function 
\begin{equation}\label{eq:ratio}
   A_\gamma:=\frac{d(u\star G_*)}{d(u\star G)},
\end{equation}
which does not depend on the ambiguity as in \eqref{eq:normalize-amb},
that is, depends only on $\gamma$.
When $A_\gamma(z)$ is not a constant function,
the ramification order $l(=l_\gamma)$ of the
meromorphic function $A_\gamma(z)$ at $z=0$
is called the {\it indentation number} with respect to
the geodesic $\gamma$. 
On the other hand, if $A_\gamma(z)$ is constant, we set 
\begin{equation}\label{eq:l-gamma-rot}
  l_\gamma:=\infty.
\end{equation}
This exceptional case corresponds exactly to
the case that $f$ is rotationally symmetric with respect to the
axis $\gamma$ (see Example~\ref{ex:revolution}).

Since $(G,\omega)$ is dominant, the ramification number of $G$ at $0$
is equal to the multiplicity $m$ of the end.
Then, replacing $f$ by $\iota_u\circ f$
for a suitable $u \in \SL(2,\C)$,
there exists a complex coordinate $z$ on a neighborhood of the origin
such that
\begin{equation}\label{eq:normal-gauss-coord}
  G(z) = z^m.
\end{equation}
Unless $A_\gamma(z)$ is constant, the other hyperbolic Gauss map
$G_*(z)$ can be written as 
\begin{equation}\label{eq:normal-gauss-coord2}
   G_*(z)=\alpha z^m+\alpha_1 z^{m+l}+o(z^{m+l}),
       \qquad (\alpha_1\ne 0,~ l=l_\gamma), 
\end{equation}
where $\alpha$ is the ratio of the Gauss maps
(see Remark~\ref{rem:dominant-pair}).

We denote by $[\gamma]$ the image of the geodesic $\gamma$.
\begin{lemma}\label{lem:indentation}
 Let $f\colon{}D^*\to H^3$ be a regular WCF-end of a flat front
 with multiplicity $m$.
 Then one of the following three cases 
occurs{\rm:}
 \begin{enumerate}
 \item\label{item:indent:0} 
      The end $f$ is horospherical, and the indentation number
      $l_{\gamma}$ does not depend on the choice of geodesic
      $\gamma$.
 \item\label{item:indent:1}
      The end $f$ is not horospherical, and
      the indentation number 
      $l_{\gamma}$ does not depend on the choice of geodesic
      $\gamma$.  Moreover, $l_{\gamma}<m$ holds.
  \item\label{item:indent:2} 
       The end $f$ is not horospherical, and 
       there exists a unique geodesic $\sigma$ 
       satisfying $\sigma(+\infty)=G(0)$
       such that 
       \[
         l_\gamma \begin{cases} 
		   =m \qquad& 
		   ([\gamma]\ne [\sigma]), \\
		   >m \qquad& 
		   ([\gamma]  = [\sigma]) 
		\end{cases}
       \]
       holds for each geodesic $\gamma$ satisfying
 $\gamma(+\infty)=G(0)$.
 \end{enumerate}
\end{lemma}
\begin{proof}
 Without loss of generality, we may assume that $G(0)=G_*(0)=0$
 by replacing $(G,G_*)$ with $(u\star G,u\star G_*)$ ($u\in \SL(2,\C)$)
 if necessary.
 We fix a geodesic $\gamma_0$ such that
 \[
   \gamma_0(+\infty)=0,\qquad \gamma_0(-\infty)=\infty
 \]
 and let $l:=l_{\gamma_0}$ be the indentation number with respect to
 $\gamma_0$.

 Firstly, we consider the case $l_{\gamma_0}<\infty$.
 Then we may assume that $G(z)$ and $G_*(z)$
 satisfy \eqref{eq:normal-gauss-coord} and
 \eqref{eq:normal-gauss-coord2} respectively.
 We now take another geodesic $\gamma$
 such that
 \begin{equation}\label{eq:sigma}
  \gamma(+\infty)=0,\qquad 
   (a:=)\gamma(-\infty)\in \C\setminus \{0\},
 \end{equation}
 and set
 \begin{equation}\label{eq:u}
  u\bigl(=(u_{ij})\bigr):=
        \begin{pmatrix}
	   1 & 0 \\
	  -a^{-1} & 1
	\end{pmatrix}.
 \end{equation}
 Then 
 \[
    \iota_u\circ\gamma(+\infty)=0\qquad\text{and}\qquad
    \iota_u\circ\gamma(-\infty)=\infty
 \]
 hold.
 Here, 
 \begin{equation}\label{eq:transf}
   A_\gamma(z)=\frac{d(u\star G_*)}{d(u\star G)}=
    \left (
          \frac{u_{21}G+u_{22}}{%
	        u_{21}G_*+u_{22}}\right)^2 
            \frac{dG_*}{dG}
        =
        \left (
	   \frac{-a^{-1}G\hphantom{_*}+1}{-a^{-1}G_*+1}
               \right)^2 \frac{dG_*}{dG}.
 \end{equation}
 By \eqref{eq:normal-gauss-coord} and
 \eqref{eq:normal-gauss-coord2}, we get
 \[
   A_{\gamma}(z)=
   \left (
        1-\frac{2}{a}(1-\alpha)z^m+o(z^m)
   \right)
   \left (
        \alpha+\frac{(m+l)\alpha_1}{m} z^l+o(z^l)
   \right).
 \]
 If $f$ is horospherical, that is, $\alpha=0$, 
 then the first non-vanishing term of  $A_\gamma(z)$ is $z^l$, 
 and we have $l_\gamma=l$.
 This implies that the indentation number 
 $l_\gamma$ does not depend on $\gamma$, 
 that is, \ref{item:indent:0} holds.

 Next, we consider the case that $f$ is not horospherical,
 that is, $\alpha\ne 0$.
 \begin{description}
  \item[Case a]
     Suppose that $l<m$, then
     \[
        A_\gamma(z)=\alpha+\frac{(m+l)\alpha_1}{m} z^l+o(z^l),
     \]
     which implies $l_\gamma=l$.
     Therefore \ref{item:indent:1} holds.
  \item[Case b]
     Suppose that $l>m$, then we have
     \begin{equation}\label{eq:l>m}
        A_\gamma(z)=\alpha\left (
	    1-\frac{2}{a}(1-\alpha)z^m+o(z^m)
                \right).
     \end{equation}
     Since $\alpha\ne 1$, we have $l_\gamma=m$.
     Thus we have
     \[
        l_{\gamma_0}=l>m=l_{\gamma}\qquad ([\gamma]\ne [\gamma_0]).
     \] 
     This is case \ref{item:indent:2}.
  \item[Case c]
Suppose that $l=m$.
     Then it holds that
     \[
       A_\gamma(z)=\alpha 
          +\left(2\alpha_1
          -\frac{2}{a}(1-\alpha)\alpha\right)z^m
          +o(z^m).
     \]
     Then $l_\gamma=m$ holds, unless
     \[
        a=\gamma(-\infty) = 
	     \frac{(1-\alpha)\alpha}{\alpha_1}.
     \]
     This exceptional value $a$ determines the image $[\gamma]$
     of the geodesic $\gamma$ uniquely. 
     Denoting this geodesic $\gamma$ by $\sigma$, 
     we fall into case  \ref{item:indent:2}.
 \end{description}

 Finally, we consider the case $l_{\gamma_0}=\infty$,
 which implies that $dG_*/dG$ is constant.
 If $f$ is horospherical, 
 it is an $m$-fold cover of the end of the horosphere 
 (see Example~\ref{ex:revolution}).
 Then $G_*$ vanishes
 identically, and $d(u\star G_*)/d(u\star G)$ vanishes for all
 $u\in \SL(2,\C)$. This implies \ref{item:indent:0}. 
 On the other hand, if $f$ is not horospherical,
 that is, $\alpha\ne 0$,
 then we have an expression
 \[
   G(z)=z^m,\qquad G_*(z)=\alpha z^m,
 \]
 which gives an $m$-fold cover of a flat front of revolution
 whose rotational axis
 is $\gamma_0$ (see Example~\ref{ex:revolution}).
 Take a geodesic $\gamma$ $([\gamma]\ne [\gamma_0])$
 and a matrix $u\in \SL(2,\C)$ 
 satisfying \eqref{eq:sigma} and \eqref{eq:u}.
 Then by \eqref{eq:transf}, we have
 \eqref{eq:l>m}, since 
 $dG_*/dG=\alpha$.
 This implies $l_\gamma=m$.
 Therefore \ref{item:indent:2} holds.
\end{proof}

\begin{definition}\label{def:central-axis}
 Let $f$ be as above.
 According to Lemma~\ref{lem:indentation},
 we call $f$ a {\em centerless end\/},
 if it satisfies \ref{item:indent:0} or \ref{item:indent:1}.
 On the other hand,
 $f$ is called a {\em centered end} if
 it satisfies \ref{item:indent:2}.
 When $f$ is a centered end, the unique geodesic $\sigma$
 is called the {\em principal axis\/} of the end $f$.
 The maximum  of the numbers $l_\gamma$, that is,
 \[
    n:=\max\{l_\gamma\,;\, 
   \mbox{$\gamma$ is a geodesic such that $\gamma(+\infty)=G(0)$}\}
 \]
 is called the {\em maximum indentation number\/} of the end $f$.
 If $f$ is a centered end,
 then $n=l_{\sigma}$ holds for the principal axis $\sigma$.
\end{definition}
\begin{remark}\label{rem:central}
Suppose that a regular WCF-end $f$ is non-horospherical 
and embedded. The embeddedness implies that $f$ is 
complete  
(i.e., free of singularities) 
and has multiplicity $m=1$. 
Hence only the case \ref{item:indent:2} in Lemma~\ref{lem:indentation}
occurs for $f$. 
In other words, an embedded, non-horospherical, regular 
WCF-end must be centered. 
\end{remark}
\subsection*{The flux axis and the principal axis}

Definition~\ref{def:central-axis} of the principal axis
for a centered end looks rather technical.
However, it is nothing but the axis of the flux matrix.
In this subsection, we assume that 
{\em the end is not horospherical}.
In fact, for a horospherical end, the only eigenvalue of the flux matrix
$\Phi_f$ is $0$ and hence the axis of the flux matrix cannot be defined.
\begin{theorem}\label{thm:central-flux}
 Let $f\colon{}D^*\to H^3$ be a centered regular WCF-end of a flat front
 in the sense of Definition~\ref{def:central-axis}.
 Then the principal axis coincides with the 
 axis of the flux matrix $\Phi_f$.
\end{theorem}

\begin{proof}
 Without loss of generality, we may assume that
 $(G,\omega)$ is a dominant pair,
 $G(0)=0$ and  the principal axis $\sigma$
 is the geodesic joining $\infty$ and $0$.
 Then by Lemma~\ref{lem:flux-diagonal},
 the axis of the flux matrix $\Phi_f$ 
 coincides with $\sigma$ if and only if $\Phi_f$ is diagonal.
 Moreover, by Theorem~\ref{thm:limit}, 
 $\Phi_f$ is a lower triangular matrix because $G(0)=0$.
 Then it is sufficient to show the lower-left component
 of the flux matrix vanishes.

 If the end is rotationally symmetric, 
 the assertion is obvious, since 
 the axis of the flux matrix is the rotation axis.
 So we assume that the end is not rotationally symmetric.
 In this case, 
 we may assume
 \eqref{eq:normal-gauss-coord},
 \eqref{eq:normal-gauss-coord2},
 and
 \[
    \alpha\neq 0,1,\qquad m<l_{\sigma}=l<\infty,
 \]
 where $l$ is the maximum indentation number.
 (Here, $\alpha\neq 0$ because the end is not horospherical,
 $\alpha\neq 1$ because the end is of finite type,
 and $l_{\sigma}>m$ because the end is centered.)

 Substituting these into \eqref{eq:two-equalities-in-one-line},
 the lower-left component of $(\partial f)f^{-1}$ is computed as
 \begin{multline*}
  -\frac{dG+dG_*}{(G-G_*)^2} =
      -\frac{m(1+\alpha)z^{m-1}+(m+l)\alpha_1z^{m+l-1}+o(z^{m+l-1})}{%
            \bigl((1-\alpha)z^m-\alpha_1z^{m+l}+o(z^{m+l})\bigr)^2}
            \,dz\\
      =
      -\frac{m(1+\alpha)}{(1-\alpha)^2}z^{-m-1}
      \left[1+
        \alpha_1\left(\frac{2}{1-\alpha}+\frac{m+l}{m(1+\alpha)}
  \right)
        z^l + o (z^l)\right]\,dz.
 \end{multline*}
 Since $l>m(\ge 1)$, the residue of this form at the origin vanishes.
 This completes the proof.
\end{proof}

\subsection*{Normalization of dominant pairs}
To prove the main theorems, we introduce the 
normalized form of the canonical form $\omega$:
\begin{lemma}\label{lem:normalized-omega}
 Let $f\colon{}D^*\to H^3$ be a regular WCF-end of a flat front
 with $G(0)=0$,
 and $l=l_{\gamma}$ the indentation number with 
 respect to the geodesic $\gamma$ joining  $\infty$ and $0$.
 Assume $l_{\gamma}<\infty$, that is, $f$ is not rotationally symmetric
 with respect to $\gamma$ {\rm(}see \eqref{eq:l-gamma-rot}{\rm)}.
 Suppose that $(G,\omega)$ is a dominant pair
 {\rm (}see Remark \ref{rem:dominant-pair}{\rm)}.
 Then there exist a {\rm (}unique{\rm )} local complex coordinate $z$
 around the origin and a diagonal matrix $u\in \SL(2,\C)$
 such that
\begingroup
 \renewcommand{\theenumi}{{\rm(\roman{enumi})}}
 \renewcommand{\labelenumi}{{\rm(\roman{enumi})}}
 \begin{enumerate}
  \item\label{item:can-omega-non-cyl}
        if the end is not cylindrical, 
	\begin{equation}\label{eq:norm-omega-non-cyl}
	   u\star G=z^m\quad\text{and}\quad
	   \omega = -m z^\mu \bigl(1+bz^l+o(z^l)\bigr)^2\,dz\qquad
	   (\mu < -1,~b\in \R_+),
	\end{equation}
  \item\label{item:can-omega-cyl}
        if the end is cylindrical, 
	\begin{equation}\label{eq:norm-omega-cyl}
	    u\star G=z^m\quad\text{and}\quad
           \omega = -\lambda z^{-1} \bigl(1+z^l+o(z^l)\bigr)^2\,dz\qquad
	   (\lambda \in\R_+),
	\end{equation}
 \end{enumerate}
 \endgroup
\noindent
 where $m$ is the multiplicity of the end;
 moreover, the end is incomplete if and only if $\lambda={m}/{2}$.
\end{lemma}

\begin{proof}[Proof of Lemma~\ref{lem:normalized-omega}]
 We can take a coordinate $z$ such that
 $G$ and $G_*$ are written as \eqref{eq:normal-gauss-coord}
 and \eqref{eq:normal-gauss-coord2}, respectively.
 Substituting these into \eqref{eq:gauss-repr} and \eqref{eq:can-hopf},
 a direct calculation verifies that $\omega$ has the following
 expression:
 \begin{equation}\label{eq:omega-first}
   \omega  = c_0 z^{\mu} \left(1+b_0 z^l +o(z^l)\right)^2\,dz
 \end{equation} 
 where $c_0$ and $b_0$ are non-zero constants which can be
 computed explicitly, and 
 $\mu$ is given by the relation \eqref{eq:alpha-mu} in
 Proposition~\ref{prop:alpha-mu-new}.
 Moreover, by Proposition~\ref{prop:alpha-mu-new}, $\mu\leq -1$
 holds.
 Note that $b_0\neq 0$ because $f$ is not rotationally symmetric
 with respect to $\gamma$.

\paragraph{\ref{item:can-omega-non-cyl}}
 First, we assume that the end is non-cylindrical,
 that is, $\mu<-1$ holds.
 Let $k$ be a non-zero constant and take a new coordinate $w$
 as $z=kw$.
 Then $\omega$ is written as
 \[
    \omega = c_0 k^{\mu+1}w^{\mu}
         \bigl(1+k^lb_0w^l+o(w^l)\bigr)^2\,dw.
 \]
 Here, choose $k$ so that
 \[
   k = \kappa e^{\imag \beta},\qquad\text{where}\quad
    \kappa=\left|\frac{m}{c_0}\right|^{\frac{1}{\mu+1}},\quad
    \beta = -\frac{\arg b_0}{l}.
 \]
 Then we have
 \begin{multline*}\label{eq:omega-pre}
    \omega = -m e^{\imag\tau}
              w^{\mu}
              \bigl(
                1+b w^l+o(w^l)
              \bigr)^2\,dw,\\
    \text{where}\quad
     b = \kappa^l|b_0|\in\R_+,\quad
     \tau = \arg c_0+\beta(\mu+1)+\pi.
 \end{multline*}
 Using the $\U(1)$-ambiguity as in \eqref{eq:u-one-amb}, we 
 can write $\omega$ as
 \[
    \omega = -m w^{\mu}\bigl(1+b w^l+o(w^l)\bigr)^2 \,dw
 \]
 On the other hand,
 $G$ is written as $G=z^m = k^m w^m$.
 Let
 \begin{equation}\label{eq:u-adjust}
      u = \begin{pmatrix}
	     k^{-m/2} & 0 \\
	     0 & k^{m/2}
	  \end{pmatrix}\in\SL(2,\C).
 \end{equation}
 Then $u\star G = w^m$ and 
 the isometry $\iota_u$ of $H^3$ preserves $0$ and $\infty$
 in $\partial H^3$.
Finally, we need only to recall that $\omega$ 
is unchanged when $f$ changes to $\iota_u \circ f$. 
 
\paragraph{\ref{item:can-omega-cyl}}
 If the end is cylindrical, $\mu=-1$ holds.
 Take a coordinate $w$ as $z=kw$ for a non-zero constant $k$.
 Then $\omega$ in \eqref{eq:omega-first} is written as
 \[
   \omega = c_0 w^{-1}\bigl(1+k^l b_0 w^l+o(w^l)\bigr)^2\,dw.
 \]
 In this case, if we set $k=b_0^{-1/l}$, we have
 \[
   \omega = -\lambda e^{\imag\tau} w^{-1}
            \bigl(1+w^l+o(w^l)\bigr)^2\,dw
    \qquad\bigl(\lambda=|c_0|,\tau=\arg c_0+\pi).
 \]
 Then again using the $\U(1)$-ambiguity and the matrix
 $u$ in \eqref{eq:u-adjust}, we have 
 \eqref{eq:norm-omega-cyl}.

 Now, we shall prove the last assertion:
 By \eqref{eq:Q-top} and \eqref{eq:norm-omega-cyl},
 $\rho$ is expanded as
 \[
    \rho = \frac{\theta}{\omega} =
           \frac{\theta\omega}{\omega^2} 
         = \frac{Q}{\omega^2}
         = \frac{m^2}{4\lambda^2}\bigl(1+o(1)\bigr).
 \]
 Since $\lambda\in\R_+$, 
 Lemma~\ref{lem:incomplete} implies that the end is incomplete if and
 only if $\lambda=m/2$.
\end{proof}
\section{Asymptotic behavior of regular WCF-ends}
\label{sec:main}
In this section, we shall state and prove refinements of 
Proposition~\ref{fact:complete-asymptotic}
in the introduction.
We also give a proof of Theorem~\ref{thm:incomplete} here.

\subsection*{Notations---A family of horospheres}
To state the results, we prepare notations:
Let $\gamma$ be an oriented geodesic in $H^3$ and fix a point
$o\in\gamma$.
Parametrize $\gamma$ by the arclength parameter $s$ such that 
$\gamma(0)=o$.
For each $s\in\R$,
we denote by $\Horo_{\gamma}(o,s)$
the horosphere which meets $\partial H^3$ at $\gamma(-\infty)$
and intersects  $\gamma$ perpendicularly at 
the point $\gamma(s)$,
see Figure~\ref{fig:horospheres}, left.
Since each $\Horo_{\gamma}(o,s)$ is isometric to the Euclidean plane,
one can introduce a canonical coordinate system
on $\Horo_{\gamma}(o,s)$ such that $\gamma(s)$ corresponds to the origin
of the Euclidean plane as follows (see \eqref{eq:horospheres-isometry}):
We wish to work in the upper half-space model $\R^3_+$ of 
$H^3$ as in \eqref{eq:proj0} and \eqref{eq:upper-metric}.
We always project $H^3$ to $\R^3_+$ so that $\gamma$ is mapped to the 
downward oriented $h$-axis (vertical axis) and $o$ is mapped to 
$(0,1)\in \C\times \R_+=\R^3_+$.
Then 
\begin{equation}\label{eq:horospheres-canonical}
 \Horo_{\gamma}(o,s) =
  \{(\zeta,e^{-s})\,|\,\zeta\in\C\}\subset \C\times \R_+
        =\R^3_+.
\end{equation}
In this case,
the isometry between $\Horo_{\gamma}(o,s)\subset H^3$
and the Euclidean plane
is given by 
\begin{equation}\label{eq:horospheres-isometry}
   \hat\pi:\Horo_{\gamma}(o,s)\ni (\zeta,h) = (\zeta,e^{-s})
       \longmapsto \frac{\zeta}{h}=e^s\zeta\in\C=\Euc^2,
\end{equation}
because of \eqref{eq:upper-metric},
where $\Euc^2$ denotes the Euclidean plane 
$(\R^2;x,y)$ with the canonical metric $dx^2+dy^2$.
See Figure~\ref{fig:horospheres}, right.
\begin{figure}[t]
 \begin{tabular}{c@{\hspace{2em}}c}
  \includegraphics{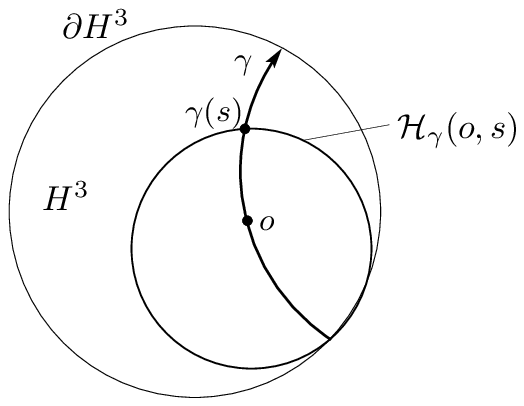} &
  \includegraphics{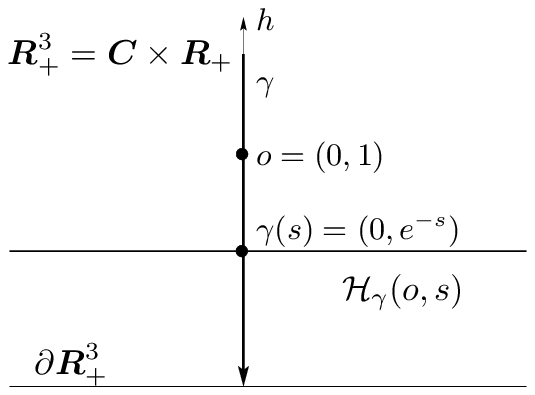} 
 \end{tabular}
\caption{A family of horospheres}\label{fig:horospheres}
\end{figure}
\subsection*{Statements of the theorems}
First, we consider the non-cylindrical case.
In this case, all ends are complete by Lemma~\ref{lem:incomplete}.
\begin{theorem}[Non-cylindrical case]\label{thm:non-cyl}
 Let $f\colon{}D^*\to H^3$ be a non-cylindrical regular WCF-end
 of a flat front with multiplicity $m$, and
 let $\gamma$ be a geodesic in $H^3$ with
 $\gamma(+\infty)=G(0)\in\partial H^3$,
 where $G$ is the hyperbolic Gauss map.
 Then there exists a unique point $o\in \gamma$
 such that, for $s$ large enough, 
$f(D^*)\cap\Horo_{\gamma}(o,s)$ is 
 identified with a curve in $\C=\Euc^2$
 parametrized by $t$ as 
 \begin{equation}\label{eq:non-cyl-0}
      e^{\imag m t}h^p\bigl(1+R_{\gamma}(h,t)\bigr)\qquad
     \left(p = -\frac{1+\alpha}{2}\in (-1,0)\right)
 \end{equation}
 using $\hat\pi$ in \eqref{eq:horospheres-isometry},
 where $\alpha$ is the ratio of the Gauss maps \eqref{eq:gauss-ratio},
 $h=e^{-s}$, and $R_{\gamma}$ is a complex-valued function of two real
 variables $(h,t)$ such that
 \[
     \lim_{h\to +0}R_{\gamma}(h,t) =0.
 \]
 Moreover,
 assume the indentation number $l_{\gamma}$ with respect to $\gamma$
 is finite {\rm(}see \eqref{eq:l-gamma-rot}{\rm)}.
 Then, setting
 \begin{equation}\label{eq:non-cyl-data}
  \begin{aligned}
	  N_{p,j,m}(h,t)&=\{2(p+1)\cos jt\}h^{\beta},\quad
         \text{where}\quad
            \beta =\beta_{p,j,m}= \frac{j(1+p)}{m}>0,\\
          S_p(h) &= -\frac{1}{4(p+1)}h^{-2p}
  \end{aligned}
 \end{equation}
 for positive integers $j$, $m$ and a real number $p\in (-1,0)$,
 we have
 \begingroup
 \renewcommand{\theenumi}{{\rm (\roman{enumi})}}
 \renewcommand{\labelenumi}{{\rm (\roman{enumi})}}
 \begin{enumerate}
  \item\label{item:non-cyl-1}
       If $f$ is a centerless end, then
       for any geodesic $\gamma$ with $\gamma(+\infty)=G(0)$,
       there exists  $b\in\R_+$ 
       {\rm(}when $b$ is used{\rm )}
       such that
	\begin{equation}\label{eq:non-cyl-asymptotic}
	   R_{\gamma}(h,t)=
          \begin{cases}
	     b N_{p,n,m}(h,t) + o(h^{\beta})\qquad
                 &(\text{if $n(1+p)<-2pm$}),\\
	     bN_{p,n,m}(h,t)+S_p(h)+o(h^{\beta}) \qquad
                 &(\text{if $n(1+p)=-2pm$}),\\
	     S_p(h) + o(h^{-2p})
                 &(\text{if $n(1+p)>-2pm$}),
	   \end{cases}
	\end{equation}
       where $n$ is the maximum indentation number, 
       and $\beta=\beta_{p,n,m}$ as in \eqref{eq:non-cyl-data}.
   \item\label{item:non-cyl-2}
       If $f$ is a centered end, then
	for the principal axis $\sigma$,
       there exists  $b\in\R_+$ 
       such that
       $R_{\sigma}(h,t)$ is written as in \eqref{eq:non-cyl-asymptotic},
       where $n(>m)$ is the maximum indentation number.
 \end{enumerate}
\endgroup
\end{theorem}
Figure~\ref{fig:nodes} shows the end with $G=z$,
$\omega=-z^{-5}\exp(2z^3)\,dz$ in 
the upper half-space model (left) and the view of it from the
bottom (right).
In each of two figures,
the end point corresponds to the center of the figure.
\begin{figure}
  \begin{center}
   \includegraphics[width=0.6\textwidth]{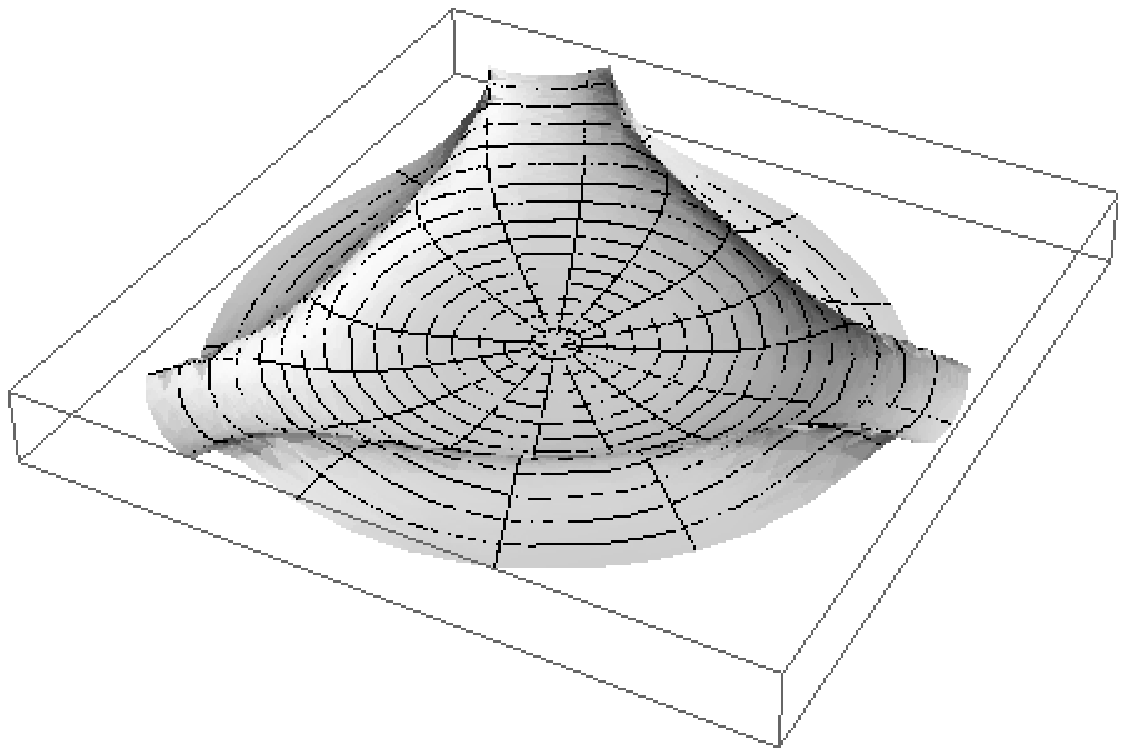}  
   \includegraphics[width=0.36\textwidth]{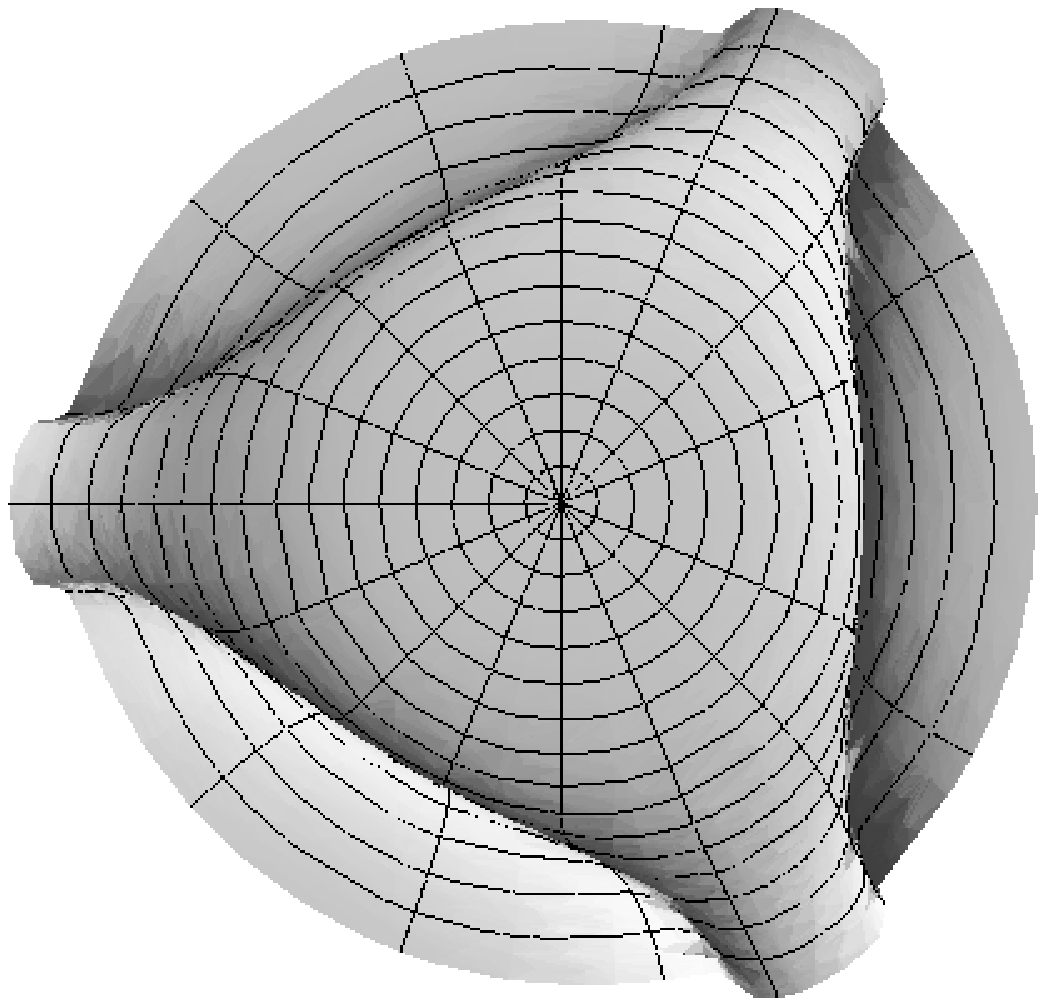} 
  \end{center}
  \caption{%
    A centered end with $m=1$, $n=3$ and $p=-0.8$.}
  \label{fig:nodes}
\end{figure}
\begin{remark}\label{rem:non-principal}
 For a centered end $f$ as in Theorem~\ref{thm:non-cyl},
 take a geodesic $\gamma$ with $\gamma(+\infty)=G(0)$
 which does not coincide with the principal axis.
 Then there exists a unique $b_{\gamma}\in\R_+$
	{\rm (}when $b_{\gamma}$ is used{\rm)}
	such that 
	\begin{equation}\label{eq:non-cyl-asymp-3}
	 R_{\gamma}(h,t)=
	  \begin{cases}
	   b_{\gamma}N_{p,m,m}(h,t) +o(h^{1+p})\qquad & 
	   \left(\text{if $-1<p<-\dfrac{1}{3}$}\right),\\[6pt]
	   b_{\gamma}N_{-\frac{1}{3},m,m}(h,t)
	   +S_{-\frac{1}{3}}(h)+o(h^{2/3}) &
	   \left(\text{if $p=-\dfrac{1}{3}$}\right),\\[6pt]
	   S_p(h) +o(h^{-2p}) &
	   \left(\text{if $-\dfrac{1}{3}<p<0$}\right).
	  \end{cases}
	\end{equation}
 This formula
 can be proved similarly to the proof of 
 Theorem~\ref{thm:non-cyl}. 
 Although a geometric interpretation of the bifurcation at $p=-1/3$
 will not be provided here,
 we remark that this bifurcation again appears, also at $-1/3$, even
 if we instead were to use the height functions induced by slicing with
 the family of horospheres which meet the end at infinity.
\end{remark}
\begin{remark}
 When $m=1$, that is, the case of an embedded end,
 only the case \ref{item:non-cyl-2} occurs.
\end{remark}
 
\begin{remark}
 If $f$ is rotationally symmetric, $n=\infty$ by definition
 and the end is of centered type.
 In this case, $n=\infty$ implies $n(1+p)>-2pm$ occurs, and then
 the third equation of \eqref{eq:non-cyl-asymptotic} holds
 for the principal axis
 and \eqref{eq:non-cyl-asymp-3} holds for non-principal axes.
\end{remark}

\begin{theorem}[Cylindrical case]\label{thm:cyl}
 Let $f\colon{}D^*\to H^3$ be a cylindrical regular WCF-end
 of a flat front
 with multiplicity $m$, and let $n$ be 
 its maximum indentation number.
 Take a geodesic $\gamma$ in $H^3$ with $\gamma(+\infty)=G(0)\in\partial H^3$,
 where $G$ is the hyperbolic Gauss map.
 We set
 \begin{equation}\label{eq:C}
 \begin{aligned}
    V_{m,l,c}(t) &=
     \left(
       \frac{4c^2+m^2}{4cm}
     \right)^{l/m}
     \left[
      2\left(c+\frac{m^2}{4c}\right)\cos lt - 
       \imag \frac{ml}{c}\sin lt
     \right],\\
    \Cyc_{m,l}(t) &= \frac{e^{\imag m t}}{m}
                        V_{m,l,\frac{m}{2}}(t)
             = \frac{1}{m}
                \bigl[(m+l)e^{\imag(m-l)t}+
                      (m-l)e^{\imag(m+l)t}\bigr].
 \end{aligned}
 \end{equation}
 We remark that $\Cyc_{m,l}(t)$ represents a cycloid 
 as in Theorem~\ref{thm:incomplete} in the introduction.
 Then
 \begingroup
 \renewcommand{\theenumi}{{\rm (\roman{enumi})}}
 \renewcommand{\labelenumi}{{\rm (\roman{enumi})}}
 \begin{enumerate}
  \item\label{item:cyl-1}
       If $f$ is a centerless end, that is, $n<m$, then 
       for any geodesic $\gamma$ with $\gamma(+\infty)=G(0)$,
       there exist  $\lambda\in\R_+$ and 
       a unique point $o\in \gamma$ 
       such that, 
       for $s$ large enough, 
       $\hat\pi\bigl(f(D^*)\cap\Horo_{\gamma}(o,s)\bigr)$ 
       is a curve in $\C=\Euc^2$
       parametrized by $t$ as
       \begin{multline}\label{eq:asymp-cyl}
	\frac{1}{m}e^{\imag m t}
	 \left[
	  \left(\lambda-\frac{m^2}{4\lambda}\right)
	  +
	  V_{m,n,\lambda}(t)h^{\beta}
	    +o(h^\beta)
	 \right],\\
       \text{where}\quad
         \beta = \frac{n}{m}, \quad \text{and}
         \quad
         h = e^{-s}.
       \end{multline}
       In particular, if $f$ is incomplete, then $\lambda=m/2$ and 
       $\hat\pi\bigl(f(D^*)\cap\Horo_{\gamma}(o,s)\bigr)$ is
       parametrized as
       \begin{equation}\label{eq:asymp-cyl-incomp}
          \Cyc_{m,n}(t) h^{\beta}+o(h^{\beta}),
          \qquad \text{where}\quad\beta=\frac{n}{m}.
       \end{equation}
  \item\label{item:cyl-2}
       If $f$ is a centered end
       which is not rotationally symmetric, that is, 
       $m<n<+\infty$, 
       \eqref{eq:asymp-cyl} 
       or \eqref{eq:asymp-cyl-incomp}
       holds when $\gamma$ coincides with  the principal axis $\sigma$. 
 \end{enumerate}
\endgroup
\end{theorem}
\begin{remark}\label{rem:cyl-non-principal}
 For a centered end $f$ as in Theorem~\ref{thm:cyl},
 take a geodesic $\gamma$ with $\gamma(+\infty)=G(0)$
 which does not coincide with the principal axis.
 Then there exists a unique point $o\in\gamma$ such that
 the intersection $\hat \pi(f(D^*)\cap \Horo_{\gamma}(o,s))$
 is parametrized as
 \begin{equation}\label{eq:asymp-cyl-2}
  \frac{1}{m}e^{\imag m t}
   \left[
    \left(\lambda-\frac{m^2}{4\lambda}\right)
    +
    V_{m,m,\lambda}(t)h
    +o(h)
   \right]\qquad (h=e^{-s}).
 \end{equation}
 In particular, when $f$ is incomplete, we have
 $\lambda=m/2$. Then \eqref{eq:asymp-cyl-2} yields that
 $\hat\pi\bigl(f(D^*)\cap\Horo_{\gamma}(o,s)\bigr)$ is
 parametrized as $2 h + o (h)$.
\end{remark}
\begin{remark}\label{rem:basepoint}
 If one chooses another point $o'$ on the geodesic $\gamma$, 
 the asymptotic behavior changes as follows:
 Under the situations of Theorems~\ref{thm:non-cyl} or \ref{thm:cyl},
 take a point $o'$ on the geodesic $\gamma$ such that
 the signed distance between $o$ (which is uniquely determined)
 and $o'$ is $\tau$,
 that is, $\Horo_{\gamma}(o',s)=\Horo_{\gamma}(o,s+\tau)$.
 Then we have 
 \begin{itemize}
  \item When the end is non-cylindrical, 
	$\hat \pi(f(D^*)\cap \Horo_{\gamma}(o',s))$
	is parametrized as
	\[
	     e^{-\tau p} e^{\imag m t}h^p\bigl(1+o(1)\bigr)
	     \qquad 
	     (h=e^{-s}),
	\]
	where $o(1)$ denotes a higher order term in $h$.
  \item In the case of a complete cylindrical end,
	$\hat \pi(f(D^*)\cap \Horo_{\gamma}(o',s))$
	is pa\-ramet\-rized as
	\[
	    \frac{1}{m}e^{\imag m t}
	    \left(
	    \left(\lambda-\frac{m^2}{4\lambda}\right)
	    +
	    e^{-\tau\beta}V_{m,n,\lambda}(t)h^{\beta}
	    +o(h)
	    \right)\qquad (h=e^{-s}),
	\]
	under the assumption that \eqref{eq:asymp-cyl} holds.
  \item In the case of an incomplete cylindrical end,
	$\hat \pi(f(D^*)\cap \Horo_{\gamma}(o',s))$
	is pa\-ramet\-rized as
	\[
	      e^{-\tau\beta}\Cyc_{m,n}(t)h^{\beta}+o(h^\beta)\qquad 
	      (h=e^{-s})
	\]
	under the assumption that \eqref{eq:asymp-cyl-incomp} holds.
 \end{itemize}
\end{remark}
First, we prove Proposition~\ref{fact:complete-asymptotic}
and Theorem~\ref{thm:incomplete} in the introduction
as corollaries of Theorems~\ref{thm:non-cyl} and \ref{thm:cyl}.
After that their proofs are given.
\begin{proof}[Proof of Proposition~\ref{fact:complete-asymptotic}]
 Take a geodesic $\gamma$ with $\gamma(+\infty)=G(0)$.
 Then by taking the first terms of \eqref{eq:non-cyl-0} and
 \eqref{eq:asymp-cyl},
 $\hat \pi(f(D^*)\cap \Horo_{\gamma}(o,s))$ is parametrized as 
 \[
    \begin{cases}
      e^{\imag m t}h^p\bigl(1+o(1)\bigr)\qquad
        &\text{if the end is not cylindrical},\\[6pt]
      \dfrac{1}{m}\left(\lambda-\dfrac{m^2}{4\lambda}\right)
      e^{\imag m t}\bigl(1+o(1)\bigr)
        &\text{if the end is cylindrical}.
    \end{cases}
 \]    
 Hence by the correspondence \eqref{eq:horospheres-isometry},
 we have Proposition~\ref{fact:complete-asymptotic}.
\end{proof}
\begin{proof}[Proof of Theorem~\ref{thm:incomplete}]
 Take a geodesic $\gamma$ with $\gamma(+\infty)=G(0)$
 arbitrarily when the end is centerless.
 If the end is centered, we take $\gamma$
 to be the principal axis.
 Then by  \eqref{eq:asymp-cyl-incomp} and
 \eqref{eq:horospheres-isometry},
 we get \eqref{eq:asymp-incomplete} of
 Theorem~\ref{thm:incomplete}, where $n$ is the maximum indentation
 number.
 Finally, applying  Proposition~\ref{lem:cuspidal} we can
 conclude that
 $n$ is the ramification order of $\rho$,
 since the map $\Cyc_{m,n}(t)$ has $2n$ cusps in $[0,2\pi)$.
\end{proof}

\begin{corollary}\label{cor:maxind=ram}
Let $f\colon D^* \to H^3$ be an incomplete regular WCF-end. 
\begin{enumerate}
 \item \label{cor:maxind=ram-1}
The maximum indentation number of $f$ 
 coincides with the ramification order of $\rho$ at $z=0$.  
 \item $f$ is  epicycloid-type (or hypocycloid-type)
if and only if it is centerless (or centered). 
\end{enumerate}
\end{corollary}
{\it Proof.} 
 \begin{enumerate}
  \item We have already seen this 
in the proof of Theorem~\ref{thm:incomplete}. 
 \item Recall that the multiplicity $m$ of the end and 
the ramification order $n$ of $\rho$ determine 
which type the end $f$ is,  
in such a way that it is epicycloid-type if $n <m$ and 
 is hypocycloid-type if $n > m$.  
Since $n$ equals the indentation number, 
the condition $n<m$ or $n>m$ implies that the end $f$ 
is centerless or centered, respectively.   
\hfill \qed
 \end{enumerate}

\begin{remark}\label{rmk:referee}
Here we provide another proof of 
Corollary \ref{cor:maxind=ram} \ref{cor:maxind=ram-1}:  
 Take a geodesic $\gamma$ with $\gamma(+\infty)=G(0)$
 arbitrarily when the end is centerless, and
 to be principal when the end is centered. 
Then the indentation number $l_{\gamma}$ equals 
the maximum indentation number $n$, which is not equal to 
the multiplicity $m$ of the end 
(see Lemma \ref{lem:indentation} (2), (3)); 
\begin{equation}\label{eq:l=nnot=m}
 l_{\gamma} = n \ne m. 
\end{equation}  
On the other hand, the formula \eqref{eq:formula-drhorho} 
in the introduction 
with expansions \eqref{eq:normal-gauss-coord},  
\eqref{eq:normal-gauss-coord2} gives 
\begin{equation}\label{eq:drhooverrho}
 \frac{d \rho}{\rho} = \left\{
\frac{\alpha_1}{m}(m+l)(m-l)z^{l-1}+o(z^{l-1})
\right\}dz, \quad (l = l_{\gamma}). 
\end{equation}
It follows from 
\eqref{eq:l=nnot=m}, \eqref{eq:drhooverrho} that 
the ramification order of $\rho$ equals 
$n$, the maximum indentation number. 
\end{remark}

\subsection*{Proofs of Theorems \ref{thm:non-cyl} and \ref{thm:cyl}}
\begin{proof}[Proof of Theorem~\ref{thm:non-cyl}]
 By Lemma~\ref{lem:normalized-omega},
 there exist a unique isometry $u\in\SL(2,\C)$
 and a complex coordinate $z$ 
 such that \eqref{eq:norm-omega-non-cyl} holds.
 Then $\iota_u\circ\gamma$ is the geodesic joining $\infty$ and $0$,
 that is, the $h$-axis of the upper half-space model $\R^3_+$ in 
 \eqref{eq:proj0},
 and we write $G$ instead of $u\star G$.

 Replacing $f$ by $\iota_u\circ f$, we assume $\gamma$ itself
 is the geodesic joining $\infty$ and $0$.

 We set $o=(0,1)$ in $\R^3_+$,
 which is a point on the geodesic $\gamma$.
 In this case, the family of horospheres
 is written as in \eqref{eq:horospheres-canonical}.
 So, by \eqref{eq:horospheres-isometry},
 if the image of $\pi\circ f$ is parametrized as
 $\bigl(\zeta(h,t),h\bigr)$, 
 the image under $\hat\pi$ of the intersection 
 $f(D^*)\cap \Horo_{\gamma}(o,s)$ is parametrized
 as  $t\mapsto\zeta(h,t) / h \in \C=\Euc^2$, where $h=e^{-s}$.
 Thus, to prove the theorem, it is sufficient 
 to compute $\zeta/h$ in terms of $h$.

 According to \eqref{eq:proj-mat},
 we have $\zeta/h=E_{11}\overline{E}_{21}+E_{12}\overline{E}_{22}$ and
 $1/h=E_{21}\overline{E}_{21}+E_{22}\overline{E}_{22}$.
 Hence we start with calculations of $E_{ij}$.
 Substituting \eqref{eq:norm-omega-non-cyl} into
 \eqref{eq:g-omega-repr},
 we obtain
 \begin{align*}
  E_{11}&=  -z^{(1+\mu+m)/2}\bigl(1+bz^l+o(z^l)\bigr),\\
  E_{12}&=  \frac{1}{2m}z^{-(1+\mu-m)/2}
              \left[
            (1+\mu+m)+b\bigl(2l-(1+\mu+m)\bigr)z^l+o(z^l)
              \right],\\
  E_{21}&=  -z^{(1+\mu-m)/2}\bigl(1+bz^l+o(z^l)\bigr),\\
  E_{22}&=  \frac{1}{2m}z^{-(1+\mu+m)/2}
              \left[(1+\mu-m)+b\bigl(2l-(1+\mu-m)\bigr)z^{l}+o(z^l)
              \right].
 \end{align*}
 Thus, 
 \begin{equation}\label{eq:ccdd}
 \begin{aligned}
  E_{21}\overline{E_{21}} &= 
    r^{\frac{-m}{1+p}}\bigl(1+2br^l \cos l t+o(r^l)\bigr),\\
  E_{22}\overline{E_{22}} &= \frac{1}{4}r^{\frac{-m}{1+p}}
              \left(
                \frac{1}{(1+p)^2}r^{\frac{-2mp}{1+p}}+
                   o(r^{\frac{-2mp}{1+p}})
              \right),\\
  E_{11}\overline {E_{21}}&=
           e^{\imag m t}r^{\frac{mp}{1+p}}
           \bigl(1 + 2 b r^l \cos l t + o (r^l)\bigr),\\
  E_{12}\overline{E_{22}} &=
           \frac{1}{4} e^{\imag m t}r^{\frac{mp}{1+p}}
           \left(\frac{-(2p+1)}{(1+p)^2} r^{\frac{-2mp}{1+p}}+
                   o(r^{\frac{-2mp}{1+p}}) \right),
 \end{aligned}
 \end{equation}
 where $z=re^{\imag t}$.
 Here, we used the relation 
 \[
     p = -\frac{1+\alpha}{2} = -\frac{1+\mu}{1+\mu-m}\in (-1,0),
 \]
 see Proposition~\ref{prop:alpha-mu-new}.
 Change the coordinate $r e^{\imag t}$ to $\eta e^{\imag t}$
 so that
 \[
     \eta=r^{\frac{m}{1+p}}.
 \]
 Then 
 \begin{equation}\label{eq:ccdd-tau}
 \begin{aligned}
  E_{21}\overline{E_{21}} 
     &= \eta^{-1}\bigl(1+(2b  \cos l t)\eta^\beta+o(\eta^\beta)\bigr),\\
  E_{22}\overline{E_{22}} &= \frac{1}{4}\eta^{-1}
              \left(
                \frac{1}{(1+p)^2}\eta^{-2p}+
                   o(\eta^{-2p})
              \right),\\
  E_{11}\overline{E_{21}}&=
           e^{\imag m t}\eta^p 
           \bigl(1 + (2 b  \cos l t)\eta^{\beta} + o (\eta^{\beta})\bigr),\\
  E_{12}\overline{E_{22}}&=
           \frac{1}{4}e^{\imag m t}\eta^p 
           \left(\frac{-(2p+1)}{(1+p)^2} \eta^{-2p}+
                   o(\eta^{-2p}) \right),
 \end{aligned}
 \end{equation}
 where $\beta=\beta_{p,l,m}$ as in \eqref{eq:non-cyl-data}.
\par\noindent
{\bf The case $(1+p)l<-2pm$:}
 In this case, $\beta<-2p$ holds.
 Then by \eqref{eq:ccdd-tau}, we have
 \begin{gather*}
    \frac{1}{h} = E_{21}\overline{E_{21}} + E_{22}\overline{E_{22}} =
                 \frac{1}{\eta}
                 \bigl(
                    1+ (2b \cos lt)\eta^{\beta}+o(\eta^{\beta})
                 \bigr),\\
    h    = \eta\bigl(1- (2b \cos l t)\eta^{\beta} +
                              o(\eta^{\beta})\bigr),\\
    \eta =
         h \left(1+(2 b \cos l t) h^{\beta}+o(h^{\beta})\right).
 \end{gather*}
 Thus, $\hat\pi\bigl(f(D^*)\cap \Horo_{\gamma}(o,s)\bigr)$
 is parametrized as
 \begin{align}\label{eq:non-cyl-case-1}
   \frac{\zeta}{h} & = E_{11}\overline{E_{21}} + E_{12}\overline{E_{22}}
     =  e^{\imag m t}\eta^p
        \bigl(1+(2b  \cos l t)\eta^{\beta}+o(\eta^{\beta})\bigr) \\
    &= e^{\imag m t}h^p\bigl(1+b N_{p,l,m}(h,t) +
        o(h^{\beta})\bigr)
       \qquad (h=e^{-s}).
       \nonumber
 \end{align}
\par\noindent
{\bf The case $(1+p)l>-2pm$:} 
 In this case, $\beta>-2p$ holds.
 Then by \eqref{eq:ccdd-tau}, we have
 \begin{gather*}
    \frac{1}{h} =  E_{21}\overline{E_{21}} + E_{22}\overline{E_{22}} =
                 \frac{1}{\eta}
                 \left(
                    1+ \frac{1}{4(1+p)^2}\eta^{-2p} + o(\eta^{-2p})
                 \right),\\
    h    = \eta\left(
                    1-
                    \frac{1}{4(1+p)^2}\eta^{-2p}+
                      o(\eta^{-2p})
                    \right),\\
    \eta =
                h\left(
                    1+ \frac{1}{4(1+p)^2}h^{-2p}+o(h^{-2p})
                    \right).
 \end{gather*}
 Thus, $\hat\pi\bigl(f(D^*)\cap \Horo_{\gamma}(o,s)\bigr)$
 is parametrized as
 \begin{align}\label{eq:non-cyl-case-2}
   \frac{\zeta}{h} & = E_{11}\overline{E_{21}} + E_{12}\overline{E_{22}}
     =  e^{\imag m t}\eta^p
        \left(1-\frac{2p+1}{4(1+p)^2}\eta^{-2p}+o(\eta^{-2p})\right) \\
    &= e^{\imag m t} h^p
         \left(1-\frac{1}{4(p+1)}h^{-2p}+o(h^{-2p})\right)\nonumber\\
    &= e^{\imag m t}h^p
         \left(1+S_p(h)+o(h^{-2p})\right).\nonumber
 \end{align}
\par\noindent
{\bf The case $(1+p)l=-2pm$:} 
 In this case, $\beta=-2p$ holds, and by a similar calculation,
 the intersection $\hat \pi(f(D^*)\cap \Horo_{\gamma}(o,s))$
 is parametrized as
 \begin{align}\label{eq:non-cyl-case-3}
   \frac{\zeta}{h} & = 
    e^{\imag m t} h^p
         \left(1+bN_{p,l,m}(h,t)+S_p(h)+o(h^{-2p})\right).
 \end{align}
\par\noindent
{\bf The case of centerless end:}
 In this case, $l$ in \eqref{eq:norm-omega-non-cyl} is always 
 equal to the maximum indentation number $n$ because 
 of \ref{item:indent:0}
and \ref{item:indent:1} in
 Lemma~\ref{lem:indentation}.
 Substituting $l=n$ into \eqref{eq:non-cyl-case-1}, 
 \eqref{eq:non-cyl-case-2} and \eqref{eq:non-cyl-case-3},
 we have \eqref{eq:non-cyl-asymptotic}.
\par\noindent
{\bf The case of centered end:}
 If the geodesic $\gamma$ coincides with the principal axis $\sigma$,
 $l$ in \eqref{eq:norm-omega-non-cyl} equals the 
 maximum indentation number $n$.
 Then \ref{item:non-cyl-2} holds.
We also explain Remark \ref{rem:non-principal} here.
If the geodesic $\gamma$ is not a principal axis, $l$ equals $m$. 
Substituting $l=m$ into \eqref{eq:non-cyl-case-1},
 \eqref{eq:non-cyl-case-2} and \eqref{eq:non-cyl-case-3},
we have \eqref{eq:non-cyl-asymp-3}. 
\end{proof}

\begin{proof}[Proof of Theorem~\ref{thm:cyl}]
 By Lemma~\ref{lem:normalized-omega},
 we normalize as in \eqref{eq:norm-omega-cyl}.
 Then by \eqref{eq:g-omega-repr}, we have
 \begin{align*}
  E_{11}&=  -\sqrt{\frac{\lambda}{m}}z^{\frac{m}{2}}
          \bigl(1+z^l+o(z^l)\bigr),\\
  E_{12}&=  \frac{1}{2\sqrt{\lambda m}}z^{\frac{m}{2}}
             \bigl(m+(2l-m)z^l+o(z^l)\bigr),\\
  E_{21}&=  -\sqrt{\frac{\lambda}{m}}z^{-\frac{m}{2}}
          \bigl(1+z^l+o(z^l)\bigr),\\
  E_{22}&=  \frac{1}{2\sqrt{\lambda m}}z^{-\frac{m}{2}}
             \bigl(-m+(2l+m)z^l+o(z^l)\bigr).
 \end{align*}
 Hence, we successively have
 \begin{align*}
   \frac{1}{h}&=
     E_{21}\overline{E_{21}} + E_{22}\overline{E_{22}} 
      = \frac{1}{m}\left(\lambda+\frac{m^2}{4\lambda}\right)
        r^{-m}
        \bigl(1+o(1)\bigr)\\
    h&=\left(\frac{4\lambda m}{4\lambda^2+m^2}\right)r^m \bigl(1+o(1)\bigr),\\
    \frac{\zeta}{h} &=
       \frac{e^{\imag m t}}{m}
       \left[
          \left(\lambda-\frac{m^2}{4\lambda}\right)+
           \left(
           2\left(\lambda+\frac{m^2}{4\lambda}\right)\cos l t -
             \imag\frac{ml}{\lambda}\sin l t
	   \right)r^l + o(r^l)
       \right],
 \end{align*}
 where $z=re^{\imag t}$.
 Setting $r^m=\eta$, we have
 \begin{align*}
   h &= \left(
          \frac{4\lambda m}{4\lambda^2+m^2}
        \right)\eta\bigl(1+o(1)\bigr),\\
   \frac{\zeta}{h} &=
       \frac{e^{\imag m t}}{m}\left[
          \left(\lambda-\frac{m^2}{4\lambda}\right)+\right.\\
        &\hspace{4em}
         \left.
          \left(
              \frac{4\lambda m}{4\lambda^2+m^2}
          \right)^{-\beta}
          \left(
           2\left(\lambda+\frac{m^2}{4\lambda}\right)\cos l t -
             \imag\frac{ml}{\lambda}\sin l t
	   \right)
           h^{\beta} + o(h^{\beta})
       \right],
 \end{align*}
 where $\beta=l/m$.

 If the end is centered and $\gamma$ is not principal,
 then $l=m$ and we have \eqref {eq:asymp-cyl-2}.
 In all other cases, $l$ equals the maximum indentation number $n$, 
 and \eqref{eq:asymp-cyl} holds.

 When the end is incomplete, $\lambda=m/2$ holds because of 
 Lemma~\ref{lem:normalized-omega}.
 Then the first term of \eqref{eq:asymp-cyl} 
 vanishes, 
 and we have \eqref{eq:asymp-cyl-incomp}.
\end{proof}

\begin{remark}[Behavior of the singular curvature]
 In \cite{SUY}, 
 the notion of singular curvature 
 for cuspidal edges of fronts was introduced.  
 It was seen there that the singular curvature of a cuspidal edge is
 negative, (resp. positive) if and only if the cuspidal edge curves 
 outward, (resp. curves inward) with respect to the 
 location of the surface. 
 (For representative figures, see \cite{SUY}.)
 Let $f \colon D^*\to H^3$ be a regular incomplete WCF-end.
 Since $f$ is flat in $H^3$, the extrinsic curvature $\Kext$ is 
 identically $1$. 
 Then the singular curvature $\kappa_s$ 
 of cuspidal edges on $f$ is negative, 
 by \cite[Theorem 3.1]{SUY}.
 Let $(\omega,\theta)$ be the canonical forms associated with $f$,
 and set $\rho=\theta/\omega$.
 Here, we take a complex coordinate $z=x+\imag y$ so that
 the image of the $x$-axis is a cuspidal edge and $z=0$ 
 corresponds to the incomplete end.
 Then  we can write
 \[
       \rho(z)=1+\imag \rho_1 z^n + o(z^n) \qquad (
             \rho_1\in\R\setminus\{0\}).
 \]
 A somewhat lengthy but straightforward calculation 
 gives the following explicit formula for the
 singular curvature on the $x$-axis:
 \begin{equation*}
  \kappa_s = -\frac{n|\rho_1|}{2m^2}x^n + o(x^n),
 \end{equation*}
 where $m$ is the multiplicity of the end.
 This implies that the singular curvature is 
 always negative, and hence that the cuspidal edges always curve 
 outward with respect to the surface.  Moreover,
 the limiting value of the singular curvature is zero at the end, 
 in accordance with the fact that the cuspidal edges become 
 asymptotically straight as they extend out to the end.  
\end{remark}

\section{The pitch of caustics}
\label{sec:example}
As pointed out in the introduction,
the caustic $C_f$ of a flat front gives locally a flat front.
In this section, we give a useful formula for the 
pitch of an end of $C_f$.
We suppose that the flat front 
$f\colon M^2\to H^3$ 
is weakly complete and of finite type.
Then there exist a compact Riemann surface 
$\overline{M}^2$ and finitely many points $p_1,\dots,p_n \in 
\overline{M}^2$ 
so that $M^2 = \overline{M}^2 \setminus \{ p_1,\dots,p_n \}$.  
Moreover, we assume the 
 restriction of $f$ to a neighborhood of any $p_j$
corresponds to a regular WCF-end.
So we call each $p_j$ a regular WCF-end of $f$.
On the other hand, let
\[
     q_1,\dots,q_m\in M^2
\]
be all of the umbilics of $f$.
Then it is known that the caustic 
\[
   C_f: \overline{M}^2 \setminus \{p_1,\dots,p_n,q_1,\dots,q_m\}
   \longrightarrow H^3
\]
is a weakly complete flat p-front of finite type, and
$\{p_1,\dots,p_n,q_1,\dots,q_m\}$ are all regular ends.
In particular, the number of ends of $C_f$ is $m+n$.
The ends $p_1,\dots,p_n$ of $C_f$ come from the ends
of $f$, so they are called {\it E-ends\/}.
The ends $q_1,\dots,q_m$ of $C_f$ come from the umbilics
of $f$, so they are called {\it U-ends\/}. (See \cite{KRUY}.)
Even when $C_f$ is not a front but only a p-front, 
by taking its double cover, we may consider it as a front.

From now on, we fix a point
\[
   e=p_j \quad \mbox{or} \quad 
   q_k \qquad (j=1,\dots, n,\,\, k=1,\dots,m)
\]
and denote by
\[
    C:D^*=D\setminus \{e\} \longrightarrow H^3
\]
a restriction of $C_f$ around a neighborhood $D(\subset \overline{M}^2)$ 
of the end $e$ of $C_f$, which is a regular WCF-end of a p-front.
If the unit normal vector field of $C$ is globally
defined on $D^*$ (namely $C$ is a front), 
the end $C$ is called {\it co-orientable}.
(In this case, $C$ itself is a regular WCF-end of a front.)
Otherwise, $C$ is called {\it non-co-orientable}.
If $C$ is non-co-orientable, it is not a front, but
taking the double cover $\pi:\hat D^*\to D^*$, then
$C\circ \pi$ is a regular WCF-end of a front.
\begin{proposition}[{\cite[Theorems 7.4 and 7.6]{KRUY}}]
 All of the U-ends $q_1,\dots,q_m$ 
 are incomplete regular WCF-ends of the $p$-front $C_f$.
 Each E-end $p_j$ is an incomplete regular WCF-end
 of $C_f$, unless $p_j$ is a snowman-type end of $f$.
\end{proposition}
If $p_j$ is a snowman-type end of $f$,
then $C_f$ has a complete cylindrical end at $p_j$,
and the pitch is equal to $0$.
So to give a formula for the pitch of the ends of $C_f$,
we may assume that $p_j$ is not a snowman-type end of $f$.
We can prove the following: 
\begin{theorem}\label{thm:Sec5}
 Let $C:D^*\to H^3$ be a regular WCF p-front end, which is
 the restriction of the caustic $C_f$
 around an E-end $z=p_j$ or a U-end $z=q_k$.
 The end is incomplete if and only if 
 $z = q_k$, or $z=p_j$ and $p_j$ is not a 
 snowman-type end of $f$.
 In this case,
 for a sufficiently small $\varepsilon>0$, 
 the image of $C$ in $\R^3_+$
 is congruent to a portion of the image of 
 $[0,2\pi) \times (0,h_0) \ni (t,h)\mapsto (\varphi_h(t),h)\in \R_+^3$,
 with
 \begin{equation}\label{eq:asymp-incomplete-2}
    \varphi_h(t)=
             h^{1+p}\Cyc_{m_c,n_c}(t)+o(h^{1+p}),
	     \quad p:=\frac{n_c}{m_c}\in (0,1)\cup (1,\infty) \;,
 \end{equation}
 for the map $\Cyc_{m_c,n_c}$ of a cycloid,
and for 
 \begin{align*}
       n_c&:=\frac1{2}\ord Q+\ord\bigl(S(G_*)-S(G)\bigl)+3, \\
       m_c&:=
       \begin{cases}
	  \dfrac12\ord Q+m_j +1 & (\mbox{if $z=p_j$}) \\[6pt]
	  \dfrac12\ord Q & (\mbox{if $z=q_k$}), 
       \end{cases}
 \end{align*}
where 
$m_j$ is the multiplicity of the end $p_j$ of
 $f$.
 In particular, the pitch $p=n_c/m_c$
 of $C$ is a rational number.
 {\rm (}In fact, if $p_j$ is a snowman-type end of $f$, it is
 a complete cylindrical end of $C_f$, and
 its pitch vanishes.{\rm )}
\end{theorem}

\begin{proof}
 The multiplicity $m_c$ of the end $p_j$ or $q_k$ of
 $C_f$ has been
 computed in \cite[Theorems 7.4 and 7.6]{KRUY}. 
So we consider only the  formula for $n_c$.
 The canonical forms $(\omega_c,\theta_c)$ of $C_f$
 are given by (see \cite[(6.5)]{KRUY})
 \[
      \omega_c=\imag\sqrt{Q}+\frac{d(\log \rho)}4,
          \quad
      \theta_c=-\imag\sqrt{Q}+\frac{d(\log \rho)}4,
 \]
 where $\rho=\theta/\omega$, $Q=\omega\theta$ and
 $(\omega,\theta)=(\hat \omega dz,\hat \theta dz)$ 
 is a pair of canonical forms
 of $f$. 
 (The meaning of the square root $\sqrt{Q}$ is
 explained in \cite{KRUY}.) 
 Then the Hopf differential $Q_c$ of $C_f$ is given by
 \[
      Q_c:=\omega_c\theta_c=Q+\left(\frac{d(\log \rho)}{4}\right)^2.
 \]
 We set $\rho_c:=\theta_c/\omega_c$.
 By a straightforward calculation, we have
 \begin{align*}
  d(\log \rho_c)&=\frac1{Q_c}(\hat\omega_c\hat\theta'_c
  -\hat\omega'_c\hat\theta_c)\, dz^3 \\
  &=\frac{\imag \sqrt{Q}}{2Q_c}
  \left\{ 
  \Bigl(\frac{\hat\theta'}{\hat\theta}\Bigr)'
      -\frac12\Bigl(\frac{\hat \theta'}{\hat \theta}\Bigr)^2
      -\Bigl(\frac{\hat\omega'}{\hat\omega}\Bigr)'
      +\frac12\Bigl(\frac{\hat\omega'}{\hat\omega}\Bigr)^2
  \right\} dz^2 \\
  &=\frac{\imag \sqrt{Q}}{2Q_c}
  \left\{
  \bigl(2Q+S(G_*)\bigr)-\bigl(2Q+S(G)\bigr)
  \right\}
  =\frac{\imag \sqrt{Q}}{2Q_c}\bigl(S(G_*)-S(G)\bigr) \;,
 \end{align*}
 where $'=d/dz$ and $S(\cdot)$ denotes the Schwarzian derivative as in
 \eqref{eq:schwarz}.
 Since $z=p_j$ or $z=q_k$ is incomplete,
 it must be cylindrical and the order of $Q_c$ at the end equals
 $-2$ (i.e., a pole of order $2$).
 The singular set of $C_f$ is represented as $\{|\rho_c|=1\}$. 
 Then
 \begin{align*}
    n_c&=\ord \bigl(d(\log \rho_c)\bigr)+1\\
    &=
    \frac1{2}\ord(Q)-\ord(Q_c)+\ord\bigl(S(G_*)-S(G)\bigr)+1\\
    &=\frac1{2}\ord(Q)+\ord\bigl(S(G_*)-S(G)\bigr)+3.
 \end{align*}
 If $\ord Q$ at the end is even, $C$ is a front, and
 the assertion  follows directly from Theorem~\ref{thm:incomplete}.
 On the other hand, if $\ord Q$ is odd, then $C$ is non-co-orientable.
 In this case, we get 
 the assertion by applying Theorem~\ref{thm:incomplete}
to the double cover of $C$.
\end{proof}

We shall now compute the pitches of ends of some complete flat 
fronts and their caustics, showing that a variety of cases do 
indeed occur on global examples.  

\begin{example}[Flat fronts of revolution]
\label{exa:revolution}
 Recall the flat fronts of revolution as in Example~\ref{ex:revolution}.
 The caustic of the horosphere is the empty set because 
 the horosphere is totally umbilic,
 and the caustic of a hyperbolic cylinder is a geodesic line.
 The caustic of an hourglass is a geodesic, which
 can be considered as a parallel surface of 
 the hyperbolic cylinder, regardable
 as a regular incomplete cylindrical end with pitch $p=\infty$.
 The caustic of the snowman is congruent to
 a hyperbolic cylinder. 
\end{example}

\begin{example}
\label{exa:n-noid}
 The third and the fourth authors \cite{UY2}
 constructed constant mean curvature one surfaces in
 $H^3$ from a given hyperbolic Gauss map and
 a polyhedron of constant Gaussian curvature $1$.
 Here, we explain the canonical symmetric flat $k$-noid
 via a construction similar to that one.
 We consider a domain $D_k$ in $\C$ bounded by
 a regular $k$-polygon $P_k$.
 Consider a pair of $\overline{D}_k$ and 
 glue them along $P_k$. 
 Then we get an 
 abstract flat surface with $k$ conical singularities,
 up to a homothety, which gives a flat
 symmetric conformal metric $|\omega|^2$ on 
 $S^2=\C\cup\{\infty\}$, that is, 
 \[
   \omega = c(z^k-1) ^{-2/k}\,dz \qquad (k\geq 3).
 \]
 We set $G=z$. By \cite[Theorem 4.1]{KRUY},
 the pair $(G,\omega)$ gives a flat symmetric $k$-noid
 \[
    f \colon \C\cup\{\infty\}\setminus\{z^k=1\}\longrightarrow
             H^3,
 \]
 whose Hopf differential $Q$ and other 
 hyperbolic Gauss map $G_*$ are given by
 \[
   Q=\frac{(k-1)z^{k-2}}{(z^k-1)^2}dz^2,\qquad G_*=z^{1-k}.
 \]
 In particular, $z=0,\infty$ are umbilics.
 Since $\lim_{z\to 1}(z-1)^2 Q /dz^2=(k-1)/k^2>0$ 
 and the ends of $f$ are congruent to each other, 
 the ends of $f$ are
 all of hourglass type, and they are embedded, as $G$ does
 not branch at the ends. Moreover, their pitch is given by
 \[
    p=-\frac{k-2}{2k-2}<0.
 \]
 Next, we consider the caustic $C_f$.
 The points $z=0,\infty$ are ends coming from the umbilics 
 of $f$.
 Since $Q$ has order $k-2$ at those two points, 
 we have $m_c=(k-2)/2$.  Since $G_*(z)=z^{1-k}$, 
 $S(G_*)$ has order $-2$ at $z=0,\infty$.
 Thus
 \[
    n_c=\frac12 \ord{Q}+\ord{S(G_*)}+3=\frac{k}2 \; , 
 \]
 and $n_c/m_c=k/(k-2)>1$ gives the pitch of
 the end $z=0,\infty$. They are hypocycloid-type ends.

 On the other hand, the other remaining ends of 
 $C_f$ are all congruent.
 By a similar computation, we have $m_c=1$ and 
 $n_c=2$, and thus the pitch is equal to $2$, 
 namely, these ends are of hypocycloid-type.  
 In particular, the case of $k=4$ 
 (Figure~\ref{fig:caustic} in the introduction)
 is very interesting.
 As pointed out in the introduction, the caustic $C_f$
 has octahedral symmetry. 
 In this case, one can easily
 get $n_c=2$ from the picture since the four cuspidal edges
 accumulate at each end.
\end{example}
\begin{figure}
\begin{center}
  \includegraphics[width=0.45\textwidth]{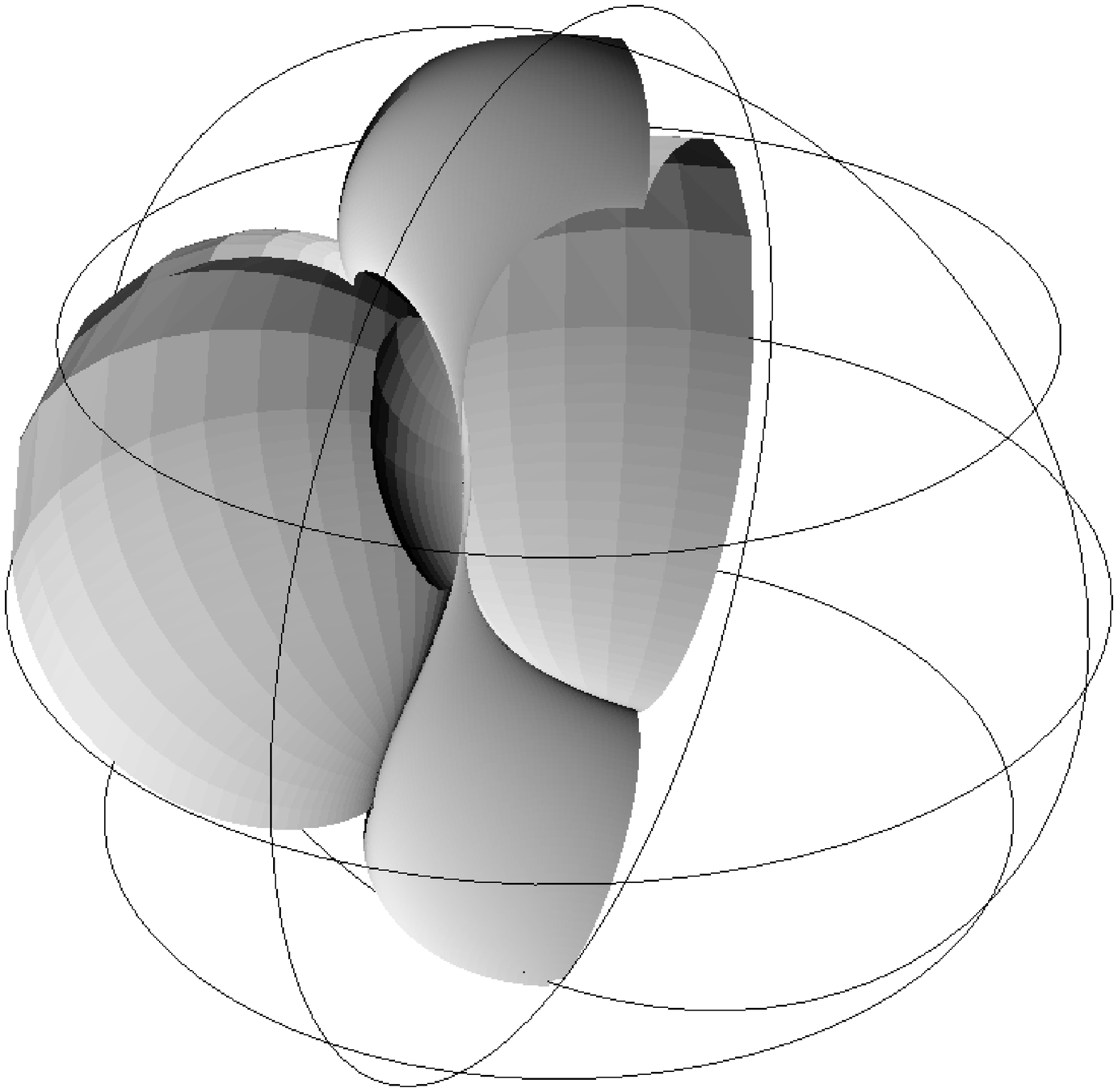} 
  \includegraphics[width=0.45\textwidth]{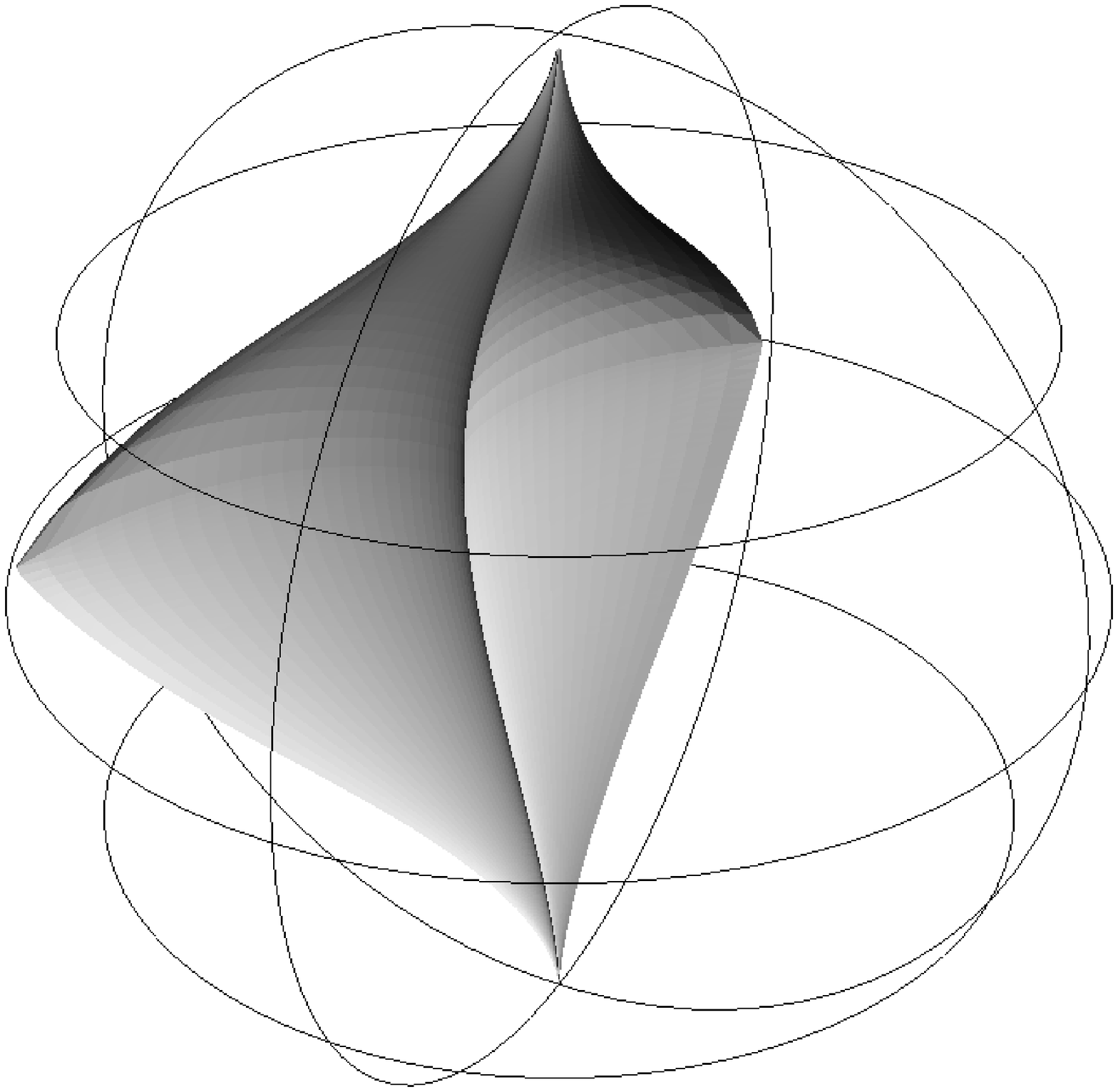} 
\end{center}
\caption{%
 One of three congruent portions of the front with data 
 $G=z^k$ and $G_*=z^{k+d}$ and one of three 
 congruent portions of its caustic, for
 $(k,d)=(1,3)$.
}\label{fig:nis1andmis4}
\end{figure}
\begin{example}
\label{exa:n+2-noid}
 We consider a cone $C_d (\subset \R^3)$ 
 over the domain $D_d$ in $\C$ bounded by
 a regular $d$-polygon $P_d$, which gives a polyhedron whose
 sides consist of $d$ regular triangles.
 Consider a pair of $\overline C_d$ and 
 glue them along $\overline D_d$. Then we get an
 abstract flat surface with $(d+2)$-conical singularities,
 which gives a flat
 symmetric conformal metric $|\omega|^2$ on 
 $S^2=\C\cup\{\infty\}$.
 We set $G=z^k$. Then the pair $(G,\omega)$ 
 gives a flat front with $d+2$ ends and dihedral symmetry 
 \[
    f \colon \C\setminus \bigl(\{0\}\cup\{ z^d =1\}\bigr)
       \longrightarrow H^3,
 \]
 whose Hopf differential $Q$ and other 
 hyperbolic Gauss map $G_*$ are given by
 \[
      Q=-\frac{k(k+d)z^{d-2}}{(z^d-1)^2}dz^2,\qquad G_*=z^{k+d}.
 \]
 In particular, $f$ has no umbilics.
 The ends $z=0,\infty$ are ends
 of multiplicity $k$ (embedded if and only if $k=1$) 
 and with pitch $p=-1/2$,
 that is, they are horospherical.
 On the other hand, the other $d$ ends are all
 mutually congruent, and they are embedded snowman-type
 ends with pitch $p=-(2k+d)/(2k+2d)<-1/2$.

 Next, we consider the caustic $C_f$.
 Since $f$ has no umbilics, the ends of
 $C_f$ are the same as those of $f$.
 The ends $z=0,\infty$ of $C_f$ satisfy
 $m_c=k+d/2$ and $n_c=d/2$. So the pitch is
 equal to $p_c=d/(2k+d)$, 
 and thus they are epicycloid-type.  
 The other ends of $C_f$ are mutually congruent.
 Since they are caustics of snowman-type ends,
 they are regular complete
 cylindrical ends (i.e., $p_c=0$).  
 See Figure \ref{fig:nis1andmis4}.  
\end{example}

\appendix
\section{The hyperbolic Gauss maps}
\label{app:boundary}
In this appendix,  we show that the hyperbolic Gauss maps 
defined as limits of the normal geodesic
of the flat front  coincide with $G$ and $G_*$ defined in 
Section~\ref{sec:prelim}.

\smallskip

Let $\Lor^4$ be the Minkowski $4$-space with the Lorentzian metric
$\langle \, , \, \rangle$ and consider the hyperbolic $3$-space $H^3$ 
as the hyperboloid as in \eqref{eq:hyp-lor}.
We denote the cone of future pointing light-like vectors by $L_+$:
\[
   L_+:=\{n=(n_0,n_1,n_2,n_3)\in \Lor^4\,;\,\inner{n}{n}=0,n_0>0\}.
\]
The multiplicative group $\R_+$ acts on $L_+$ by scalar multiplication.
The {\em ideal boundary\/} of $H^3$ is defined as
\begin{equation*}
 \partial H^3:= L_+/\R_+.
\end{equation*}
This is also considered as the {\em asymptotic classes of geodesics}.
In fact, if we denote by $\gamma_{x,\vect{v}}$ the geodesic
starting at $x$ with velocity $\vect{v}$ ($|\vect{v}|=1$), 
\begin{equation}\label{eq:asymptotic}
\begin{aligned}
   \text{the }&\text{asymptotic } \text{class of } \gamma_{x,\vect{v}}(s)
   :=(\cosh s)x+(\sinh s)\vect{v}\\
   &\leftrightarrow
   [x+\vect{v}]\in \partial H^3=L_+/\R_+\\
   &\leftrightarrow
   \frac{1}{(x_0+v_0)-(x_3+v_3)}\bigl((x_1+v_1)+\imag(x_2+v_2)\bigr)
   \in \C\cup\{\infty\}
\end{aligned}
\end{equation}
are  one-to-one correspondences,
where $x=(x_0,x_1,x_2,x_3)$, $\vect{v}=(v_0,v_1,v_2,v_3)$.

Now, identify  $\Lor^4$ with the set of $2\times 2$ hermitian matrices
$\Herm(2)$ as in \eqref{eq:lor-herm}:
\[
   \Lor^4\ni (x_0,x_1,x_2,x_3) \longleftrightarrow
   \begin{pmatrix}
     x_0 + x_ 3 & x_1 + \imag x_2 \\
     x_1 - \imag x_2  & x_0 - x_3
   \end{pmatrix}
   \qquad \left(\imag=\sqrt{-1}\right).
\]
Then the Lorentzian inner product $\inner{~}{~}$ is represented by  
\begin{equation*}
     \inner{X}{Y} = -\frac{1}{2}\trace X\widetilde Y,
\end{equation*}
where $\widetilde Y$ is the cofactor matrix of $Y$,
that is,  $\widetilde Y Y=Y \widetilde Y =(\det{Y}) \id$ holds.
In particular, $\inner{X}{X} = -\det X$.
Here, we can write
\begin{equation}\label{eq:herm-hyp}
  H^3 = \{x\in \Herm(2)\,;\,\det x=1,\trace x>0\}\\
=\{aa^*\,|\,a\in\SL(2,\C)\}.
\end{equation}
We write
\begin{equation}\label{eq:pauli}
   e_0=\id,\quad
   e_1=\begin{pmatrix}
	 0 & 1 \\ 1 & 0 
       \end{pmatrix},\quad
   e_2=\begin{pmatrix}
	 0 & \imag \\ -\imag & 0 
       \end{pmatrix},\quad
   e_3=\begin{pmatrix} 1 & \hphantom{-}0 \\ 0 & -1\end{pmatrix}.
\end{equation}
Let $x\in H^3$ and $X,Y\in T_xH^3$. 
The tangent space $T_xH^3$ is identified with the orthogonal 
compliment of the position vector $x$ in $\Lor^4=\Herm(2)$. 
Then we define a skew-symmetric bilinear form
\begin{equation}\label{eq:exterior}
T_xH^3\times T_xH^3\ni X,Y \mapsto
  X\times Y = \frac{\imag}{2}\bigl(Xx^{-1}Y-Yx^{-1}X\bigr)\in T_xH^3,
\end{equation}
called the {\em exterior product\/},
where $x\in H^3$ is considered as a matrix in $\SL(2,\C)$,
and products of the right-hand side are matrix multiplications.
Then one can show the following:
\begin{itemize}
  \item $X\times Y$ is perpendicular to both $X$ and $Y$.
  \item If $\iota$ is an orientation preserving isometry of $H^3$,
	$\iota_*X\times \iota_*Y=\iota_*(X\times Y)$.
	In particular, $e_1\times e_2=e_3$ holds  for matrices as in
	\eqref{eq:pauli},
	where $e_j$ $(j=1,2,3)$ are considered as vectors in 
	$T_{e_0}H^3$.
  \item $(x,X,Y,X\times Y)$ is a positively oriented basis of
	$\Lor^4$.
\end{itemize}

Let $f\colon{}D^*\to H^3$ be a complete regular end as in
Section~\ref{sec:prelim}, $\nu$ the unit
normal vector field and  $\E_f$ its holomorphic Legendrian lift.
Then we have: 
\begin{proposition}\label{eq:orientation-normal}
At each regular point of $f$,
 \[
     \nu = \sign\bigl(|\hat\theta|^2-|\hat\omega|^2\bigr) 
     \frac{f_u\times f_v}{|f_u\times f_v|}
 \]
  holds.
  That is, $\nu$ is compatible to the orientation of $D^*$
  if and only if $|\hat\theta|^2-|\hat\omega|^2>0$,
  where $z=u+\imag v$ is the complex coordinate and 
  $\omega=\hat\omega\,dz$, $\theta=\hat\theta\,dz$.
\end{proposition}
\begin{proof}
 By \eqref{eq:can-form}, we have
 \[
    f_z = \E_f\begin{pmatrix}
	      0 & \hat\theta \\
	     \hat\omega & 0 
	   \end{pmatrix}\E_f^*,\qquad
    f_{\bar z} = \E_f\begin{pmatrix}
	      0 & \overline{\hat\omega} \\
	     \overline{\hat\theta} & 0 
	   \end{pmatrix}\E_f^*.
 \]
 Thus, using \eqref{eq:exterior}, we have
 \begin{equation*}
   f_u\times f_v = -2 \imag f_z\times f_{\bar z}
       =
      \left(|\hat\theta|^2-|\hat\omega|^2\right)
               \E_f\begin{pmatrix}1 &\hphantom{-}0 \\ 0 & -1\end{pmatrix}\E_f^*.
   \end{equation*}
   Then we have the conclusion.
\end{proof}

Now, we define
\begin{equation*}
  G_{\pm} = [f\pm\nu]\colon{}D^* \longrightarrow \partial H^3,
\end{equation*}
where $[~~]$ denotes the equivalence class in $\partial H^3=L_+/\R_+$.
\begin{proposition}\label{prop:lightcone-g}
 Let $\E_f=(E_{ij})$ be the holomorphic Legendrian lift of $f$.  Then
 under the identification as in \eqref{eq:asymptotic}, it holds that
 \[
    G_+=G = \frac{E_{11}}{E_{21}},\qquad G_{-}=G_*=\frac{E_{12}}{E_{22}}.
 \]
\end{proposition}
\begin{proof}
 We denote $f+\nu=(n_0,n_1,n_2,n_3)$, and 
 consider $\Lor^4$ as $\Herm(2)$.
 Then  
 \begin{align*}
    \begin{pmatrix}
         n_0 + n_3 & n_1 + \imag n_2 \\
         n_1- \imag n_2 & n_0 -  n_3
    \end{pmatrix} &=
     f+\nu =
     \E_f\E_f^* + \E_fe_3\E_f^* =
     2 \E_f \begin{pmatrix}1 & 0 \\ 0 & 0 \end{pmatrix}\E_f^*\\
    &=
      2\begin{pmatrix}
	  E_{11}\overline{E_{11}} & E_{11}\overline{E_{21}} \\
	  E_{21}\overline{E_{11}} & E_{21}\overline{E_{21}}
       \end{pmatrix}.
 \end{align*}
 By \eqref{eq:asymptotic}, 
 the map
 $[f+\nu]$ is identified with
 \[
      \frac{n_1+ \imag n_2}{n_0-n_3}=
      \frac{E_{11}\overline{E_{21}}}{E_{21}\overline{E_{21}}}
         =\frac{E_{11}}{E_{21}}=G.
 \]

 On the other hand, since
 \[
  f-\nu  = 2  \E_f\begin{pmatrix} 0 & \hphantom{-}0 \\
		    0 &   -1\end{pmatrix} \E_f^*
               =-2\begin{pmatrix}
		   E_{12}\overline{E_{12}}  & E_{12}\overline{E_{22}} \\
		   E_{22}\overline{E_{12}} &  E_{22}\overline{E_{22}}
		  \end{pmatrix},
 \]
 $[f-\nu]$ is identified with
 \[
    \frac{E_{12}\overline{E_{22}}}{E_{22}\overline{E_{22}}}=
    \frac{E_{12}}{E_{22}}=G_*.
 \]
 This concludes the proof.  
\end{proof}

\section{A differential equation of cycloids}\label{B}
We will show that the plane curve 
$\varGamma \colon \vartheta \mapsto r(\vartheta)e^{\imag \vartheta}$ 
determined by a general solution $r$ of 
\begin{align}\label{eq:diffeq-1}
\begin{cases}
  \dfrac{d \log r}{d \vartheta} = u  \\[6pt]
 (p^2-1)\dfrac{du}{d \vartheta}= p^2+ (p^2 + 1)u^2+u^4 \qquad (p>0,~ p\ne 1)
\end{cases}
\end{align}
is a hypo-/epi-cycloid  if $p=n/m$.  

Integrating the second equation of \eqref{eq:diffeq-1}, 
 we have 
\begin{equation}\label{eq:integlated-once}
 \arctan u - \frac{1}{p} \arctan \frac{u}{p}
= \vartheta + C_1.
\end{equation}
Without loss of generality, we take   
the arbitrary constant $C_1$ in \eqref{eq:integlated-once} to be zero, 
because the constant $C_1$ can be cancelled by 
a change of the parameter of the curve $\varGamma$. 
From now on, we let $u=u(\vartheta)$ be the implicit function 
determined by \eqref{eq:integlated-once} with $C_1=0$. 

We introduce a new parameter
$s$($=s(\vartheta)$) by
\begin{equation}\label{eq:newparameter-t}
 s=\vartheta - \arctan u(\vartheta). 
\end{equation} 
Note that $s$ is monotone in $\vartheta$ because  
$ds/d \vartheta = (1+u^2)/(1-p^2) \ne 0$. 

It follows from \eqref{eq:integlated-once} and 
\eqref{eq:newparameter-t} that 
\begin{equation}\label{u-intermsof-t}
 -\frac{1}{p} \arctan \frac{u(\vartheta)}{p} = s, 
\end{equation}
that is, 
\begingroup
\addtocounter{equation}{-1}
\renewcommand{\theequation}{\Alph{section}.$\arabic{equation}'$}
\begin{equation}\label{u-intermsof-t-2}
 u(\vartheta)=-p  \tan (p s). 
\end{equation}
\endgroup
It follows from \eqref{eq:newparameter-t} and 
\eqref{u-intermsof-t-2} that 
\begin{equation}\label{eq:t-theta}
 \vartheta = s - \arctan\bigl(p \tan (ps)\bigr).
\end{equation}

Next, we rewrite the first equation of \eqref{eq:diffeq-1} 
in terms of $s$ instead of $\vartheta$. 
Since the first equation of \eqref{eq:diffeq-1} is 
equivalent to 
\begin{equation}\label{eq:dlogr-dt}
 \frac{d \log r}{ds} \frac{ds}{d \vartheta} = u, \qquad
  \text{i.e., }
 \frac{d \log r}{ds} = u\frac{d \vartheta}{ds},  
\end{equation} 
we first calculate $d \vartheta /ds$. In fact, differentiating 
\eqref{eq:t-theta}, we have
\begin{equation}\label{eq:dtheta-dt}
 \frac{d \vartheta}{ds} = 
\frac{(1-p^2) \cos^2(ps)}{\cos^2(ps)+p^2 \sin^2(ps)}.
\end{equation} 
Substituting \eqref{u-intermsof-t-2} and \eqref{eq:dtheta-dt} into 
\eqref{eq:dlogr-dt}, we have
\begin{equation}
 \frac{d \log r}{ds} = -
\frac{p(1-p^2) \sin (2ps)}{(1+p^2)+(1-p^2) \cos (2ps)}. 
\end{equation}
It implies that 
\begin{equation}\label{eq:r2}
 r^2 = C_2 \{ (1+p^2)+(1-p^2) \cos (2ps) \},
\end{equation}
where $C_2$ is an arbitrary non-zero constant. 

On the other hand, let us also describe $e^{2\imag \vartheta}$ 
in terms of $s$. In fact, it follows from 
\eqref{eq:t-theta} that 
\begin{align*}
  e^{\imag \vartheta} 
 &= e^{\imag s} \exp\bigl(-\imag \arctan(p \tan (ps)) \bigr)\\
 &= e^{\imag s} \{ \cos \left( \arctan(p \tan (ps)) \right) 
         - \imag \sin \left( \arctan (p \tan (ps)) \right) \} \\
 &= e^{\imag s} \left\{ 
       \frac{1}{\sqrt{1+(p \tan (ps))^2}}
      - \imag \frac{p \tan (ps)}{\sqrt{1+(p \tan (ps))^2}}
     \right\}.
\end{align*}
Hence, we have 
\begin{align}\label{eq:e2itheta}
  e^{2 \imag \vartheta} &= e^{2 \imag s} \, 
    \frac{1 -2 \imag p \tan (ps) - p^2 \tan^2 (ps)}
         {1+p^2 \tan^2 (ps)} \\
  &= e^{2 \imag s} \, \frac{(1-p^2)+(1+p^2)\cos 2ps -2\imag p \sin 2ps}%
  {(1+p^2)+(1-p^2) \cos 2ps}. 
 \nonumber
\end{align} 
It follows from \eqref{eq:r2} and \eqref{eq:e2itheta} that 
\begin{align*}
 (r e^{\imag \vartheta})^2 
   &= C_2 e^{2 \imag s} \{ (1-p^2)+(1+p^2)\cos 2ps -2\imag p \sin 2ps \} \\
   &= \frac{C_2}{2} e^{2 \imag s}
          \left\{ 
            (1-p) e^{\imag ps} + (1+p) e^{- \imag ps}
          \right\}^2 . 
\end{align*}
Hence, we can conclude that 
\begin{equation*}
 r e^{\imag \vartheta} 
   = C_3 
   \left\{ 
    (1-p) e^{\imag (1+p)s} + (1+p) e^{ \imag (1-p)s}
   \right\}
\end{equation*}
for arbitrary non-zero constant $C_3$. 

In the case of $p=n/m$, using a new parameter $t=s/m$, we have 
\begin{equation}\label{eq:retheta=cyc}
 \varGamma \colon r e^{\imag \vartheta}= 
  C \left\{ 
     (m-n) e^{\imag (m+n)t} + (m+n) e^{ \imag (m-n)t}
    \right\}
  \left( = C \cdot \Cyc_{m,n}(t) \right)
\end{equation}
for arbitrary non-zero constant $C$. 
The equation \eqref{eq:retheta=cyc} proves 
that the solutions of \eqref{eq:diffeq-1}
 give hypo-/epi-cycloids.


\end{document}